\documentclass[10pt,a4paper]{article}
\title{On the GL(2n) eigenvariety: branching laws, Shalika families and $p$-adic $L$-functions}
\author{Daniel Barrera Salazar, Mladen Dimitrov, Andrew Graham, \\Andrei Jorza and Chris Williams}
\date{}

\usepackage[margin=1.27in]{geometry}

\newcommand{\s}{\setlength{\itemsep}{0pt}}

\usepackage{multicol}
\usepackage{amssymb}

\usepackage{fancyhdr,xcolor}
\newcommand{\UPS}{\theta} %% Unramified character in unramified principal series
\newcommand{\fsigma}{f^\sigma}

\makeatletter 
\def\input@path{{../}} 
\makeatother
\usepackage{preamble}
\usepackage{comment}

\newcommand{\sar}[2]{\ar@{}[#1]|-*[@]{#2}}

\newcommand{\jw}{{(-j,\sw+j)}}

\usepackage{tikz}
\usetikzlibrary{positioning}
\usepackage{mathtools}
\usepackage{environ}
\NewEnviron{myquote}{\vspace{1ex}\par
	\hfill\llap{($\dagger$)}\hfill\parbox{\textwidth-2cm}%
	{\emph{\BODY}}%
	\vspace{1ex}\par}

\usepackage{tocloft}

\cftpagenumbersoff{part}

\newcommand{\tbyt}[4]{\left( \begin{array}{cc} #1 & #2 \\ #3 & #4 \end{array} \right)}

\newcommand{\tr}{\operatorname{tr}}

\newcommand{\vol}{\operatorname{vol}}
\titleformat{\part}
{\large\scshape\centering}% formatting commands to apply to the whole heading
{Part \thepart.}% the label and number
{0.5em}% space between label/number and subsection title
{}% formatting commands applied just to subsection title
[]% punctuation or other commands following subsection title
\usepackage{titlesec}
\titleformat{\subsubsection}[runin]% runin puts it in the same paragraph
{\normalfont\bfseries}% formatting commands to apply to the whole heading
{\thesubsubsection.}% the label and number
{0.5em}% space between label/number and subsection title
{}% formatting commands applied just to subsection title
[.\hspace*{6pt}]% punctuation or other commands following subsection title
\titlespacing*{\subsubsection}{0pt}{5pt}{0pt}

\titleformat{\subsection}
{\normalsize\bfseries\centering}% formatting commands to apply to the whole heading
{\thesubsection.}% the label and number
{0.5em}% space between label/number and subsection title
{}% formatting commands applied just to subsection title
[]% punctuation or other commands following subsection title

\titleformat{\section}
{\large\bfseries\centering}% formatting commands to apply to the whole heading
{\thesection.}% the label and number
{0.5em}% space between label/number and subsection title
{}% formatting commands applied just to subsection title
[]% punctuation or other commands following subsection title
\titlespacing*{\section}
{0pt}{8ex plus 1ex minus .2ex}{10pt}

\begin{document}

	\maketitle

	\renewcommand{\thefootnote}{\fnsymbol{footnote}} 
		%AMS : Comment here
	\footnotetext{\today. \emph{2020 MSC:} Primary 11F33,  11F67; Secondary 11R23
	}     
	\renewcommand{\thefootnote}{\arabic{footnote}}

	\begin{abstract}
		In this paper, we prove that a $\mathrm{GL}(2n)$-eigenvariety is \'etale over the (pure) weight space at non-critical Shalika points, and construct multi-variable $p$-adic $L$-functions varying over the resulting Shalika components. Our constructions hold in tame level 1 and Iwahori level at $p$, and give $p$-adic variation of $L$-values (of regular algebraic cuspidal automorphic representations of $\mathrm{GL}(2n)$ admitting Shalika models) over the whole pure weight space. In the case of $\mathrm{GL}(4)$, these results have been used by Loeffler and Zerbes to prove cases of the Bloch--Kato conjecture for $\mathrm{GSp}(4)$.
		
		Our main innovations are: (a) the introduction and systematic study of `Shalika refinements' of local representations of $\mathrm{GL}(2n)$, and evaluation of their attached local twisted zeta integrals; and (b) the $p$-adic interpolation of representation-theoretic branching laws for $\mathrm{GL}(n) \times \mathrm{GL}(n)$ inside $\mathrm{GL}(2n)$. Using (b), we give a construction of multi-variable $p$-adic functionals on the overconvergent cohomology groups for $\mathrm{GL}(2n)$, interpolating the zeta integrals of (a). We exploit the resulting non-vanishing of these functionals to prove our main arithmetic applications.
	\end{abstract}

	\setcounter{tocdepth}{1}
	\footnotesize
	\tableofcontents
	\normalsize

	\section{Introduction}
	
	\subsection{Motivation} The \emph{Bloch--Kato conjectures} are amongst the most important open problems in modern algebraic number theory, and predict a deep link between arithmetic and analysis. Through decades of research, a fruitful approach to Bloch--Kato has been to find and prove $p$-adic reinterpretations; for every special case of the Bloch--Kato conjecture, there should be an analogous $p$-adic \emph{Iwasawa Main Conjecture} relating $p$-adic arithmetic data to a $p$-adic $L$-function. These $p$-adic reinterpretations are usually more tractable than the original conjectures -- for example, the Iwasawa Main Conjecture for elliptic curves has been proved in many cases (for example in \cite{Kat04,SU14}, but also in many other works). Moreover, understanding the $p$-adic picture can lead to proofs of special cases of Bloch--Kato.
	
	Crucial to proofs of Bloch--Kato/Iwasawa Main Conjectures is a good understanding of $p$-adic $L$-functions, eigenvarieties, and $p$-adic $L$-functions over eigenvarieties. In this paper, we prove new results about these objects for $\GL_{2n}/\Q$. In particular, let $\pi$ be a regular algebraic symplectic cuspidal automorphic representation (RASCAR)\label{RASCAR} of $\GL_{2n}(\A)$ that is everywhere spherical; here \emph{symplectic} is the condition that $\pi$ admits a Shalika model (i.e.\ is a functorial transfer from $\mathrm{GSpin}_{2n+1}(\A)$). This ensures  there is an integer $\sw$ such that $\pi \cong \pi^\vee \otimes |\cdot|^{\sw}$, i.e.\ $\pi$ is essentially self-dual. 
	Let $\tilde\pi$ be an Iwahoric $p$-refinement which is spin and of   non-critical slope (see Definitions~\ref{def:spin-refinement} and \ref{def:non-critical slope}). As explained in \S\ref{sec:parahoric-vs-iwahoric}, by \cite[Thm.\ A]{BDW20} there exists 
	a $p$-adic $L$-function $L_p(\tilde\pi)$ attached to $\tilde\pi$, that is, a locally analytic distribution on $\Zp^\times$ of controlled growth that interpolates its Deligne-critical $L$-values. In this paper:
	\begin{itemize}\setlength{\itemsep}{0pt}
		\item[(A)] we prove that the Iwahoric $\GL_{2n}$-eigenvariety is \'etale over the ($n+1$-dimensional) pure weight space at $\tilde\pi$, and that the unique connected component $\sC$ through $\tilde\pi$ contains a very Zariski-dense set $\sC^{\mathrm{Sha}}$ of classical Shalika points; and
		\item[(B)] we construct an $(n+2)$-variable $p$-adic $L$-function $\mathcal{L}_p^{\sC}$ interpolating $L_p(\tilde\pi_y)$ for $y \in \sC^{\mathrm{Sha}}$.
	\end{itemize} 
	
We state these results in full in Theorems A and B respectively in \S\ref{sec:main constructions} of the introduction.
	
	We describe an application. In the special case of $\GL_4$, part (B) fulfils the forward compatibility required by Loeffler and Zerbes in their recent tour-de-force work \cite{LZ20} proving new cases of the Bloch--Kato conjecture for $\mathrm{GSp}_4$; in particular, the present paper is the `forthcoming work' mentioned in \S17.5 \emph{op.\ cit}., where this special case was first announced.

	\subsection{Set-up and previous work}
	Let $\pi$ be as above, and let $\lambda = (\lambda_1,...,\lambda_{2n})$ be its weight, a dominant algebraic character of the diagonal torus $T \subset G \defeq \GL_{2n}$. Then $\sw = \lambda_{n} + \lambda_{n+1}$ is the purity weight of $\lambda$ (see \S\ref{sec:algebraic weights}). Let $L(\pi,s)$ be the standard $L$-function of $\pi$, normalised so that for $j \in \Z$, the value $L(\pi,j+\tfrac{1}{2})$ is Deligne-critical if and only if 
\begin{equation}\label{eq:deligne critical}
	j \in \mathrm{Crit}(\lambda) \defeq \{j \in \Z: -\lambda_{n+1} \geq j \geq -\lambda_n\}.
\end{equation}
 Let $K = \Iw \prod_{\ell \neq p}G(\Z_\ell) \subset G(\widehat{\Z})$, where $\Iw$ is the Iwahori subgroup at $p$. Let $S_K$ be the locally symmetric space for $G$ of level $K$. As $\pi$ is regular algebraic, it contributes to the compactly supported cohomology of $S_K$ with coefficients in $\sV_\lambda^\vee$ in degrees $n^2,n^2+1,...,n^2+n-1$. Here $V_\lambda$ is the algebraic representation of $G$ of highest weight $\lambda$, and $\sV_\lambda^\vee$ is the local system on $S_K$ attached to its dual. Let $t = n^2+n-1$ (the top degree).
	
	Our work builds on ideas of Grobner--Rahuram \cite{GR2}, of Dimitrov--Januszewski--Raghuram \cite{DJR18}, and particularly of Barrera--Dimitrov--Williams \cite{BDW20}, all of which worked in the $Q$-parahoric setting, for $Q$ the $(n,n)$-parabolic subgroup of $G$. As \emph{op.\ cit}., our methods are built upon the existence of \emph{evaluation maps}, functionals on Betti cohomology groups. In particular:
	\begin{itemize}\setlength{\itemsep}{0pt}
		\item[--] In \cite{GR2}, the authors constructed $\C$-valued evaluation maps 
\[
	\mathrm{Ev}_j^{\lambda, \mathrm{GR}} : \hc{t}(S_K,\sV_\lambda^\vee(\C)) \to \C
\]
and used them to prove algebraicity for the Deligne-critical $L$-values $L(\pi,j+\tfrac{1}{2})$.
		
		\item[--] In \cite{DJR18}, for $\chi$ finite order of conductor $p^\beta$, the authors used $p$-adic analogues
\[
	\mathrm{Ev}_{\chi,j}^{\lambda, \mathrm{DJR}} : \hc{t}(S_K,\sV_\lambda^\vee(\overline{\Q}_p)) \to \overline{\Q}_p
\]
 to construct $p$-adic $L$-functions for   ordinary `Shalika' $Q$-parahoric $p$-refinements $\tilde\pi^Q$ of $\pi$.
		\item[--] In \cite{BDW20}, the authors constructed (parahoric) overconvergent evaluations 
\[
	\hc{t}(S_K,\sD_\Omega^Q) \to \cD(\Zp^\times,\cO_\Omega),
\]
 where $\Omega$ is a (2-dimensional) parahoric $p$-adic family of weights, $\sD_\Omega^Q$ is a local system attached to a space of parahoric distributions, and $\cD(\Zp^\times,R)$ is the space of $R$-valued locally analytic distributions on $\Zp^\times$; so $\cD(\zp^\times,\cO_\Omega)$ is a space of 3-variable distributions. These interpolated the $\mathrm{Ev}_{\chi,j}^{\lambda,\mathrm{DJR}}$ for varying $\lambda, \chi$ and $j$. They were used to construct $p$-adic $L$-functions $L_p(\tilde\pi^Q)$ attached to finite slope $Q$-parahoric  Shalika $p$-refinements $\tilde\pi^Q$, to construct 2-dimensional (parahoric) $p$-adic families through $\tilde\pi^Q$, and to vary $L_p(\tilde\pi^Q)$ over these families.
	\end{itemize}
	These papers work with $\GL_{2n}$ over totally real fields and do not require $\pi$ everywhere spherical; for a detailed summary of these works, we refer the reader to \cite[Intro.]{BDW20}. However, the $Q$-parahoric setting considered \emph{op.\ cit}.\ cannot see variation in more than 2 weight variables, meaning our present results are necessary for the application to Bloch--Kato in \cite{LZ20}.

	\subsection{New input}
	In our generalisation to Iwahoric families, substantial new ideas are required in two particular places: one automorphic (the computation of local zeta integrals at $p$), and one $p$-adic (the $p$-adic interpolation of classical representation-theoretic branching laws).

	\subsubsection{Shalika $p$-refinements and local zeta integrals}\label{sec:intro shalika}
	
	To study local zeta integrals attached to the (unramified) representation $\pi_p$, we introduce \emph{Shalika $p$-refinements} of $\pi_p$. In this paper, we consider Iwahoric $p$-refinements, rather than the $Q$-parahoric $p$-refinements of \cite{DJR18,BDW20}. 
	
	Let $\cH_p$ be the Hecke algebra at $p$ (see Definition~\ref{def:hecke algebra p}). An (Iwahoric) \emph{$p$-refinement} of $\pi_p$ is a 
	pair $\tilde\pi_p = (\pi_p,\alpha)$, where  $\alpha : \cH_p \longrightarrow \overline{\Q}$ is a 
	system of Hecke eigenvalues  appearing in $\pi_p^{\Iw}$.  If $\pi_p$ has regular Satake parameter, there are $(2n)!$ such $p$-refinements, indexed by elements of the Weyl group $\cW_G = \mathrm{S}_{2n}$,   all regular in the sense of Definition~\ref{def:refinement}. Attached to any regular $p$-refinement $\tilde\pi_p$ is a certain family of twisted local zeta integrals at $p$. We call $\tilde\pi_p$ a \emph{Shalika $p$-refinement} if one of these local zeta integrals is non-vanishing. 

In \cite{DJR18,BDW20}, it is implicitly predicted that a $p$-refinement should be Shalika if and only if it lies in a certain class of `spin' $p$-refinements, i.e.\ those that interact well with the Shalika model. In the parahoric case, an ad-hoc definition of a spin refinement -- there called a `$Q$-regular $Q$-refinement' -- is given in \cite[\S3.3]{DJR18}, inspired by \cite{AG94}; and the relevant zeta integrals are shown to be non-vanishing.
	
	In \S\ref{sec:spin refinements}, we give a much more conceptual definition of spin $p$-refinements, generalising and justifying \cite[\S3.3]{DJR18}. Since $\pi_p$ admits a Shalika model, it is the functorial transfer of a representation $\Pi_p$ of $\cG(\Q_p)$, where $\cG = \mathrm{GSpin}_{2n+1}$. Via a careful study of the root systems of $G$ and $\cG$, we construct a map $\jmath^\vee : \cH_p \to \cH_p^{\cG}$ of Hecke algebras at $p$. We then say $\alpha$ is a \emph{spin $p$-refinement} if there exists an eigensystem $\alpha^{\cG}$ for $\cG$ such that $\alpha$ factors as
\[
	\alpha : \cH_p \xrightarrow{\ \jmath^\vee\ } \cH_p^{\cG} \xrightarrow{\ \alpha^{\cG}\ } \overline{\Q},
\]
and show the system $\alpha^{\cG}$ then appears in $\Pi_p^{\Iwahori_{\cG}}$, where  $\Iwahori_{\cG} \subset \cG(\Zp)$ denotes the Iwahori subgroup for $\cG(\Zp)$.  When $\pi_p$ (hence $\Pi_p$) is regular, there are $2^n n!$ such spin $p$-refinements. 
	
	We prove that the family of local zeta integrals attached to a spin $p$-refinement is non-vanishing, and hence that spin $p$-refinements are Shalika $p$-refinements. We study the converse (that Shalika refinements are spin refinements) in a sequel paper \cite{classical-locus}; there we prove it in many cases, and conjecture that it always holds.

We actually compute the local zeta integral in two different ways, with different benefits. 
\begin{itemize}\s
\item In \S\ref {sec:local zeta p}, we compute the local zeta integral at Iwahoric level, with the restriction that our method works only for ramified characters. 
\item The unramified case at Iwahoric level appears to be very difficult. We get around this in \S\ref{sec:local zeta unramified} by instead computing the unramified integral \emph{at parahoric} level, via a totally different method.  
\end{itemize}
The latter result strengthens the results of \cite{DJR18,BDW20} by proving that the $p$-adic $L$-functions they construct satisfy the expected interpolation property at the trivial character, a fundamental case these works omitted.

In later sections, we use interpolation at ramified characters to prove that the $p$-adic $L$-functions at Iwahori level and parahoric level agree; and thus we obtain interpolation at the trivial character for the $p$-adic $L$-functions of the present paper.

 All of these local results are proved in \S\ref{sec:local zeta p}--\S\ref{sec:local zeta unramified}.

	\subsubsection{$p$-adic interpolation of branching laws}
	
	Over $\Q$, for any $n$ the evaluation maps of \cite{BDW20} are valued in a space of distributions in only 3 variables. In the present paper we construct evaluation maps in the full expected $n+2$ variables. 

Our key input is a $p$-adic interpolation, in $(n+2)$-variables, of classical representation-theoretic branching laws. More precisely: in \cite{GR2,DJR18,BDW20} the subgroup $H = \GL_n \times \GL_n$ (embedded diagonally in $\GL_{2n}$) plays a distinguished role. 
For $j_1,j_2 \in \Z$, let $V_{(j_1,j_2)}^H$ denote the $H$-representation $\det_1^{j_1}\det_2^{j_2}$, the algebraic representation of highest weight $(j_1,...,j_1,j_2,...,j_2)$. Recall $\sw$ is the purity weight of $\lambda$. Then we have the following reinterpretation of the Deligne-critical $L$-values \eqref{eq:deligne critical}: \\[-4pt]
	
	\begin{myquote}\label{eq:branching}
		\text{\textbf{\emph{Branching law:}} $j \in \mathrm{Crit}(\lambda) \iff V_{(-j,\sw+j)}^{H} \subset V_\lambda\big|_{H}$ with multiplicity one.}
	\end{myquote}

	\begin{example}
		Let $G = \GL_2, H = \GL_1\times\GL_1$, and let $\pi$ be a RACAR of weight $(k,0)$, corresponding to a classical newform $f$ of weight $k+2$. Then (in our normalisations) $\pi$ has Deligne-critical $L$-values $L(\pi,j+\tfrac{1}{2})$ with $-k \leq j \leq 0$. Here $\sw = k$ and $V_\lambda = \mathrm{Sym}^k(\C^2),$ the space of homogeneous polynomials in two variables $X,Y$ of degree $k$. We have $V_\lambda|_{H} = \oplus_{j=-k}^0 [\C \cdot X^{-j}Y^{k+j}]$. The summand at $j$ is the character $\smallmatrd{s}{}{}{t} \mapsto s^{-j}t^{k+j}$ of $H$, and corresponds to the Deligne-critical $L$-value $L(\pi,j+\tfrac{1}{2})$. 
	\end{example}

Attached to $H$ is a family of \emph{automorphic cycles} $\{X_\beta\}_{\beta \geq 1}$, which are modified locally symmetric spaces for $H$ whose  dimension is crucially $t$. For $j \in \mathrm{Crit}(\lambda)$, the evaluation maps $\mathrm{Ev}_{\chi,j}^{\lambda,\mathrm{DJR}}$ of \cite{DJR18} were constructed as a composition
\[
\hc{t}\big(S_K,\sV_\lambda\big) \to \hc{t}\Big(X_\beta,\sV_\lambda|_H\Big) \to \hc{t}\Big(X_\beta, \sV^H_{(-j,\sw+j)}\Big) \to \bigoplus_{(\Z/p^\beta)^{\times}}\overline{\Q}_p \to \overline{\Q}_p, \text{where:} 
\]
\begin{itemize}\setlength{\itemsep}{0pt}
\item the first map is a twisted pullback under the natural inclusion $X_\beta \to S_K$, 
\item the second is projection of the coefficients via the branching law $(\dagger)$, 
\item the third is integration of scalar-valued classes (of degree $t$) over the connected components of $X_\beta$ (which have dimension $t$), 
\item and the last map is `evaluation at $\chi$', where $\chi$ has conductor $p^\beta$.
\end{itemize}	
	
	To interpolate these maps, in \cite{BDW20} the authors gave a $p$-adic interpolation of ($\dagger$). There were two aspects to this: the interpolation for a fixed $\lambda$ as $j$ varies, proved in \S5.2 \emph{op.\ cit}., used to prove existence of $L_p(\tilde\pi)$; and the interpolation as $\lambda$ varies in a (2-dimensional) $Q$-parabolic weight family $\sW^Q$ in \S6.2, used to construct parabolic families of RASCARs and 3-variable $p$-adic $L$-functions for these families.

	The full pure weight space $\sW_0^G$ for $G$ has dimension $n+1$, whilst the families of \cite{BDW20} are only 2-dimensional. The trade-off made \emph{op.\ cit}.\ was that variation in lower dimension allowed weaker assumptions on $\tilde\pi$, giving an `optimal' notion of non-criticality for $\tilde\pi$. However, if one assumes stronger conditions on $\tilde\pi$, then one expects to be able to vary the $p$-adic $L$-function over the full $(n+1)$-dimensional pure weight space; and it is this, higher-dimensional, variation that is required for the application to the Bloch--Kato conjecture of \cite{LZ20}.
	
	To extend the results of \cite{BDW20} to get full variation, one needs to interpolate the branching law $(\dagger)$ as $\lambda$ varies over $\sW_0^G$. The approach of \cite{BDW20} has the parahoric, hence 2-dimensional, setting baked into it, so to interpolate in higher dimension requires new ideas.

In Proposition \ref{prop:distribution commute}, we give a full interpolation of ($\dagger$) over $\sW_0^G$ (in the language of $p$-adic distributions). This result occupies the entirety of \S\ref{sec:branching laws}. 

Our approach exploits properties of \emph{spherical varieties}: in particular, we use crucially the existence of an `open orbit element' $u \in G(\Zp)$ such that $\overline{B}(\Zp) u^{-1}H(\Zp)$ is Zariski-open in $G(\Zp)$ (for $\overline{B} \subset G$ the opposite Borel of lower triangular matrices). The approach should apply to much more general spherical pairs $(G,H)$; this has now been studied, for example, in \cite{rockwood-spherical}.
	
	In \S\ref{sec:distribution valued evaluations}, via \S\ref{sec:abstract evaluation maps}, we use our $p$-adic branching laws to construct $(n+2)$-variabled evaluation maps
	\begin{equation}\label{eq:intro evaluation}
		\mathrm{Ev}_{\beta}^{\Omega} : \hc{t}(S_K,\sD_\Omega) \longrightarrow \cD(\Zp^\times,\cO_\Omega),
	\end{equation}
	for an $(n+1)$-dimensional affinoid $\Omega \subset \sW_0^G$. These maps interpolate (Iwahoric analogues of) $\mathrm{Ev}_{\chi,j}^{\lambda,\mathrm{DJR}}$ as $\lambda$ varies in $\Omega$, $j$ varies over $\mathrm{Crit}(\lambda)$, and $\chi$ varies over finite order Hecke characters of conductor $p^\beta$.

	\subsection{Main constructions}\label{sec:main constructions}
	
	We give two main applications of these results. Using the Shalika refinements of Definition~\ref{def:shalika refinement}, we give simple automorphic criteria for the non-vanishing of the evaluation maps \eqref{eq:intro evaluation}. This non-vanishing puts tight restrictions on the structure of $\hc{t}(S_K,\sD_\Omega)$, and thus -- via \cite{Han17} -- on the structure of the $G$-eigenvariety $\sE^{G}$. Exploiting ideas developed in \cite{BDW20}, in Theorem \ref{thm:shalika family} we prove the following. Recall that a RASCAR is a regular algebraic cuspidal automorphic representation that is symplectic (i.e.\ it admits a Shalika model).

	\begin{theorem-intro}\label{thm:intro A}
		 Let $\pi$ be a RASCAR of $G(\A)$ of weight $\lambda = (\lambda_1,\dots,\lambda_{2n})$
that is everywhere spherical.  Suppose that $\pi_p$ is regular and that 
			\[
				\lambda_n > \lambda_{n+1}. 
			\]
		Let $\tilde\pi = (\pi, \alpha)$ be spin $p$-refinement having  non-critical slope (see Definitions~\ref{def:spin-refinement} and \ref{def:non-critical slope}). 
			
			Then the $G$-eigenvariety $\sE^{G}$ is \'etale over $\sW_0^G$ at $\tilde\pi$. There exists a neighbourhood of $\tilde\pi$ in $\sE^{G}$ containing a very Zariski-dense set of classical points corresponding to RASCARs.
	\end{theorem-intro}
	
	In other words, there exists an $(n+1)$-dimensional  affinoid $\Omega \subset \sW_0^G$ containing $\lambda$ such that:
	\begin{itemize}\setlength{\itemsep}{0pt}
		\item[(a)] $\tilde\pi$ varies in a unique $p$-adic family $\sC \subset \sE^{G}$ over $\Omega$,
		\item[(b)] $\sC$ contains a very Zariski-dense set $\sC^{\mathrm{Sha}}$ of classical points corresponding to RASCARs,
		\item[(c)] and the weight map $w: \sC \isorightarrow \Omega$ is an isomorphism.
	\end{itemize}
	To prove this result, we observe that the weight condition $\lambda_n > \lambda_{n+1}$ implies existence of a non-vanishing Deligne-critical $L$-value. Since $\mathrm{Ev}_\beta^\Omega$ interpolates these $L$-values, it is therefore non-vanishing. We use non-vanishing of $\mathrm{Ev}_\beta^\Omega$ twice: once to produce existence of an $(n+1)$-dimensional family, and again to prove existence of a very Zariski-dense set of classical points attached to RASCARs.

	\medskip
	
	Our second main result, under the same hypotheses, is the construction of an $(n+2)$-variable $p$-adic $L$-function over $\sC$. We show that $\mathrm{Ev}_{\beta+1}^\Omega = \mathrm{Ev}_{\beta}^\Omega \circ U_p^\circ$, where $U_p^\circ$ is the full (normalised) Iwahori Hecke operator at $p$. We thus use \eqref{eq:intro evaluation} to attach a well-defined distribution $\mu^\Omega(\Phi) \defeq (\alpha_p^\circ)^{-\beta}\mathrm{Ev}^{\Omega}_\beta(\Phi)$ to any finite-slope eigenclass $\Phi \in \hc{t}(S_K,\sD_\Omega)$ with $U_p^\circ \Phi = \alpha_p^\circ \Phi$. Note this is independent of $\beta$. We show existence of a distinguished eigenclass $\Phi_{\sC} \in \hc{t}(S_K,\sD_\Omega)$ attached to the family $\sC$, and then define $\cL_p^{\sC} \defeq \mu^\Omega(\Phi_{\sC}) \in \cD(\Zp^\times,\cO_\Omega)$. Under the Amice transform, we view $\cL_p^{\sC}$ as a rigid analytic function $\cL_p^{\sC}$ on $\sC\times \sX(\Zp^\times)$, 	where $\sX(\Zp^\times)$ is the $\Q_p$-rigid space of characters on $\Zp^\times$. In Theorem \ref{thm:family p-adic L-functions} we show:

	\begin{theorem-intro}\label{thm:intro family}
		Let $\tilde\pi$ satisfy the hypotheses of Theorem \ref{thm:intro A}, and let $\sC \subset \sE^{G}$ be the (unique) corresponding $p$-adic family through $\tilde\pi$. There exists a rigid analytic function
			\[
				\cL_p^{\sC} : \sC \times \sX(\Zp^\times) \longrightarrow \overline{\Q}_p
			\]	
		satisfying the following interpolation property: for all $y \in \sC^{\mathrm{Sha}} \subset \sC$,  there exist $c_y^\pm \in \overline{\Q}_p^\times$ such that for all characters $\xi \in \sX(\Zp^\times)$ with $\xi(-1) = \pm 1$, we have
		\begin{align}\label{eq:int in C}
			\cL_p^{\sC}(y,\xi) &= c_y^{\pm} \cdot L_p(\tilde\pi_y, \xi).
		\end{align}
	Here $L_p(\tilde\pi_y,-)$ is the (1-variable) $p$-adic $L$-function from \cite{BDW20}. 
	\end{theorem-intro}
	
	 One could restate this interpolation property in terms of critical $L$-values as follows: for all $y \in \sC^{\mathrm{Sha}}$ of weight $\lambda_y = w(y)$, all integers $j$ with $-\lambda_{y,n+1} \geq j \geq -\lambda_{y,n}$, and all finite order characters $\chi$ of $\Q^\times\backslash \A^\times$ of $p$-power conductor (including trivial conductor), we have
		\[
			\cL_p^{\sC}(y, \chi\chi_{\cyc}^j)	= c_y^{\pm} \cdot \gamma_{(pm)} \cdot e_p(\tilde\pi_y,\chi,j) \cdot e_\infty(\pi_y,\chi,j) \cdot \frac{L^{(p)}\big(\pi_y\otimes\chi, j+\tfrac{1}{2}\big)}{\Omega_{\pi_y}^{\pm}}.
		\]
		Here $\chi_{\cyc}$ is the cyclotomic character, the sign is determined by $\chi(-1)(-1)^j = \pm 1$,  $\gamma_{(pm)}$ is a volume constant, $e_p$ and $e_\infty$ are the Coates--Perrin-Riou factors at $p$ and $\infty$, $L^{(p)}$ is the $L$-function with the local factor at $p$ removed, and $\Omega_{\pi_y}^\pm$ are periods. All of this is explained in Theorem \ref{thm:non-ordinary}.

	\subsection{Application: Bloch--Kato for $\mathrm{GSp}(4)$}
	
	In \cite{LZ20}, Loeffler and Zerbes prove new cases of the Bloch--Kato conjecture for Galois representations attached to Siegel modular forms of genus 2 (i.e.\ for automorphic representations of $\mathrm{GSp}_4$). More precisely, if $\cF_{\mathrm{new}}$ is a Siegel modular form of level 1 and sufficiently high weight, and $\cF$ is an ordinary $p$-stabilisation, they prove the Bloch--Kato conjecture holds for the  $4$-dimensional spin Galois representation attached to $\cF$ in analytic rank 0. This has also led to new understanding of the Bloch--Kato conjecture for symmetric cube modular forms \cite{LZ-cube} and of Iwasawa theory for quadratic Hilbert modular forms \cite{LZ-HilbertQuadratic}.
	
	In \cite{LZ20}, the authors built on previous joint works with Skinner and Pilloni \cite{LSZ17,LPSZ19} constructing Euler systems and $p$-adic $L$-functions for $\mathrm{GSp}_4$. For applications to Bloch--Kato in analytic rank 0, one wants to show the Euler system of \cite{LSZ17} is non-trivial. The main new input in \cite{LZ20} was an explicit reciprocity law relating the Euler system of \cite{LSZ17} to a specific value of the $p$-adic $L$-function of \cite{LPSZ19}. If this $p$-adic $L$-value does not vanish, then the Euler system is non-trivial and can be used to bound a Selmer group. 

This non-vanishing is delicate, since the $p$-adic $L$-value seen by the explicit reciprocity law is \emph{outside} the region of interpolation (so it does not directly relate to a Deligne-critical  $L$-value). In \cite[\S17]{LZ20}, Loeffler--Zerbes deform this into the region of interpolation -- and thus prove the Bloch--Kato conjecture -- conditional on the existence of a family of $p$-adic $L$-functions on $\GL_4$, stated as Theorem 17.6.2 \emph{op.\ cit}. This theorem, whose proof was deferred to `forthcoming work' of the present authors, is a special case of Theorem \ref{thm:intro family}.

	\subsection{Remarks on assumptions}
	
	In the above, and the main body of the paper, we restrict to base field $\Q$ and $\pi$ of tame level 1, a setting where all of our key new ideas are already present. These assumptions drastically simplify the notation and reduce technicality, allowing for a shorter, more conceptual article, whilst still including the results required by \cite{LZ20}. We indicate which of our various assumptions could be relaxed.
	
	Firstly, all of these results can be modified in a conceptually straightforward, but notationally intricate, way to work for $\GL_{2n}$ over an arbitrary totally real field $F$. This was the setting treated in \cite{BDW20}; the reader could consult that paper for the extra details occurring in this case.
		
\medskip

A second minor assumption is the weight condition that $\lambda_n > \lambda_{n+1}$. This rules out, for example, weight 2 modular forms in the case of $\GL_2$. This condition is used in exactly one place: to guarantee existence of a non-central, and hence non-vanishing, Deligne-critical $L$-values. It could be replaced by instead assuming the existence of finite order Hecke characters $\chi^\pm$ of $p$-power conductor, with $\chi^\pm_\infty(-1)(-1)^{\lambda_n} = \pm 1$, such that $L\left(\pi\otimes\chi^\pm,\tfrac{1}{2}-\lambda_n\right) \neq 0$ (see \cite[Def.\ 7.3]{BDW20}). One would expect this to always be true, and indeed it is known for weight 2 modular forms thanks to Rohrlich \cite{Roh89}.

\medskip

	Our most serious assumption is that $\pi$ has tame level 1 (is everywhere spherical).  We impose this for our global applications (Theorems \ref{thm:intro A} and \ref{thm:intro family}), as there is a  subtlety in choosing local test vectors in Shalika models when  $\pi_\ell$ is ramified. More precisely:
\begin{itemize} 
\item One knows there exists a test vector in $\pi_\ell$ whose Friedberg--Jacquet local zeta integral computes exactly the local $L$-factor for $\pi_\ell$. This test vector is smooth for  \emph{some} open compact subgroup $K_\ell \subset \GL_{2n}(\Q_\ell)$; but $K_\ell$ is inexplicit for ramified $\pi_\ell$. To relate our evaluation maps to $L$-values (in order to show their non-vanishing) requires working at level $K = \mathrm{Iw}_G \prod_{\ell \neq p} K_\ell$. 

\item Thanks to work of \cite{JPSS}, variation in families, however, works best  at \emph{new} level 
	\[
		K_1(\tilde\pi) = \mathrm{Iw}_G \prod_{\ell \neq p} K_{1,\ell}(m(\pi_\ell)), 
	\]
where  $K_{1,\ell}(m) \subset \mathrm{GL}_{2n}(\Z_\ell)$ is the open compact subgroup of all matrices whose bottom row is congruent to $(0,\dots,0,1) \newmod{\ell^m}$, and  $m(\pi_\ell)$ is the minimal $m \in \Z_{\geq 0}$ with $\pi_\ell^{K_{1,\ell}(m)} \neq \{0\}$. Crucially, \cite{JPSS} showed that $\pi_\ell^{K_{1,\ell}(m(\pi_\ell))}$ is a line.
\end{itemize}
This leads to tensions between working at level $K$, where we see $L$-values, and level $K_1(\tilde\pi)$, suitable for proving strong results about families. In tame level 1, we have  $K=K_1(\tilde\pi)$ and there is no problem. For simplicity -- and since this issue is somewhat perpendicular to the main novelties of the present paper -- we have presented all our proofs for tame level 1.
	
	\medskip
	We briefly indicate how to drop this assumption in proving Theorem \ref{thm:intro A}, following the strategy employed in 
	 \cite[\S7]{BDW20}, in the parahoric setting. There it was first shown that there exists a $p$-adic family of the correct dimension at level $K$, containing a very Zariski-dense set of classical points corresponding to RASCARs. Then, by studying conductors and the Local Langlands Correspondence, a delicate level-switching operation was carried out that transferred this family from level $K$ to level $K_1(\tilde\pi)$, where strong results can be deduced about its geometry. Since nothing about this argument used the parahoric setting, it applies equally well to our Iwahoric case, and the same methods prove the following generalisation of Theorem \ref{thm:intro A} (the base field still being $\Q$):

		\begin{theorem-intro-2}\label{thm:intro A prime}
			Let $\pi$ be a RASCAR of $G(\A)$ that is spherical and regular at $p$. Suppose that $\lambda_n > \lambda_{n+1}$, where $\pi$ has weight $(\lambda_1,\dots,\lambda_{2n})$. Let $\tilde\pi = (\pi, \alpha)$ be a non-critical slope spin $p$-refinement of $\pi$. Then the level $K_1(\tilde\pi)$ eigenvariety $\sE^{G}(K_1(\tilde\pi))$ for $G$ is \'etale over $\sW_0^G$ at $\tilde\pi$. There exists a neighbourhood of $\tilde\pi$ in $\sE^{G}(K_1(\tilde\pi))$ containing a very Zariski-dense set of classical points corresponding to RASCARs.
		\end{theorem-intro-2}

	However, to prove Theorem \ref{thm:intro family} requires much more control: we need not only that systems of Hecke eigenvalues vary $p$-adically, but also that we can vary all the local test vectors over the family. This is unconditionally possible in tame level 1. This could be relaxed to the assumption that $\pi$ admits certain parahoric-fixed vectors at every finite place using the forthcoming work \cite{DJ-parahoric}. For a discussion on possible generalisations beyond this, see \cite[\S8]{BDW20}.

	\subsection{Structure of the paper}
	
	This paper falls into three parts. 
	
	In Part I (\S\ref{sec:preliminaries}--\ref{sec:abstract evaluation maps}), we fix notation and recall relevant automorphic results. In \S\ref{sec:abstract evaluation maps}, we generalise the abstract construction of evaluation maps from \cite[\S4]{BDW20}, showing that these evaluation maps compute classical $L$-values of RASCARs. 
	
	In Part II (\S\ref{sec:local zeta p}--\S\ref{sec:local zeta unramified}), we develop the theory of spin and Shalika refinements, and compute local zeta integrals (in two different ways). In these sections we prove all the results described in \S\ref{sec:intro shalika}. Our Iwahoric results are summarised in \S\ref{sec:running assumptions}.
	
	In Part III (\S\ref{sec:overconvergent cohomology}--\S\ref{sec:shalika families}), we build our $p$-adic machine on this automorphic foundation, reinterpreting the above in the context of overconvergent cohomology. The heart of  is \S\ref{sec:branching laws}-\ref{sec:distribution valued evaluations}, where we give our main technical results on $p$-adic interpolation of branching laws. In \S\ref{sec:shalika families}, we obtain our main arithmetic applications, following strategies developed in \cite{BDW20}.
	
	\subsection*{Acknowledgements}
	This paper owes a great debt to the observations of David Loeffler, and we thank him profusely. We thank Mahdi Asgari for patiently answering our questions about his work on $\mathrm{GSpin}_{2n+1}$. We also thank the three anonymous referees for their valuable comments and corrections.
	
	DBS was supported by ANID FONDECYT grant 11201025 and ANID PAI grant 77180007. MD was supported by Agence	Nationale de la Recherche grant ANR-18-CE40-0029. AG was supported in part by ERC-2018-COG-818856-HiCoShiVa. AJ is grateful for support from CEMPI, Lille, the Center for Mathematics at Notre Dame, and the NSF. CW was supported by EPSRC Postdoctoral Fellowship EP/T001615/1. Finally we thank the CIRM in Luminy and the CMND program on Eigenvarieties in Notre Dame, where parts of this work were carried out.

	%%================================================================
	%%
	%%				PRELIMINARIES
	%%
	%%================================================================

	\part{Automorphic Results}
	
	\section{Set-up and notation}\label{sec:preliminaries}

	\subsection{Notation}\label{sec:notation}
We will use the following notation for almost all of this paper. In Sections 5, 6, 7 and 9,  however, we will use purely local analogues of this notation, as fixed in the introduction to Part II.

	Let $n\geq 1$ and let $G \defeq \mathrm{GL}_{2n}$. We write $B $ for the Borel subgroup of upper triangular matrices, $\overline{B} $ for the opposite Borel of lower triangular matrices and $T$ for the maximal split torus of diagonal matrices. We have decompositions $B= TN$ and $\overline{B} = \overline{N}T$ where $N = N_{2n}$ and $\overline{N} = \overline{N}_{2n}$ are the unipotent radicals of $B$ and $\overline{B}$. We also let $G_n = \GL_n$, with $B_n, T_n, N_n$ etc.\ the analogous subgroups. Let $H \defeq \GL_n \times \GL_n$, with an embedding $\iota : H \hookrightarrow G$, $\iota(h_1,h_2) = \smallmatrd{h_1}{0}{0}{h_2}$. 
	
	Let $\cW_G = \mathrm{S}_{2n}$ (resp.\ $\cW_{n} = \mathrm{S}_n$) be the Weyl group of $G$ (resp.\ $G_n$), identified with the permutation subgroup of $G(\Z)$ (resp.\ $G_n(\Z)$). We write $w_{2n}$ and $w_n$ for the longest Weyl elements (i.e. the antidiagonal matrices with 1s on the antidiagonal). 
	
	Let $K_\infty=C_\infty Z_\infty$, where  $Z_\infty$ is the center and $C_\infty$ is the maximal compact subgroup $O_{2n}(\R) \leq G(\R)$. 
	If $A$ is a reductive real Lie group, then $A^\circ$ denotes the connected component of the identity.

	Fix a rational prime $p$ and an embedding $i_p : \overline{\Q} \hookrightarrow \overline{\Q}_p$. We fix a (non-canonical) extension of $i_p$ to an isomorphism $i_p : \C \isorightarrow \overline{\Q}_p$. Let $\Q^{p\infty}$ be the maximal abelian extension of $\Q$ unramified outside $p\infty$, and let $\Galp \defeq  \mathrm{Gal}(\Q^{p\infty}/ \Q) \cong \Zp^\times$ be its Galois group.
	
		If $\pi$ is a regular algebraic cuspidal automorphic representation (RACAR) of $G(\A)$, then we write $L(\pi,s)$ for its standard $L$-function.

	Throughout, we work in `tame level 1', that is, with the open compact level subgroup 
	\begin{equation} \label{eq:K}
		K = \Iw \cdot \prod_{\ell\neq p} G(\Z_\ell)\  \subset G(\A_f),
	\end{equation}
	so away from $p$ we take maximal hyperspecial level and at $p$ we take Iwahori level
	\begin{equation}\label{eq:Iw G}
		\Iw \defeq \{g \in G(\Zp) : g \newmod{p} \in B(\F_p)\} \subset G(\Zp).
	\end{equation}
	Let $\delta_B : T(\A) \to \C^\times$ be the standard modulus character 
	\begin{equation}\label{eq:delta_B}
		t = (t_1,...,t_{2n}) \mapsto |t_1|^{2n-1}|t_2|^{2n-3}\cdots |t_{2n-1}|^{3-2n}|t_{2n}|^{1-2n}.
	\end{equation}
	We repeatedly use that if $\pi_\ell$ is the local component at $\ell$ of a RACAR $\pi$, and $\pi_\ell^{G(\Z_\ell)} \neq 0$ (i.e. $\pi_\ell$ is spherical), 
 then $\pi_\ell$ is a generic unramified principal series representation. By this, we mean there exists an unramified character $\UPS : T(\Ql) \to \C^\times$ such that
	\begin{equation}\label{eq:unramified induced}
		\pi_\ell = \Ind_B^G \UPS
	\end{equation}
is the normalised parabolic induction\footnote{To see this, note by the Bernstein--Zelevinsky classification that any generic irreducible $\pi_\ell$ is isomorphic to an induction of a tensor product of unlinked segments (see \cite[Thm.\ 9.7(b)]{Zel80}, stated in this form in \cite[Thm.\ 8.4.4]{GetzHahn}). If $\pi_\ell$ is unramified, these segments are all unramified, and hence all are characters by \cite[Cor.\ 1.3]{Mat13}. But this says exactly that $\pi_\ell \cong \Ind_B^G \UPS$ for some unramified character $\UPS$ of $T(\Q_\ell)$.} of $\theta$ to $G$ (that is, the true induction of $\delta_B^{1/2}\UPS$).

	All our group actions will be on the left. If $M$ is a $R$-module, with a left action of a group $\Gamma$, then we write $M^\vee = \mathrm{Hom}_R(M,R)$, with associated left dual action $(\gamma \cdot\mu)(m) = \mu(\gamma^{-1} \cdot m)$.
	
	In later sections we work extensively with affinoid rigid spaces. For such a space $X$, we write $\cO_X$ for the ring of rigid functions on $X$, so $X = \mathrm{Sp}(\cO_X)$.
	
	All Hecke characters in this paper are understood to be idelic, i.e.\ characters of $\Q^\times\backslash \A^\times$. Finite order Hecke characters correspond bijectively with Dirichlet characters.

	\subsection{Algebraic weights}\label{sec:algebraic weights}
	Let $X^{\ast}(T)$ be the set of algebraic characters of $T$. Each element of $X^{\ast}(T)$ corresponds to an integral weight  $\lambda= (\lambda_1,...,\lambda_{2n}) \in \Z^{2n}$. If $\lambda_1 \geq \lambda_2 \geq \cdots \geq \lambda_{2n}$, we say $\lambda$ is \emph{dominant}, and write $X^{\ast}_+(T) \subset X^{\ast}(T)$ for the subset of dominant weights. We say that $\lambda$ is \emph{pure} if there exists $\sw  \in \Z$, the \emph{purity weight} of $\lambda$, such that $\lambda_{i}+ \lambda_{2n- i+ 1}= \sw $ for all $1 \leq i \leq n$; we write $\mathrm{pure}(\lambda) \defeq \sw$. We write $X^{\ast}_0(T) \subset X^{\ast}_+(T)$ for the subset of pure $B$-dominant  integral weights, which are exactly those supporting   cuspidal cohomology \cite[Lem.~4.9]{Clo90}.  
	
	For $\lambda \in X^{\ast}_{+}(T)$, let $V_{\lambda}$ be the algebraic irreducible representation of $G$ of highest weight $\lambda$, and let $V_{\lambda}^\vee$ denote its linear dual, with its (left) dual action. We have an isomorphism $V_{{\lambda}}^{\vee} \cong V_{\lambda^\vee}$ where $\lambda^{\vee}= (-\lambda_{2n}, \dots , -\lambda_{1}).$ Given a pure dominant algebraic weight $\lambda \in X_0^*(T)$, let
	\begin{equation}\label{eqn:crit lambda}
		\mathrm{Crit}(\lambda) \defeq \{j \in \Z: -\lambda_{n} \leqslant j \leqslant -\lambda_{n+1}\}.
	\end{equation}
	If $\pi$ is a RACAR for $G(\A)$ of weight $\lambda$ (which we take to mean cohomological with respect to $V_\lambda^\vee$), then $j \in \mathrm{Crit}(\lambda)$ if and only if the $L$-value  $L(\pi,j+\tfrac{1}{2})$ is  Deligne-critical (see \cite[\S6.1]{GR2}).

	\subsection{Shalika models}\label{sec:shalika models}
	Our main results come in the setting of RACARs that admit \emph{Shalika models}. We recall relevant definitions and properties (see e.g.\ \cite[\S1,\S3.1]{GR2}, \cite[\S2.6]{BDW20}).
	
	\subsubsection{Definition of Shalika models}\label{sec:shalika basic properties}
	Let $\cS = \{s = \smallmatrd{h}{}{}{h}\cdot \smallmatrd{1_n}{X}{}{1_n}: h \in \GL_n, X \in \mathrm{M}_n\}$ be the Shalika subgroup of $\GL_{2n}$. Let $\psi$ be the standard non-trivial additive character of $\Q\backslash \A$ from \cite[\S4.1]{DJR18}, and let $\eta$ be a character of $\Q^\times\backslash\A^\times$.  Let $\eta\otimes\psi$ be the character of $\cS(\A)$ given by $(\eta\otimes\psi)(s) = \eta(\det(h))\psi(\mathrm{tr}(X))$. 
	
	A cuspidal automorphic representation $\pi$ of $G(\A)$ is said to have an $(\eta,\psi)$-\emph{Shalika model} if there exist $\varphi \in \pi$ and $g \in G(\A)$ such that
	\begin{equation}\label{eq:shalika integral}
		\cS_{\psi}^\eta(\varphi)(g) \defeq \int_{Z_G(\A)\cS(\Q)\backslash\cS(\A)} \varphi(sg) \ (\eta \otimes \psi)^{-1}(s)ds \neq 0.
	\end{equation}
	If this non-vanishing holds, then:
	\begin{itemize}\setlength{\itemsep}{0pt}
		\item $\eta^n$ is the central character of $\pi$, 
		\item $\eta$ has the form $\eta_0|\cdot|^{\sw}$, where $\eta_0$ has finite order and $\sw = \mathrm{pure}(\lambda)$ is the purity weight of the weight $\lambda$ of $\pi$,
		\item and $\cS_\psi^\eta$ defines an intertwining $\pi \hookrightarrow \mathrm{Ind}_{\cS(\A)}^{G(\A)}(\eta \otimes \psi)$.
\end{itemize}
	
	If $\pi$ has an $(\eta,\psi)$-Shalika model, then for each prime $\ell$ the local component $\pi_\ell$ has a local $(\eta_\ell,\psi_\ell)$-Shalika model \cite[\S3.2]{GR2}, that is, we have (non-canonical) intertwinings 
	\begin{equation}\label{eq:shalika integral local}
		\cS_{\psi_\ell}^{\eta_\ell} : \pi_\ell \hookrightarrow \mathrm{Ind}_{\cS(\Q_\ell)}^{\GL_{2n}(\Q_\ell)}(\eta_\ell\otimes\psi_\ell).
	\end{equation}
	We fix a choice of intertwining $\cS_{\psi_f}^{\eta_f}$ of $\pi_f$ (or equivalently, via \eqref{eq:shalika integral}, an intertwining $\cS_{\psi_\infty}^{\eta_\infty}$ of $\pi_\infty$).
	
	By \cite[Prop.~1.3]{AG94}, if $\pi_\ell$ is spherical then it admits a $(\eta_\ell,\psi_\ell)$-Shalika model if and only if $\pi_\ell^\vee =\pi_\ell\otimes \eta_\ell^{-1}$, in which case $\eta_\ell$ is unramified. 
	
	We recall another characterisation of Shalika models:
	
	\begin{theorem}[Jacquet--Shalika, Asgari--Shahidi]\label{thm:transfer}
		Let $\pi$ be a cuspidal automorphic representation of $G(\A)$. Then $\pi$ admits a $(\eta,\psi)$-Shalika model for some character $\eta$ if and only if $\pi$ is the Langlands functorial transfer of a globally generic cuspidal automorphic representation $\Pi$ of $\mathrm{GSpin}_{2n+1}(\A)$.
	\end{theorem}
	\begin{proof}
		By \cite{JS90}, having a global Shalika model is equivalent to a (partial) exterior square $L$-function having a pole at $s=1$. But in \cite{AS06,AS14} this is shown to be equivalent to being such a functorial transfer. (For further details see \cite[Prop.\ 3.1.4]{GR2}).
	\end{proof}

	\subsubsection{Friedberg--Jacquet integrals}\label{sec:friedberg-jacquet}
	Let $\pi$ be a cuspidal automorphic representation of $G(\A)$, and $\chi$ a finite order Hecke character for $F$. For $W \in \cS_{\psi}^{\eta}(\pi)$ consider the \emph{Friedberg--Jacquet} zeta integral
	\[
	\zeta(s,W,\chi) \defeq \int_{\GL_n(\A)} W\left[\matrd{h}{}{}{I_n}\right] \ \chi(\det(h)) \ |\det(h)|^{s-\tfrac{1}{2}} \ dh,
	\]
	converging absolutely in a right-half plane and extending to a meromorphic function in $s \in \C$. When $W = \bigotimes_{\ell\leq\infty} W_\ell$ for $W_\ell \in \cS_{\psi_\ell}^{\eta_\ell}(\pi_\ell)$, this decomposes into a product of local zeta integrals $\zeta_\ell(s,W_\ell,\chi_\ell)$. 
	Suppose $\pi$ is a RACAR admitting a $(\eta,\psi)$-Shalika model. If $\pi_\ell$ is spherical, then $\pi_\ell^{G(\Z_\ell)}$ is a line; let $W_{\ell}^{\circ} \in \cS_{\psi_\ell}^{\eta_\ell}(\pi_\ell^{G(\Z_\ell)})$ be the spherical test vector normalised so that $W_{\ell}^{\circ}(1_{2n}) = 1$. Then by \cite[Props.~3.1, 3.2]{FJ93} $W_\ell^\circ$ is a \emph{Friedberg--Jacquet test vector}, i.e.\ for all unramified quasi-characters $\chi_\ell : \Q_\ell^\times \to \C^\times$ we have
	\begin{equation}\label{eq:jacquet-friedberg test vector}
		\zeta_\ell\left(s+\tfrac{1}{2}, W^{\circ}_\ell, \chi_\ell\right) =  L\left(\pi_\ell \otimes \chi_\ell,s+\tfrac{1}{2}\right).
	\end{equation}
	We apply this to choose local test vectors at all finite $\ell \neq p$.

	\subsection{The Hecke algebra and $p$-refinements}\label{sec:hecke algebra}
	
	Recall we took $K = \Iw \cdot \prod_{\ell\neq p} G(\Z_\ell)$.

	\subsubsection{Away from $p$}\label{sec:hecke away from p} If $\nu \in X_*^+(T)$ is a dominant cocharacter, and $\ell \neq p$ a prime, define $T_{\nu,\ell} \defeq [G(\Z_\ell) \cdot \nu(\ell)\cdot G(\Z_\ell)]$, a double coset operator. 
	
	\begin{definition}
		The \emph{spherical Hecke algebra} is the commutative algebra $\cH' \defeq \Z[T_{\nu,\ell} : \nu \in X_*^+(T), \ell \neq p]$ generated by all such operators.
	\end{definition}
	
	If $\pi$ is a RACAR with $\pi^K \neq 0$, then $\pi_\ell$ is spherical, and $\pi_\ell^{G(\Z_\ell)}$ is a line preserved by each $T_{\nu,\ell}$. Let $E$ be a number field containing the Hecke field of $\pi_f$ (which exists by \cite[Thm.\ 3.13]{Clo90}). Attached to $\pi$ there is a homomorphism $\psi_\pi : \cH' \otimes E \to E$ sending $T_{\nu,\ell}$ to its eigenvalue on $\pi_\ell^{G(\Z_\ell)}$. We define $\m_\pi \defeq \ker(\psi_\pi)$, a maximal ideal in $\cH'\otimes E$. If $M$ is a module on which $\cH'\otimes E$ acts, we write $M_{\pi}$ for its localisation at $\m_{\pi}$.

	\subsubsection{The Hecke algebra at $p$}\label{sec:hecke p 1}
	For $i = 1,...,2n-1$ define matrices $t_{p,r} \in T(\Qp)$ by
	\begin{equation}\label{eq:t_p i}
		t_{p,1} = \mathrm{diag}(p,1,...,1),\ \  t_{p,2} = \mathrm{diag}(p,p,1,...,1),\  ..., \ t_{p,2n-1} = \mathrm{diag}(p,...,p,1),
	\end{equation}
	and let
	\begin{equation}\label{eq:t_p}
		t_p = t_{p,1}\cdots t_{p,2n-1} = \mathrm{diag}(p^{2n-1},p^{2n-2},..., p,1) \in T(\Qp).
	\end{equation}
	Then define operators $U_{p,r} = [\Iw \cdot t_{p,r} \cdot \Iw]$ and $U_p = [\Iw \cdot t_p \cdot \Iw] = U_{p,1}\cdots U_{p,2n-1}$. 
	
	\begin{definition}\label{def:hecke algebra p}
		The \emph{Hecke algebra at $p$} is $\cH_p \defeq \Z[U_{p,r} : r = 1,...,2n-1]$, and the \emph{full Hecke algebra} is $\cH = \cH' \otimes \cH_p$. 
	\end{definition}

	\subsubsection{$p$-refinements}\label{sec:hecke p}
	Suppose $\pi_p$ is spherical. In particular $\pi_p^{G(\Zp)} \neq 0$, hence $\pi_p^{\Iw} \neq 0$. 

\begin{definition}\label{def:refinement}
A \emph{$p$-refinement} of $\pi_p$ is a pair $\tilde\pi_p = (\pi_p, \alpha)$, where $\alpha : \cH_p \to \overline{\Q}$ is a system of $\cH_p$-eigenvalues appearing in $\pi_p^{\Iw}$; i.e.\ if we set $\alpha_{p,r} = \alpha(U_{p,r})$, then there is an eigenvector $\varphi_p \in \pi_p^{\Iw}$ with $U_{p,r}\varphi_p = \alpha_{p,r}\varphi_p$ for each $r$. Such a $\tilde\pi_p$ is \emph{regular} if the attached generalised eigenspace
	\begin{equation}\label{eq:U_p refined line}
		  \pi_p^{\Iw}[\![U_{p,r} - \alpha_{p,r} : r = 1,...,2n-1]\!] \defeq \left\{ \varphi_p \in \pi_p^{\Iw} : \begin{array}{c} \exists \; m \geq 1 \text{ such that } \\ \left(U_{p,r} - \alpha_{p,r}\right)^m \varphi_p = 0 \\ \text{ for all } r=1, \dots, 2n-1 \end{array} \right\}
	\end{equation}
is one-dimensional over $\C$. We will write $\varphi_p \in \tilde\pi_p$ as shorthand for $\varphi_p \in \pi_p^{\Iw}[\![U_{p,r} - \alpha_{p,r} : r = 1,...,2n-1]\!]$.
\end{definition}

	Let $\tilde\pi$ be a $p$-refinement. After possibly extending $E$, we extend $\psi_\pi$ to a homomorphism
	\begin{equation}\label{eq:psi tilde pi}
		\psi_{\tilde\pi} : \cH \otimes E \longrightarrow E, \hspace{12pt} U_{p,r} \mapsto \alpha_{p,r}.
	\end{equation}
	We let $\m_{\tilde\pi} \defeq \ker(\psi_{\tilde\pi})$ be the corresponding maximal ideal of $\cH \otimes E$. If $M$ is a module upon which $\cH\otimes E$ acts, we write $M_{\tilde\pi}$ for  the localisation of $M$ at $\m_{\tilde\pi}$. Note that if $L$ is a field containing $E$ and $M$ is a finite-dimensional $L$-vector space, then $M_{\tilde\pi}$ is the generalised eigenspace for  $\cH\otimes L$ corresponding to $\psi_{\tilde\pi}$. 
	
	 Let $\pi_p = \Ind_B^G \UPS$ be a generic unramified principal series representation (see \eqref{eq:unramified induced}). We recall the standard classification of $p$-refinements for $\pi_p$. Recall $\cW_G = \mathrm{S}_{2n}$ and $\delta_B$ from \S\ref{sec:notation}.
	
	\begin{proposition}\emph{\cite[Lem.\ 4.8.4]{Che04}.} \label{prop:p-refinement}
		\begin{enumerate}[(i)]
			\item 
		 The semisimplification of $\pi_p^{\Iw}$ as a $\cH_p$-module is isomorphic to $\bigoplus_{\sigma\in \cW_G}(\delta_B^{1/2}\UPS^\sigma) \circ \mathrm{ev}_p$, where $\mathrm{ev}_p : \cH_p \to T(\Qp)$ is the map sending $U_{p,r} \mapsto t_{p,r}$. 
		Thus if $\tilde\pi_p = (\pi_p, \alpha)$ is a $p$-refinement, then there exists $\sigma \in \cW_G$ such that for each $r$,
		\[
		\alpha_{p,r} = \alpha(U_{p,r}) = \left[\delta_B^{1/2}\UPS^\sigma\right]^{w_{2n}}(t_{p,r}) = \prod_{j=1}^r p^{\tfrac{2n-2j+1}{2}}\UPS_{\sigma(2n+1-j)}(p).
		\]
	\item There are equivalences
	\begin{align*}
		\text{$\pi_p$ admits a regular $p$-refinement} &\iff \text{its Satake parameter is regular}\\
		&\iff\text{every $p$-refinement of $\pi_p$ is regular}.
	\end{align*} 
	In this case, the choice of $\sigma$ in (i) is unique, and via this correspondence there are exactly $(2n)!$ $p$-refinements of $\pi_p$, all of which are regular.
\end{enumerate}
	\end{proposition}
	
	\begin{remark}\label{rem:conjugate}
		\begin{itemize}\setlength{\itemsep}{0pt}
			\item[(i)] This normalisation, where the character $\delta_B^{1/2}\UPS^\sigma$ is conjugated by the Weyl group element $w_{2n}$, matches \cite{DJR18} but might appear strange. Chenevier uses antidominant cocharacters, and switching to dominant characters is equivalent to conjugating by $w_{2n}$. This normalisation will be convenient in \S\ref{sec:shalika refinements}.
			
			\item[(ii)] When the Satake parameter is regular, there is a bijection 
\begin{equation}\label{eq:delta ups}
\Delta_\UPS : \{p\text{-refinements of }\pi_p\}  \longrightarrow \cW_G
\end{equation}
 induced by the above. This is \emph{not} canonical, depending on the choice of character $\UPS$ from which we induce. For $\tau \in \cW_G$, replacing $\UPS$ by $\UPS^\tau$ conjugates the image of $\Delta_\UPS$ by $\tau$. 

		\end{itemize}
	\end{remark}
	
	 When $\pi_p$ is the local component of a RACAR, the eigenvalues $\alpha_{p,r}$ are algebraic but not necessarily $p$-integral. To account for this, we make the following definition.
	
	\begin{definition}\label{def:integral}
		Let $\tilde\pi_p = (\pi_p,\alpha)$ be a $p$-refinement. Define \emph{integral normalisations} 
\[
U_{p,r}^\circ \defeq \lambda(t_{p,r})U_{p,r}, \qquad \alpha_{p,r}^\circ \defeq \lambda(t_{p,r}) \alpha_{p,r} = p^{\lambda_{1} + \cdots + \lambda_{r}} \alpha_{p,r}.
\]
 The $\alpha_{p,r}^\circ$ are $p$-integral (see Remark \ref{rem:integral normalisation} or \cite[Rem.\ 3.23]{BW20}).
	\end{definition}

	%%================================================================
	%%
	%%				AUTOMORPHIC COHOMOLOGY
	%%
	%%================================================================

	\section{Automorphic cohomology classes}\label{sec:automorphic cohomology}
	
	For $K \subset G(\A_f)$ an open compact subgroup, the \emph{locally symmetric space of level $K$} is 
\[
S_K = G(\Q)\backslash G(\A)/K K_\infty^\circ.
\]
 It is a $(2n-1)(n+1)$-dimensional real orbifold. We will now recall how to realise RACARs in the compactly supported Betti cohomology of $S_K$.

	\subsection{Local systems}\label{sec:local systems} We recall standard facts about local systems on $S_K$ (e.g.\ \cite[\S1]{Urb11}, \cite[\S2.3]{BDW20}). If $M$ is a left $G(\Q)$-module such that  $Z_G(\Q)\cap KK_\infty^\circ$ acts trivially, let $\cM$ be the local system on $S_K$ given by locally constant sections of $G(\Q)\backslash [G(\A) \times M]/KK_\infty^\circ$, with action $\gamma(g,m)kz = (\gamma g kz, \gamma \cdot m)$. We denote such local systems with calligraphic letters.
	
	If $M$ is a left $K$-module, let $\sM$ (with a script letter) be the local system on $S_K$ given by locally constant sections of $G(\Q)\backslash [G(\A) \times M]/KK_\infty^\circ$ with action $\gamma(g,m)kz = (\gamma gkz, k^{-1}\cdot m)$.
	
	If $M$ is a left $G(\A)$-module, then it has actions of the subgroups $G(\Q)$ and $K$, and there is an isomorphism $\cM\isorightarrow\sM$ of associated local systems given by $(g,m)  \mapsto (g,g_f^{-1} \cdot m)$. The key example of such $M$ for this paper is $M = V_\lambda^\vee$, whence $\cV_{\lambda}^\vee \isorightarrow \sV_{\lambda}^\vee$.

	\subsection{Hecke operators}
	
	Let $\gamma \in G(\A_f)$ and $M$ be a left $G(\Q)$-module (resp.\ $K$-module). We suppose $\gamma$ acts on $M$. We have a natural projection map $p_{K,\gamma} : S_{\gamma K \gamma^{-1}\cap K} \to S_K$, and a double coset operator $[K\gamma K]$ on $\hc{\bullet}(S_K,\cM)$ (resp.\ $\hc{\bullet}(S_K,\sM)$) defined as the composition 
	\begin{equation}\label{eq:Hecke composition}
		[K\gamma K] \defeq \mathrm{tr}(p_{K,\gamma}) \circ [\gamma] \circ p_{K,\gamma^{-1}}^*,
	\end{equation}
	where $\mathrm{tr}$ is the trace and $[\gamma] : \hc{\bullet}(S_{K\cap\gamma^{-1}K\gamma},\cM) \to \hc{\bullet}(S_{\gamma K \gamma^{-1}\cap K}, \cM)$ is given on local systems by $(g,m) \mapsto (g\gamma^{-1}, \gamma\cdot m)$ (and similarly for $\sM$).
	
	\subsubsection{Localisation at RACARs} \label{sec:localisation}
	Recall $K = \Iw \cdot \prod_{\ell\neq p} G(\Z_\ell)$ and $\cH$ from \S\ref{sec:hecke algebra}. For appropriate $M$ (e.g.\ $M = V_\lambda^\vee$), this acts on $\hc{\bullet}(S_K,\cM)$ and $\hc{\bullet}(S_K,\sM)$ via the process above. If $\pi$ is a RACAR with $\pi^K \neq 0$, it therefore makes sense to localise $\hc{\bullet}(S_K,\cM)$ at $\m_\pi$ as in \S\ref{sec:hecke away from p}. We denote the localisation by $\hc{\bullet}(S_K,-)_\pi$.

	\subsubsection{The action at infinity}
	
	We have $K_\infty/K_\infty^\circ = \{\pm 1\}$. This group has two characters $\epsilon^\pm : K_\infty/K_\infty^\circ \to \{\pm 1\}$, where $\epsilon^\pm$ sends $-1$ to $\pm 1$. If $M$ is a module on which $K_\infty/K_\infty^\circ$ acts and 2 acts invertibly -- for example, the cohomology of $S_K$ over a field of characteristic not 2 -- then we have $M = M^+ \oplus M^-$, where $M^\pm$ are the eigenspaces where $K_\infty/K_\infty^\circ$ acts via $\epsilon^\pm$. We obtain a (Hecke-equivariant) decomposition of the cohomology groups $\hc{\bullet}(S_K,-)$ into $\pm$-submodules (as the action of $K_\infty/K_\infty^\circ$ commutes with the $G(\A_f)$-action).

	\subsubsection{Integral normalisations}\label{sec:integral normalisations cohomology}
	
	The module $V_\lambda^\vee$ comes equipped with the natural (algebraic) action of $\GL_{2n}$, which we have been denoting with a $\cdot$. As we have already remarked, the resulting Hecke operators $U_{p,r} = U_{p,r}^\cdot = [K_p t_{p,r} K_p]$ on the cohomology of $\sV_{\lambda}^\vee$ are not integrally normalised. 
	
	In \S\ref{sec:slope-decomp}, we will equip $V_\lambda^\vee(\overline{\Q}_p)$ with another natural action of $\GL_{2n}(\Zp)$ and $t_{p,r}$, denoted $*$. Concretely, we will have $t_{p,r} * \mu = \lambda(t_{p,r})(t_{p,r} \cdot \mu)$. In light of Definition \ref{def:integral}, if we let $U_{p,r}^*$ be the Hecke operator defined via \eqref{eq:Hecke composition} with the $*$-action instead of the $\cdot$-action, then $U_{p,r}^* = U_{p,r}^\circ$ is integrally normalised. This is all explained in detail in the remark at the end of \cite[\S3.3]{BDW20}.

	\subsection{Cohomology classes attached to RACARs} \label{sec:cohomology for RACARs}
	
	Let $t = n^2 + n -1$, which is the top degree of cohomology to which RACARs for $G(\A)$ contribute. In particular, let $\pi$ be a RACAR of $G(\A)$; then we recall that there exists a Hecke-equivariant isomorphism
	\begin{equation}\label{eq:pi to cohomology}
		\pi_f^K \isorightarrow \hc{t}(S_K,\sV_\lambda^\vee(\overline{\Q}_p))^\pm_{\pi},
	\end{equation}
	for a unique $\lambda \in X_0^*(T)$. The isomorphism \eqref{eq:pi to cohomology} is non-canonical, depending on our fixed choice of $i_p : \C \isorightarrow \overline{\Q}_p$ and a choice of basis $\Xi_{\infty}^\pm$ of the 1-dimensional $\C$-vector space $\h^t(\fg_\infty, K_\infty^\circ; \pi_\infty\otimes V_\lambda^\vee(\C))^\pm$, where $\fg_\infty \defeq \mathrm{Lie}(G_\infty)$. This is all standard, explained e.g. in \cite[\S2.5]{BDW20}.
	
	Suppose $\pi$ admits an $(\eta,\psi)$-Shalika model, and recall we chose an intertwining $\cS_{\psi_f}^{\eta_f} : \pi_f \to \cS_{\psi_f}^{\eta_f}(\pi_f)$. Combining with \eqref{eq:pi to cohomology}, we get a (non-canonical) Hecke-equivariant isomorphism
	\begin{equation}\label{eq:Theta}
		\Theta^{\pm} : \cS_{\psi_f}^{\eta_f}(\pi_f^K) \isorightarrow \hc{t}(S_K,\sV_\lambda^\vee(\overline{\Q}_p))^\pm_{\pi}.
	\end{equation}
	
	Possibly enlarging the number field $E$, there is a natural $E$-rational subspace $\cS_{\psi_f}^{\eta_f}(\pi_f^K, E) \subset \cS_{\psi_f}^{\eta_f}(\pi_f^K)$. As in \cite[Prop.\ 4.2.1]{GR2} (cf.\ \cite[\S2.10]{BDW20}), there exist $\Omega_\pi^\pm \in \C^\times$ (canonical up to $E^\times$-multiple) and finite $L/\Qp$ such that $\Theta^\pm/i_p(\Omega_\pi^\pm)$ maps $\cS_{\psi_f}^{\eta_f}(\pi_f^K,E)$ into $\hc{t}(S_K,\sV_\lambda^\vee(L))^\pm_{\pi}$. Moreover, for $\ell \neq p$ the spherical test vector $W_\ell^\circ$ is $E$-rational.

	%%================================================================
	%%
	%%				EVALUATION MAPS
	%%
	%%================================================================
	
	\section{Evaluation maps}\label{sec:abstract evaluation maps}

	\emph{Evaluation maps} were crucial to the methods of \cite{GR2,DJR18,BDW20}. We give constructions of abstract evaluation maps, generalising \cite{BDW20} and \cite{DJR18}.

	\subsection{Automorphic cycles and abstract evaluation maps}\label{sec:auto cycles}
	In this section we generalise the abstract theory in \cite[\S4]{BDW20}, where the evaluation maps were defined with respect to the parabolic $Q$ with Levi $H$. These `parahoric' evaluation maps can be interpolated over $2$-dimensional parabolic subsets of weight space, but are not suitable for our goal of interpolation in $(n+1)$-weight variables. We now construct evaluation maps defined with respect to any standard parabolic $P \subset Q$. For the Iwahoric case, we are most interested in $P = B$. The proofs of \cite[\S4]{BDW20} go through almost identically with these modifications, so we are terse with details here. 
	
	\begin{remark}
		Since the notation is heavy, we sketch the differences between our new definitions and those of \cite{BDW20}. Firstly, to better suit the more general theory, we replace the twisting operator $\xi = \smallmatrd{1_n}{w_n}{0}{w_n}$ of \cite[Def.\ 4.2]{BDW20} with $u^{-1}$, where 
		\begin{equation}\label{eq:u}
			u = \smallmatrd{1_n}{w_n}{0}{1_n} \in G(\Zp).
		\end{equation}
		 Unlike $\xi$, the element $u^{-1}$ lies in $\Iw$. We will show that the definitions/results of \cite{BDW20} are essentially unchanged with this switch.
		
		In  \cite{DJR18,BDW20}, the evaluation maps for $Q$ used the matrix $t_{p,n} = \mathrm{diag}(p,...,p,1,...,1) \defeqrev t_Q$ and operator $U_{p,n}^\circ = [K_p t_{p,n} K_p] \defeqrev U_Q^\circ$, a $Q$-controlling operator (in the sense of \cite[\S2.5]{BW20}). For a general parabolic $P$, we instead use a different matrix $t_P \in \GL_{2n}(\Qp)$ (see Definition \ref{def:parabolic}), giving a Hecke operator $U_P^\circ$ attached to $P$.
	\end{remark}
	
	\subsubsection{Automorphic cycles}\label{sec:auto cycles subsec}
	Automorphic cycles are coverings of locally symmetric spaces for $H$ that have real dimension equal to $t$, the top degree of cohomology to which RACARs for $G$ contribute. This `magical numerology' was exploited in \cite{GR2,DJR18} to define classical evaluation maps and to give a cohomological interpretation of the Deligne-critical $L$-values of RASCARs. 
	
	\begin{definition}\label{def:parabolic}
		Let $P\subset Q \subset \GL_{2n}$ be  a standard parabolic with Levi $\GL_{m_1} \times \cdots \times \GL_{m_r}$. Define a block diagonal matrix 
		\[
			t_P = \mathrm{diag}(p^{r-1}\mathrm{I}_{m_1}, p^{r-2}\mathrm{I}_{m_2}, \cdots, p\mathrm{I}_{m_{r-1}}, \mathrm{I}_{m_r}),
		\]
		where for an integer $m \geq 1$, we let $\mathrm{I}_m$ denote the $(m \times m)$ identity matrix. Note $t_P \in T_P^{++}$ as in \cite[\S2.5]{BW20}; and since $P \subset Q$, the first $n$ diagonal entries are all a positive power of $p$.
	\end{definition}

	For example, we have $t_Q = t_{p,n}$ (from \eqref{eq:t_p i}), and 
	\[
	t_B= \mathrm{diag}(p^{2n-1},p^{2n-2},...,p,1).
	\]
	The most significant change in passing from the $Q$-evaluation maps (from \cite{BDW20}) to those here is the use of the matrix $t_P$ rather than $t_Q$.

	Define $J_P \subset \GL_{2n}(\Zp)$ to be the parahoric subgroup for $P$.

	\begin{definition}\label{def:auto cycle}
		Fix $m \in \Z_{\geq 0}$ prime to $p$, and let $K = K_pK^p \subset G(\A_f)$ be an open compact subgroup. We assume $N(\Zp) \subset K_p \subset J_P$. For $\beta \in \Z_{\geq 0}$, define an open compact subgroup $L_\beta^P = L_p^{P,\beta} L^p \subset H(\A_f)$ by setting:
		\begin{itemize}\setlength{\itemsep}{0pt}
			\item[(i)]  $L_p^{P,\beta} \defeq H(\Zp) \cap K_p \cap (u^{-1} t_P^{\beta}) K_p (u^{-1} t_P^\beta)^{-1},$ and
			\item[(ii)]$L^{p} \defeq \{h \in H(\widehat{\Z}^{(p)}) : h \equiv 1 \newmod{m}\}$, the principal congruence subgroup of level $m$.
		\end{itemize}
		The \emph{automorphic cycle of level $L_\beta^P$} is
		\[
		X_{\beta}^P \defeq H(\Q)\backslash H(\A)/L_\beta L_\infty^\circ,
		\]
		where $L_\infty = H_\infty\cap K_\infty$ for $H_\infty = H(\R)$. This is a real orbifold of dimension $t$ \cite[(23)]{DJR18}.
	\end{definition}
	
	We will always take $m$ to be the smallest positive integer such that $L^p \subset K^p \cap H(\A_f)$ and $H(\Q) \cap hL_\beta^P L^\circ_\infty h^{-1} = Z_G(\Q) \cap L_\beta^P L_\infty^\circ$ for all $h \in H(\A)$ and for both $P = B,Q$ (compare \cite[(4.1),(4.2)]{BDW20}). This means $X_\beta^P$ is a real manifold \cite[(21)]{DJR18}. The impact of changing $m$ is discussed in \cite[\S4.1]{BDW20}.
	
	\begin{lemma}\label{lem:volume}
		We have 
\[
	\vol(L_p^{P,\beta}) = \delta_B(t_P^\beta) \cdot A_P,
\]
where $A_P = \delta_B(t_P^{-1})\vol(L_p^{P,1})$ is a constant independent of $\beta$.
	\end{lemma}
\begin{proof}
Let $N \subset G$ be the upper unipotent subgroup, and let $N^\beta \defeq t_P^\beta N(\Zp) t_P^{-\beta} \subset N(\Zp)$. By \cite[Lem.\ 4.4.1]{loeffler-parabolics}, for $\beta \geq 1$ we have
\[
[ L_p^{P,\beta} : L_p^{P,\beta+1}] = [N^\beta : N^{\beta+1}] = \delta_B(t_P^{-1}),
\]
the second equality being by definition of the modulus character $\delta_B$. It follows that $\vol(L_p^{P,\beta}) = \delta_B(t_P^{\beta-1}) \vol(L_p^{P,1})$, from which the result follows.
\end{proof}
	
	\begin{lemma}\label{lem:det 1 mod beta}
		If $(\ell_1,\ell_2) \in L_p^{P,\beta}$, then $\ell_2 \equiv w_n\ell_1w_n \newmod{p^\beta}$. Hence there is an isomorphism 
		\[
		\det(L_p^{P,\beta}) \isorightarrow (1+p^\beta\Zp) \times \Zp^\times, \hspace{12pt} (x,y) \mapsto (xy^{-1},y).
		\]
	\end{lemma}
	\begin{proof}
		Similar to \cite[Lem.\ 2.1]{DJR18}. First, compute that for $\smallmatrd{\ell_1}{}{}{\ell_2} \in L_\beta^{P,\beta}$, we need
		\[
		t_P^{-\beta}u  \smallmatrd{\ell_1}{}{}{\ell_2} u^{-1}t_P^\beta = t_P^{-\beta}  \smallmatrd{\ell_1}{w_n(\ell_2 - w_n\ell_1w_n)}{}{\ell_2} t_P^\beta \in K_p.
		\] 
		Since $P \subset Q$, each of the first $n$ diagonal entries of $t_P^\beta$ is congruent to $0 \newmod{p^\beta}$. In particular, after expanding we see $p^{-\beta}(\ell_2 - w_n\ell_1w_n) \in \GL_n(\Zp)$, so  $\ell_2 \equiv w_n\ell_1w_n \newmod{p^\beta}$, giving the first statement. We then have $\det(\ell_2) \equiv \det(\ell_1) \newmod{p^\beta}$, so to see the isomorphism, it suffices to prove surjectivity. But given $(a,b) \in (1+p^\beta\Zp) \times \Zp^\times$, we see $\ell_1 = \smallmatrd{1_{n-1}}{}{}{ab}, \ell_2 = \smallmatrd{b}{}{}{1_{n-1}}$ works (for any $P$).
	\end{proof}
	
	\begin{corollary}\label{cor:B inside Q}
 We have $L_p^{P,\beta} \subset L_p^{Q,\beta}$.
	\end{corollary}
	\begin{proof}
		By the proof of Lemma \ref{lem:det 1 mod beta} (or \cite[Lem.\ 2.1]{DJR18}), we deduce $L_p^{Q,\beta} = \{(\ell_1,\ell_2) \in K_p : \ell_2 \equiv w_n\ell_1w_n \newmod{p^\beta}\}$. But by Lemma \ref{lem:det 1 mod beta}, any element of $L_p^{P,\beta}$ satisfies this.
	\end{proof}
	
	By Lemma \ref{lem:det 1 mod beta} and strong approximation for $H$, via the map 
	\[
		(h_1,h_2) \mapsto (\det(h_1)/\det(h_2), \det(h_2))
	\]
	 the cycle $X_\beta^P$ decomposes into connected components indexed by 
	\begin{equation}\label{eq:component group}
		\pi_0(X_\beta^P) \defeq \Cl(p^\beta m) \times \Cl( m) \cong (\Z/p^\beta m\Z)^\times \times (\Z/m\Z)^\times
	\end{equation}
	(cf.\ \cite[(22)]{DJR18}). Here for an ideal $I \subset \Z$, we let $\sU(I) \defeq \big\{x \in \widehat{\Z}^\times: x \equiv 1 \newmod{I}\big\} \subset \widehat{\Z}^\times$ and let
	\begin{equation}\label{eq:ray class group}
		\Cl(I) = \Q^\times\backslash \A^\times/\sU(I)\R_{>0} \cong (\Z/I)^\times \hspace{12pt}
	\end{equation}
	be the narrow ray class group of conductor $I$. For $\delta \in H(\A_f)$, we write $[\delta]$ for its associated class in $\pi_0(X_\beta^P)$ and denote the corresponding connected component
	\[
	X_\beta^P[\delta] \defeq H(\Q)\backslash H(\Q)\delta L_\beta^P H_\infty^\circ/L_\beta^P L_\infty^\circ.
	\]

	As $L^p \subset K^p\cap H(\A_f)$, by definition of $L_p^{P,\beta}$ there is a proper map (see \cite[Lemma 2.7]{Ash80})
	\begin{equation}\label{eq:iota beta}
		\iota_\beta^P : X_\beta^P \longrightarrow S_K,\ \ \ \ \
		[h] \longmapsto [\iota(h)u^{-1} t_P^\beta].
	\end{equation}

	%%==========================================================
	%%
	%%			Abstract evaluation maps
	%%
	%%==========================================================
	\subsubsection{Abstract evaluation maps} \label{sec:abstract evaluations}
	Define $\Delta_P \subset \GL_{2n}(\Qp)$ to be the semigroup generated by the parahoric subgroup $J_P \subset \GL_{2n}(\Zp) $ and the matrices
	\[
		t_{p,m_1},\ \  t_{p,m_1+m_2}, \ \ ...,\ \  t_{p,m_1 + \cdots + m_{r-1}},
	\]
	 where $P$ has Levi $\GL_{m_1} \times \cdots \times \GL_{m_r}$ (and recalling $t_{p,r}$ from \eqref{eq:t_p i}). For example:
	\begin{itemize}\setlength{\itemsep}{0pt}
		\item[--] $J_B = \Iw$, and $\Delta_B$ is generated by $\Iw$ and $t_{p,1}, t_{p,2}, ..., t_{p,2n-1}$.
		\item[--] $J_Q = \{g \in \GL_{2n}(\Zp): g \newmod{p} \in Q(\F_p)\}$ and $\Delta_Q$ is generated by $J_Q$ and $t_{p,n}$.
	\end{itemize}

	Let $K \subset G(\A_f)$ be an open compact subgroup such that $N_Q(\Zp) \subset K_{p} \subset J_P$. Let $M$ be a left $\Delta_P$-module, with action denoted $*$. Then $K$ acts on $M$ via its projection to $K_p \subset \Delta_P$, giving a local system $\sM$ on $S_K$ via \S\ref{sec:local systems}. The notation is suggestive: as in \S\ref{sec:integral normalisations cohomology}, using this $*$-action in \eqref{eq:Hecke composition} we get `integrally normalised' Hecke operators $U_{p,r}^\circ$ on the cohomology $\hc{t}(S_K,\sM)$.
	
	The constructions here are almost identical to those of \cite[\S4.2]{BDW20} where they are motivated and explained in great detail; thus we give only the briefest description here.
	
	For $\beta \in \Z_{\geq 0}$ and $\delta \in H(\A_f)$, define a congruence subgroup
	\begin{equation}\label{eq:gamma beta delta}
		\Gamma_{\beta,\delta}^P \defeq H(\Q) \cap \delta L_\beta^P H_\infty^\circ \delta^{-1}.
	\end{equation} 
	This acts on $M$ via 
	\begin{equation}\label{eq:gamma action}
		\gamma *_{\Gamma_{\beta,\delta}^P}m \defeq (\delta^{-1}\gamma\delta)_f* m.
	\end{equation}
	Let $M_{\Gamma_{\beta,\delta}^P} \defeq M/\{m - \gamma *_{\Gamma_{\beta,\delta}^P} m : m \in M, \gamma \in \Gamma_{\beta,\delta}^P\}$ be the coinvariants of $M$ by $\Gamma_{\beta,\delta}^P$. 
	
	\begin{definition}\label{def:evaluation map}
		The \emph{evaluation map for $M$ and $P$ of level $p^\beta$ at $\delta$} is the composition
		\begin{align}\label{eq:evaluation definition}
			\mathrm{Ev}_{P, \beta,\delta}^{M}   : \htc(S_K,\sM)  \xrightarrow{\ \tau_{\beta}^{P,\circ} \circ (\iota_\beta^P)^*\ } \hc{t}(X_\beta^P,&\iota^*\sM) \xrightarrow{c_\delta^*} \hc{t}(\Gamma_{\beta,\delta}^P\backslash \cX_H, c_\delta^*\iota^*\sM)\\
			& \xrightarrow{\coinv_{\beta,\delta}^P}  \hc{t}(\Gamma_{\beta,\delta}^P\backslash\cX_H, \Z) \otimes M_{\Gamma_{\beta,\delta}^P} \labelisorightarrow{ - \cap \theta_\delta^P} M_{\Gamma_{\beta,\delta}^P}, \notag
		\end{align}
		where:
		\begin{itemize}\setlength{\itemsep}{0pt}
			\item $\iota_\beta^P$ is the map from \eqref{eq:iota beta}, and $\tau_\beta^{P,\circ}$ is the map $(\iota_\beta^{P})^*\sM \to \iota^*\sM$ of local systems on $X_\beta^P$ induced by $(h,m) \mapsto (h,u^{-1} t_P^\beta * m)$;
			
			\item $\Gamma_{\beta,\delta}^P$  acts on $\cX_H \defeq  H_\infty^\circ/L_\infty^\circ$ by left translation, and there is an isomorphism
			\[
			c_\delta : \Gamma_{\beta,\delta}^P\backslash \cX_H \isorightarrow X_\beta^P[\delta] \subset X_\beta, \ \ \ [h_\infty]_\delta \mapsto [\delta h_\infty],
			\]
			where if $[h_\infty] \in \cX_H$, we write $[h_\infty]_\delta$ for its image in $\Gamma_{\beta,\delta}^P\backslash \cX_H$;
			
			\item $\mathrm{coinv}_{\beta,\delta}^P$ is the quotient map $M \twoheadrightarrow M_{\Gamma_{\beta,\delta}^P}$, which induces a map on cohomology with image in the cohomology of the trivial local system on $\Gamma_{\beta,\delta}^P\backslash \cX_H$ attached to $M_{\Gamma_{\beta,\delta}^P}$;
			
			\item and $(-\cap \theta_\delta^P)$ is induced from cap product  $(-\cap\theta_\delta^P) : \hc{t}(\Gamma_{\beta,\delta}^P\backslash \cX_H, \Z) \isorightarrow \Z$, for $\theta_\delta^P$ a fundamental class in the Borel--Moore homology $\h^{\mathrm{BM}}_t(\Gamma_{\beta,\delta}^P\backslash\cX_H,\Z) \cong \Z$.
		\end{itemize}
	\end{definition}
	
	We choose the classes $\theta_\delta^P$ compatibly in $\delta$ and $P$. Let $\theta_\delta^Q$ be exactly as in \cite[\S4.2.3]{BDW20}. We have a natural map $\mathrm{pr}_Q^B : \Gamma_{\beta,\delta}^B\backslash\cX_H \to \Gamma_{\beta,\delta}^Q\backslash\cX_H$; we let $\theta_\delta^B = (\mathrm{pr}_Q^B)^*\theta_\delta^Q$.

	Exactly as in \cite[\S4.3]{BDW20}, we can track dependence of these maps as we allow $M$, $\beta$ and $\delta$ to vary. Each of the following results is proved exactly as their given counterpart \emph{op.\ cit}.:
	
	\begin{lemma}(Variation in $M$) \label{lem:pushforward}
		Let $\kappa : M \rightarrow N$ be a $\Delta_P$-module map. There is a commutative diagram
		\[
		\xymatrix@C=18mm@R=6mm{
			\hc{t}(S_K, \sM) \ar[r]^-{\mathrm{Ev}_{P,\beta,\delta}^M} \ar[d]^-{\kappa_*} & M_{\Gamma_{\beta,\delta}^P}\ar[d]^-{\kappa}\\
			\hc{t}(S_K,\sN) \ar[r]^-{\mathrm{Ev}_{P,\beta,\delta}^N} & N_{\Gamma_{\beta,\delta}^P}.
		}
		\]
	\end{lemma}

	\begin{proposition}(Variation in $\delta$) \label{prop:ind of delta}
		Let $N$ be a left $H(\A)$-module, with action denoted $*$, such that $H(\Q)$ and $H_\infty^\circ$ act trivially. Let $\kappa : M\to N$ be a map of $L_\beta^P$-modules (with $N$ an $L_\beta^P$-module by restriction). Then
		\[
		\mathrm{Ev}_{P,\beta,[\delta]}^{M,\kappa} \defeq \delta * \left[\kappa \circ \mathrm{Ev}_{P,\beta,\delta}^M\right] : \hc{t}(S_K,\sM) \longrightarrow N
		\]
		is well-defined and independent of the representative $\delta$ of $[\delta]$.
	\end{proposition}
	
	To vary $\beta$, we have a natural projection $\mathrm{pr}_{\beta} : X_{\beta+1}^P \longrightarrow X_\beta^P,$ inducing a projection $\mathrm{pr}_{\beta} : \pi_0(X_{\beta+1}^P) \rightarrow \pi_0(X_\beta^P)$. The action of $t_P$ on $M$ yields an action of $U_{P}^\circ$ on $\hc{t}(S_K,\sM)$, where $U_B^\circ = U_p^\circ$ and $U_Q^\circ = U_{p,n}^\circ$.
	
	For compatibility in $\beta$, we need to assume additionally that $K_p = J_P$ is the parahoric for $P$.
	
	\begin{proposition}(Variation in $\beta$) \label{prop:evaluations changing beta} Let $N$ and $\kappa$ be as in Proposition~\ref{prop:ind of delta}. If $\beta > 0$, then as maps $\hc{t}(S_K,\sM) \to N$ we have
		\[
		\sum_{[\eta] \in \mathrm{pr}_{\beta}^{-1}([\delta])} \mathrm{Ev}_{P,\beta+1,[\eta]}^{M,\kappa}   =
		\mathrm{Ev}_{P,\beta,[\delta]}^{M,\kappa} \circ U_P^\circ.
		\]
	\end{proposition}
	
	\begin{proof}
		The proof follows almost exactly as in \cite{BDW20}. There is a unique point at which more detail is required. The left-hand square of diagram (4.14) \emph{op.\ cit}.\ generalises to
		\[
			\xymatrix{
				S_K & S_{K^0_P(p)} \ar[l] \\
				X_\beta^P \ar[u]^{\iota_\beta^P} & X_{\beta+1}^P \ar[l] \ar[u]^{\iota_\beta^{P,0}}
			},
		\] 
		where $K^0_P(p) = K\cap t_PKt_P^{-1}$, the map $\iota_\beta^{P,0}$ is induced by the map $[h] \mapsto [\iota(h)u^{-1}t_P^\beta],$ and where the horizontal maps are the natural projections. We need to show that this square is Cartesian in this generality. For this, since the vertical maps are embeddings, it is enough to show that the horizontal maps have the same degree; i.e., that $[K : K^0_P(p)] = [L_p^{P,\beta} : L_p^{P,\beta+1}]$ for any $\beta \geq 1$. Let $N_P$ denote the unipotent radical of $P$ and set $N_{\beta} = t_P^\beta N_P(\Z_p) t_P^{-\beta}$. By the Iwahori decomposition for $K_p$, we easily find that $[K : K^0_P(p)] = [N_P(\Z_p) : N_1] = [N_{\beta} : N_{\beta+1}]$ for any $\beta \geq 1$. On the other hand, the element $u^{-1}$ is a representative of the (unique) Zariski dense $H$-orbit in the flag variety $G/\overline{P}$, where $\overline{P}$ denotes the opposite of $P$. Therefore the proof of \cite[Lem.\ 4.4.1]{loeffler-parabolics} implies that $[L_p^{P,\beta} : L_p^{P,\beta + 1}] = [N_{\beta} : N_{\beta+1}]$. Hence we have 
	\[
	[K : K^0_P(p)] = [L_p^{P,\beta} : L_p^{P,\beta + 1}]
	\]
	as required.
	\end{proof}

	%%==========================================================
	%%
	%%			Evaluation maps from DJR
	%%
	%%==========================================================
	\subsection{Classical evaluation maps, Shalika models and $L$-values}\label{sec:classical evaluations}
	The classical evaluation maps $\cE_{\beta,\delta}^{j,\sw}$ of \cite{DJR18} were rephrased in the abstract language of Definition \ref{def:evaluation map} in \cite[\S5]{BDW20}.  We recap the construction, whilst again generalising it to parahoric level for a general parabolic $P$. When $P = Q$ this recovers  \cite{DJR18,BDW20}; in this paper we are primarily interested in $P = B$ (which is new). We relate the values of these evaluation maps to critical $L$-values. Throughout, we assume $K_p = J_P$.
	
	The definition of $\cE_{\beta,\delta}^{j,\sw}$ fundamentally used the following branching law from \cite[Prop.\ 6.3.1]{GR2} and \cite[Lem.\ 5.2]{BDW20}. Let $\lambda \in X_0^*(T)$ be a pure algebraic weight, with purity weight $\sw$. For integers $j_1,j_2$, let $V_{(j_1,j_2)}^H$ denote the 1-dimensional $H(\Zp)$-representation given by the character
	\[
	H(\Zp) \longrightarrow \Zp^\times,\ \ \ \ \ \ (h_1,h_2) \longmapsto \det(h_1)^{j_1}\det(h_2)^{j_2}.
	\]
	\begin{lemma}\label{lem:branching law 2}
		Let $j \in \Z$. Then $j \in \mathrm{Crit}(\lambda)$ if and only if $\mathrm{dim}\  \mathrm{Hom}_{H(\Zp)}\big(V_\lambda^\vee, V^H_{(j,-\sw-j)}\big) = 1$.
	\end{lemma}
	
	For each $j \in \mathrm{Crit}(\lambda)$, fix some choice of non-trivial $H(\Zp)$-map $\kappa_{\lambda,j} : V_\lambda^\vee(L) \to V_{(j,-\sw-j)}^H \cong L$. We will make more precise choices in \S\ref{sec:branching laws}, but for now they can be arbitrary.
	
	The \emph{$p$-adic cyclotomic character} is
	\begin{equation}\label{eq:cyc}
		\chi_{\cyc} : \Q^\times\backslash \A^\times \longrightarrow \Zp^\times, \hspace{12pt}	y \mapsto \mathrm{sgn}(y_\infty) \cdot |y_f| \cdot y_p.
	\end{equation}	
	It is the $p$-adic character associated to the adelic norm \cite[\S2.2.2]{BW_CJM}. It is trivial on $\R_{>0}$.
	
	It is simple to see that the $H(\Zp)$-representation $V_{(j_1,j_2)}^H$ extends to $H(\A)$ via the character
	\[
	H(\A) \longrightarrow \Zp^\times, \hspace{12pt} (h_1,h_2) \mapsto \chi_{\cyc}\big[\det(h_1)^{j_1}\det(h_2)^{j_2}\big].
	\]
	Note the action of $L_\beta \subset H(\A_f)$ factors through projection to $H(\Zp)$, so the map $\kappa_{\lambda,j} : V_{\lambda}^\vee(L) \to L$ chosen above is a map of $L_\beta$-modules. Moreover, $H(\Q)$ and $H_\infty^\circ$ act trivially on $V_{(j_1,j_2)}^H$, so we can use the formalism of Proposition \ref{prop:ind of delta}.

	\begin{definition}\label{def:classical evalution delta}
		Let $L/\Qp$ be an extension. The \emph{classical evaluation map for $P$ of level $p^\beta$ at $\delta$} is 
		\[
		\cE_{P,\beta,[\delta]}^{j,\sw} \defeq \mathrm{Ev}_{P,\beta,[\delta]}^{V_{\lambda}^\vee, \kappa_{\lambda,j}} : \hc{t}(S_K,V_\lambda^\vee(L)) \longrightarrow L.
		\]
	\end{definition}
	
	Here $\mathrm{Ev}_{P,\beta,[\delta]}^{V_{\lambda}^\vee, \kappa_{\lambda,j}}$ was defined in Proposition \ref{prop:ind of delta}, which shows $\cE_{P,\beta,[\delta]}^{j,\sw}$ is independent of the choice of $\delta$ representing the class $[\delta]$. We introduce the notation $\cE_{P,\beta,[\delta]}^{j,\sw}$ for consistency with \cite{DJR18,BDW20}; via \cite[Lem.\ 5.3]{BDW20} this definition is consistent with that in \cite{DJR18}.

	%%=====================================
	%%		 L-VALUE FORMULA
	%%=====================================

	Recall $\pi_0(X_\beta^P) = (\Z/p^\beta m)^\times \times (\Z/m)^\times$ from \eqref{eq:component group}. Write $\mathrm{pr}_1, \mathrm{pr}_2$ for the projections of $\pi_0(X_\beta^P)$ onto the first and second factors respectively, and let $\mathrm{pr}_{\beta}$ denote the natural composition
	\begin{equation}\label{eq:pr_beta}
		\mathrm{pr}_{\beta} : (\Z/p^\beta m)^\times \times (\Z/m)^\times \xrightarrow{ \ \mathrm{pr}_1 \ } (\Z/p^\beta m)^\times  \twoheadrightarrow (\Z/p^\beta)^\times.
	\end{equation}

	\begin{definition}\label{def:ev chi}
		For $\eta_0$ be any character of $(\Z/m)^\times$, and $\brep\in (\Z/p^\beta)^\times$, define 
		\[
		\cE_{P,\beta,\brep}^{j,\eta_0} \defeq \sum_{[\delta] \in \mathrm{pr}_\beta^{-1}(\brep)} \eta_0^{-1}\big(\mathrm{pr}_2([\delta])\big)\ \cE_{P,\beta,[\delta]}^{j,\sw} \ \ \ : \hc{t}(S_K,\sV_\lambda^\vee(L)) \to L.
		\]
		(In our main application, we will take $\eta_0$ trivial; but the obstructions to taking more general $\eta_0$ are automorphic, not $p$-adic, so we develop the theory in full generality here).
		
		Let $\chi$ be a finite order Hecke character of conductor $p^{\beta'}$, let $\beta = \mathrm{max}(1,\beta')$ and let $L(\chi)$ be the smallest extension of $L$ containing $\mathrm{Im}(\chi)$. For $j\in \mathrm{Crit}(\lambda)$, define 
		\begin{align}\label{eq:classical evaluation}
			\cE_{P,\chi}^{j,\eta_0} = \sum_{[\brep] \in (\Z/p^\beta)^\times}& \chi(\brep)\ \cE_{P,\beta,\brep}^{j,\eta_0}\ :\ \hc{t}(S_K, \sV_\lambda^\vee(L)) \longrightarrow L(\chi),\\
			\phi &\longmapsto \sum_{[\delta] \in (\Z/p^\beta m)^\times \times (\Z/m)^\times} \chi\big(\mathrm{pr}_\beta([\delta])\big)\cdot  \eta_0^{-1}\big(\mathrm{pr}_2([\delta])\big)\cdot\left( \delta*\left[\kappaj \circ \mathrm{Ev}_{P,\beta,\delta}^{V_{\lambda}^\vee}(\phi)\right]\right).\notag
		\end{align}
	\end{definition}

	\begin{remark}\label{rem:classical diagram}
		We see  $\cE_{P,\chi}^{j,\eta_0}$ is the composition
		\begin{equation}\label{eq:explicit classical}
			\xymatrix@R=10mm@C=5mm{
				\hc{t}(S_K,\sV_\lambda^\vee) \ar@/^3pc/[rrrrr]_-{\oplus \cE_{P,\beta,[\delta]}^{j,\sw}} \ar@/_3pc/[rrrrrrr]^-{\oplus \cE_{P,\beta,\brep}^{j,\eta_0}} \ar[rr]^-{\oplus\mathrm{Ev}_{P,\beta,\delta}^{V_\lambda^\vee}} && \displaystyle\bigoplus_{[\delta]}(V_\lambda^\vee)_{\Gamma_{\beta,\delta}^P} \ar[rrr]^-{v \mapsto \delta * \kappaj(v)} &&&
				\displaystyle\bigoplus_{[\delta]} L \ar[rr]^-{\oplus\Xi_{\brep}^{\eta_0}} && 
				\displaystyle\bigoplus_{\brep} L \ar[rrr]^-{\ell \mapsto \Sigma \chi(\brep)\ell_\brep} &&&
				L(\chi),
			}
		\end{equation}
		where the sums are over $[\delta] \in (\Z/p^\beta m)^\times \times (\Z/m)^\times$ or $\brep\in (\Z/p^\beta)^\times$, related by $\brep = \mathrm{pr}_\beta([\delta])$, and $\Xi_{\brep}^{\eta_0}$ is the $\eta_0$-averaging map 
		\[
		\Xi_{\brep}^{\eta_0} : (m_{[\delta]})_{[\delta]} \longmapsto \sum_{[\delta] \in \mathrm{pr}_{\beta}^{-1}(\brep)} \eta_0^{-1}(\mathrm{pr}_2([\delta])) \cdot m_{[\delta]}.
		\]
	\end{remark}

We give two applications of these maps. Let $\pi$ be any RACAR of weight $\lambda$ with attached maximal ideal $\m_\pi \subset \cH'$ as in \S\ref{sec:hecke algebra}. 
	
	Firstly, classical evaluation maps can detect existence of Shalika models:
	
	\begin{proposition}\label{prop:shalika non-vanishing}
		Suppose there exists $\phi \in \hc{t}(S_K,\sV_{\lambda}^\vee(\overline{\Q}_p))^\pm_{\pi}$ and some  $\chi, j$ and $\eta_0$ such that
		\begin{equation}\label{eq:non-vanishing 1}
			\cE_{P,\chi}^{j,\eta_0}(\phi) \neq 0.
		\end{equation}
		Then $\pi$ admits a global $(\eta_0|\cdot|^{\sw},\psi)$-Shalika model, where $\sw$ is the purity weight of $\lambda$.
	\end{proposition}
	\begin{proof}
		This is proved exactly as in \cite[\S5.3]{BDW20}. Whilst $\xi$ is replaced by $u^{-1}$, and some non-zero volume factors change depending on $P$, the argument of proof is identical.
	\end{proof}
	
	Secondly, we generalise \cite[\S4]{DJR18}, and show that up to a local zeta factor at $p$, these maps compute $L$-values. Let $K = J_P \prod_{\ell\neq p}\GL_{2n}(\Z_\ell)$, and let $W = W_p \otimes \bigotimes_{\ell\neq p}W_\ell^\circ \in \cS_{\psi_f}^{\eta_f}(\pi_f^K, E)$, where we choose the normalised spherical vector at each $\ell \neq p$, and where $W_p$ is an arbitrary $E$-rational vector in $\pi_p^{J_P}$. Recall the map $\Theta^\pm$ from \eqref{eq:Theta}, and the Friedberg--Jacquet integral $\zeta(s,W,\chi)$ from \S\ref{sec:friedberg-jacquet}.
	
	\begin{theorem}\label{thm:critical value}
		Let $\chi$ be a finite order Hecke character of conductor $p^{\beta'}$, let $\beta = \mathrm{max}(1,\beta')$, and let $j \in \mathrm{Crit}(\lambda)$. If $(-1)^j\chi_\infty\eta_\infty(-1) = \pm 1$, then
		\begin{align*}
		\cE_{P,\chi}^{j,\eta_0}\left(\frac{\Theta^\pm(W)}{\Omega_\pi^\pm}\right) = \delta_B(t_P^{-\beta})&\Upsilon_{P}\cdot \lambda(t_{P}^\beta)\cdot
		\zetainfty \\
&\times \frac{L^{(p)}(\pi\otimes\chi,j+1/2)}{\Omega_\pi^\pm} \cdot \zeta_p\Big(j+\tfrac12, (u^{-1}t_{P}^\beta)\cdot W_p, \chi_p\Big).
		\end{align*}
	 If $(-1)^j\chi_\infty\eta_\infty(-1) = \mp 1$, then $\cE_{P,\chi}^{j,\eta_0}(\Theta^\pm(W)/\Omega_\pi^\pm) = 0$. Here all the signs are chosen consistently to be either the top or bottom sign.
	\end{theorem} 
	
	Here we recall $u$ from \eqref{eq:u} and $t_P$ from Definition \ref{def:parabolic}, and:
	\begin{itemize}
		\item $\Upsilon_{P}$ is a non-zero rational volume constant independent of $\chi$ and $j$; we have 
\[
\Upsilon_P = \gamma \cdot A_P^{-1} \cdot p^{n^2}\cdot [\#\GL_n(\Z/p\Z)]^{-1},
\] 
where $\gamma$ is the constant from \cite[(77)]{DJR18} and $A_P$ is the constant from Lemma \ref{lem:volume}. (Note that when $P = Q$ has Levi $\GL_n \times \GL_n$, we have $\Upsilon_Q = \gamma$). 
\item $\zetainfty$ is an archimedean zeta integral that depends linearly on the choice of branching law $\kappa_{\lambda,j}$; this factor is non-zero by \cite[Thm.~5.5]{Sun19}.

\item $L^{(p)}(-)$ is the $L$-function with the local factor at $p$ removed (that is, the product of all the local factors for all finite $\ell \neq p$).
\end{itemize}
	
	\begin{proof}
		A rephrasing of \cite[Prop.\ 4.6, Thm.\ 4.7]{DJR18} in this language is described in \cite[\S5.5]{BDW20}. As in \cite[Prop.\ 4.6]{DJR18}, one first  writes $\cE_{\chi}^{j,\eta_0}(-)$ as an explicit integral over $X_\beta^P$, introducing the factor $\lambda(t_{P}^\beta)$. One lifts this to an integral over $Z_G(\A)H(\Q)\backslash H(\A)$, introducing the volume constant $\delta_B(t_P^{-\beta})\Upsilon_{P}$ (as we divide by $\mathrm{vol}(L_\beta^P)$, using Lemma \ref{lem:volume}). By \cite[Prop.\ 2.3]{FJ93} the integral equals a global Friedberg--Jacquet integral, which breaks into a product of local integrals. Away from $p$, the computation of these local zeta integrals is literally identical to that \emph{op.\ cit}.; and at $p$, by definition the zeta integral is the one in the statement of the theorem (with $t_P$ replacing $t_Q$).
	\end{proof}
	
	We will evaluate the local zeta integral at $p$ for $P=B$ and specific choices of $W_p$ in \S\ref{sec:local zeta p}--\ref{sec:local zeta unramified}, and the integral at infinity for specific choices of branching law in Theorem \ref{thm:non-ordinary}.

	\part{Local Theory: Shalika $p$-refinements}

	For the remainder of the paper, unless otherwise specified we specialise to $P = B$ and consider Iwahori level. 
	
		Let us summarise what we have done so far. We took $\pi$ to be a RASCAR that is everywhere spherical.  A \emph{$p$-refinement} of $\pi$ was a choice of Hecke eigenspace $\tilde\pi_p$ in $\pi_p^{\Iw}$. To any choice of $W_p \in \cS(\tilde\pi_p)$, before Theorem \ref{thm:critical value} we associated a (global) cohomology class in $\hc{t}(S_K,\sV_\lambda^\vee)$. In that theorem, we computed its image under a scalar-valued functional, and showed that it took the form 
		\[
			\Big[\text{non-zero scalar}\Big] \times \Big[\text{critical $L$-value for $\pi$}\Big] \times \zeta_p\Big(j+\tfrac{1}{2}, (u^{-1}t_{p}^\beta)\cdot W_p, \chi_p\Big).
		\] 
	Over the next few sections, we compute the third term in this product -- the local zeta integral at $p$ -- for `nice' choices of $W_p$. This is a significant computation, spanning over several sections, so we briefly sketch the steps.
	\begin{itemize}\s
	\item	In \S\ref{sec:local zeta p}, when $\chi_p$ is ramified, we compute $\zeta_p(s,(u^{-1}t_p^\beta)\cdot W_p, \chi_p)$ as an explicit non-zero multiple of a specific value of $W_p$ (depending on $\beta$).
	
	\item We are interested in finding $p$-refinements containing Hecke eigenvectors $W_p$ for which this value is non-zero. We call such $p$-refinements \emph{Shalika $p$-refinements}.
	
	\item In \S\ref{sec:spin refinements}, we begin a systematic combinatorial study of $p$-refinements, and introduce `spin $p$-refinements', a class of $p$-refinements for $\GL_{2n}$ that `come from $\mathrm{GSpin}_{2n+1}$'.
	
	\item In \S\ref{sec:shalika refinements}, we write down explicit eigenvectors attached to spin $p$-refinements. We precisely evaluate relevant values of these eigenvectors, and thus deduce that spin $p$-refinements are Shalika $p$-refinements. (In fact, we expect that the converse is true as well; we hope to return to this in a sequel to this paper). 
	
	\item In \S\ref{sec:running assumptions}, we summarise all of the above, and fold the local theory back into the global results of Part I.
	
	\item Our above computations worked with Iwahori-invariant $W_p$, but required $\chi_p$ to be ramified. Finally, in \S\ref{sec:local zeta unramified}, we compute $\zeta_p(s,(u^{-1}t_p^\beta)\cdot W_p, \chi_p)$ for arbitrary $\chi_p$, in particular allowing $\chi_p$ unramified. In this section, we use different methods and instead assume $W_p$ is invariant for the parahoric subgroup of type $(n,n)$.

	\end{itemize}
	
		\begin{notation*}
		Since it will be entirely focused on local theory, in Part II we henceforth drop subscripts $p$. In particular, we let $\pi$ be a generic unramified principal series representation of $\GL_{2n}(\Qp)$ admitting an $(\eta,\psi)$-Shalika model, for \label{shalika character} $\eta : \Qp^\times \to \C^\times$ a smooth character and $\psi : \Qp \to \C^\times$ the usual additive character (e.g.\ \cite[\S4.1]{DJR18}). Note $\pi$ is spherical. We continue to write $\Iw$ for the Iwahori subgroup of $\GL_{2n}(\Zp)$. We write $\zeta(-)$ in place of $\zeta_p(-)$. We keep the notation of \S\ref{sec:hecke algebra}, with matrices $t_{p,r}$ and Hecke operators $U_{p,r}$ on $\pi^{\Iw}$. A $p$-refinement $\tilde\pi = (\pi,\alpha)$ is a choice of Hecke eigensystem $\alpha$ occurring in $\pi^{\Iw}$.
	\end{notation*}

	\section{The local zeta integral at Iwahori level}\label{sec:local zeta p}
	
	We now give our first reduction of the local zeta integral 
	\begin{equation}\label{eq:local zeta p}
		\zeta(s, (u^{-1}t_{p}^\beta)\cdot W, \chi) = \int_{\GL_n(\Qp)} W\left[\matrd{x}{}{}{1}u^{-1}t_p^\beta\right] \ \chi(\det(x)) \ |\det(x)|^{s-\tfrac{1}{2}} \ dx.
	\end{equation}
	 The main aim of this section is Proposition \ref{prop:local zeta p}, which computes this in terms of a specific value of $W$. First, we reduce the support of the zeta integral:

	\begin{lemma} \label{SupportLemma}
		Suppose $W \in \mathcal{S}^{\eta}_{\psi}(\pi)$ is fixed under the action of $\operatorname{Iw}_G$. Then the function
		\[
		\GL_n(\Q_p) \to \C,\qquad 
		x \mapsto W\smallmatrd{x}{}{}{1}
		\]
		is supported on $M_{n}(\Z_p) \cap \GL_n(\Q_p)$.
	\end{lemma}
	\begin{proof}
		For $y \in M_{n}(\Z_p)$, right translation by $\smallmatrd{1}{y}{}{1} \in \operatorname{Iw}_G$ gives
		\[
		W \smallmatrd{x}{}{}{1} = W\left[\smallmatrd{x}{}{}{1}\smallmatrd{1}{y}{}{1} \right] = \psi(\mathrm{tr}(x  y)) W\smallmatrd{x}{}{}{1}.
		\]
		If $x \not\in  \GL_n(\Q_p)\backslash M_{n}(\Z_p)$, then we can choose $y$ such that $\mathrm{tr}(xy) \not\in \Z_p$, so $\psi(\mathrm{tr}(xy)) \neq 1$.
	\end{proof}

	Let $d^\times c$ be the Haar measure on $\Zp^\times$ with total measure 1. Let $\chi : \Qp^\times \to \C^\times$ be a finite order character of conductor $p^\beta$. In practice $\chi$ will be the local component at $p$ of a Hecke character of $p$-power conductor, which forces $\chi(p) = 1$; we thus impose this condition throughout. For such a character, denote its Gauss sum by
	\begin{equation}\label{eq:gauss sum}
	\tau(\chi) = p^{\beta}(1-p^{-1})\int_{\Z_p^{\times}} \chi(c) \psi(p^{-\beta} c) d^{\times} c.
	\end{equation}

	Recall $t_p = \mathrm{diag}(p^{2n-1},p^{2n-2},...,p,1)$. We write this in the form
	\[
	t_p = \matrd{p^nz}{}{}{z}, \qquad z \defeq \mathrm{diag}(p^{n-1},p^{n-2},...,p,1) \in T_n(\Qp).	
	\]
Let $\Iwn \subset G_n(\Z_p)$ denote the upper-triangular Iwahori subgroup of $G_n = \mathrm{GL}_n$.

	\begin{proposition}\label{prop:local zeta p}
		Let $\chi$ have conductor $p^\beta > 1$, and let $W \in \cS_{\psi}^{\eta}\left(\pi^{\Iw}\right)$. Then
		\[
	\zeta\Big(s,(u^{-1}t_p^{\beta}) \cdot W, \chi\Big) = \Upsilon' \cdot \eta(\det z^\beta) \cdot p^{-\beta \tfrac{n^2+n}{2}} \cdot p^{\beta n(s-1/2)} \cdot \tau(\chi)^n  \cdot \chi(\det(-w_n))  \cdot W\tbyt{w_n z^{2\beta}}{}{}{1},
		\]
		where $\Upsilon' = \mathrm{vol}(\Iwn) \cdot (1-p^{-1})^{-n} \cdot p^{(n^2-n)/2}$ is a scalar independent of $W, \chi$ and $\beta$.
	\end{proposition}

	\begin{proof}
		\begin{comment}
		Let $(e_{ij})_{i.j=1,\dots, n}$ denote the standard basis of $M_{n}(\Z_p)$. We will consider a certain collection of elements in $\operatorname{Iw}_G$ and use the invariance of $W$ under $\operatorname{Iw}_G$ to prove the proposition. 
		
		Let $1_n$ denote the $(n \times n)$ identity matrix. Then:
		\begin{itemize}\setlength{\itemsep}{0pt}
			\item For $c \in \Z_p$ and $1 \leq i \neq j \leq n$, let $\alpha_{i j}(c) \defeq 1_n + ce_{ij} \in \mathrm{SL}_n(\Z_p)$.
			\item For $c \in \Z_p^{\times}$ and $1 \leq i  \leq n$, let $\alpha_{i i}(c) \defeq 1_n + (c-1)e_{ii} \in T_n(\Z_p)$. 
		\end{itemize}
		If $i \leq j$, note $\smallmatrd{\alpha_{ij}(c)}{}{}{1_n} \in \operatorname{Iw}_G,$ so that 
		\begin{equation}\label{eq:local zeta unchanged}
			\zeta(s, [u^{-1}t_{p,n}^{\beta}]\cdot W, \chi ) = \zeta\left(s, \left[u^{-1}t_{p,n}^{\beta}\smallmatrd{\alpha_{ij}(c)}{}{}{1_n}\right]\cdot W, \chi \right).
		\end{equation} We use this equality to impose support conditions on $\zeta(-)$ that make it computable. 
	\end{comment}

	 Recall $u = \smallmatrd{1}{w_n}{0}{1}$ was defined in \eqref{eq:u}. Then observe that
	\[
		\matrd{x}{}{}{1}u^{-1}t_p^\beta = \matrd{z^\beta}{}{}{z^\beta}\matrd{1}{-z^{-\beta}xw_nz^\beta}{}{1}\matrd{p^{n\beta}z^{-\beta}xz^{\beta}}{}{}{1}.
	\]
	Substituting this into \eqref{eq:local zeta p}, and using the Shalika transformation property, we reduce to
	\begin{align*}
		\zeta(s,(u^{-1}t_p^\beta)&\cdot W, \chi) = \eta(\det z^\beta) \\
		\times &\int_{\GL_n(\Qp)} \psi\Big[-\mathrm{tr}(z^{-\beta}xw_nz^{\beta})\Big] W\matrd{p^{n\beta}z^{-\beta}xz^{\beta}}{}{}{1}\chi(\det x) |\det x|^{s-1/2}dx.
	\end{align*}
We make the change of variables $y = p^{n\beta}z^{-\beta}xz^{\beta}$. As $dx$ is a left and right Haar measure, we have $dy = dx$. Recalling that $\chi(p) = 1$ and $|p| = 1/p$, we get
\begin{align}\label{eq:zeta I}
	\zeta(s,(u^{-1}t_p^\beta)&\cdot W, \chi) = \eta (\det z^\beta) p^{n^2\beta(s-1/2)}\\
	&\times \int_{\GL_n(\Qp)} \psi\Big[-\mathrm{tr}(p^{-n\beta}z^{\beta}yz^{-\beta}w_n)\Big] I(y) dy.\notag
\end{align}
Here: 
\begin{itemize}\setlength{\itemsep}{0pt}
\item we define
\begin{equation}
	I(y) \defeq W\matrd{y}{}{}{1}\chi(\det y) |\det y|^{s-1/2}, 
\end{equation}
\item in the trace term, we have conjugated by $z^{\beta}$,
\item and we note that 
\[
|\det x| = |\det p^{n\beta}z^{-\beta}yz^\beta | =\det \left(\begin{smallmatrix}|p^{n\beta}| & & \\ & \ddots& \\ & & |p^{n\beta}| \end{smallmatrix}\right) \cdot |\det y| = p^{-n^2\beta}|\det y|.
\]
\end{itemize}

We now cut down the support of this integral. Firstly, by Lemma \ref{SupportLemma} we can immediately reduce the support to $\GL_n(\Qp) \cap M_n(\Zp)$. To go further, we exploit Iwahori invariance of $W$.

		\begin{notation}
		\begin{enumerate}[(1)]
			\item Let $A$ denote the set of all diagonal $n \times n$-matrices of the form
			\[
			\gamma = \operatorname{diag}(c_{11}, \dots, c_{nn}), \qquad c_{ii} \in \Z_p^\times.
			\]
			\item Let $B_{\beta}$ denote the additive group of all $n \times n$-matrices $\delta$ with
			\[
			\delta_{i, j} = \left\{ \begin{array}{cc} c_{i, j} & \text{ if } i < j \\ 0 & \text{ if } i = j \\ p^{\beta} c_{i, j} & \text{ if } i > j \end{array} \right., \qquad c_{ij} \in \Zp.
			\]
\end{enumerate}
We will consider matrices of the form $\alpha = \gamma + \delta \in M_n(\Zp)$ for $\gamma \in A$ and $\delta \in B_\beta$. Note that $\alpha$ is in the depth $p^\beta$ Iwahori subgroup $\Iwn(p^\beta) \subset \GL_n(\Zp)$ (the matrices that are upper-triangular modulo $p^\beta$). We set 
			\[
			\varepsilon = \tbyt{\alpha^{-1}}{}{}{1} \in \Iw \subset \GL_{2n}(\Zp).
			\]
	\end{notation}

Now we translate the argument of the zeta integral by $\varepsilon$. By Iwahori invariance, we get
\begin{align}
\zeta(s, &(u^{-1} t_p^\beta) \cdot W, \chi) = \zeta(s, (u^{-1} t_p^\beta\varepsilon) \cdot W, \chi) \label{eq:local zeta iwahori invariance}\\
& = \eta (\det z^\beta) p^{n^2\beta(s-1/2)} \notag\\
 &\hspace{20pt}\times \int_{\GL_n(\Qp)\cap M_n(\Zp)} \psi\Big[-\mathrm{tr}(p^{-n\beta}z^{\beta}yz^{-\beta}w_n)\Big] W\matrd{y\alpha^{-1}}{}{}{1}\chi(\det y) |\det y|^{s-1/2} dy.\notag
\end{align}
Make the change of variables $x = y\alpha^{-1}$; then this becomes
\begin{equation}\label{eq:local zeta iwahori invariance 2}
	= \eta (\det z^\beta) p^{n^2\beta(s-1/2)} \chi(\det \gamma) \int_{\GL_n(\Qp)\cap M_n(\Zp)} \psi\Big[-\mathrm{tr}(p^{-n\beta}z^{\beta}x\alpha z^{-\beta}w_n)\Big] I(x) dx.
\end{equation}
Here we used that $\det(\alpha) = \det(\gamma) \newmod{p^\beta}$, that $\chi$ has conductor $p^\beta$, and that $|\det(\gamma)| = 1$.

Now we have 
	\begin{equation}\label{eq:trace add}
	\psi\Big[ \mathrm{tr}(-p^{-n \beta} z^\beta x \alpha z^{-\beta} w_n ) \Big] = \psi\Big[ \mathrm{tr}(-p^{-n \beta} z^\beta x \gamma z^{-\beta} w_n ) \Big] \cdot \psi\Big[\mathrm{tr}(-p^{-n \beta} z^\beta x \delta z^{-\beta} w_n ) \Big].
	\end{equation}
	We cut the support down first by averaging over $\delta \in B_\beta$, then over $\gamma \in A$.
	
	\medskip
	
	\textbf{Step 1: Average over $B_\beta$}. 	For $x \in \GL_n(\Qp)\cap M_n(\Zp)$, define 
	\begin{align*}
	F_x : B_\beta &\longrightarrow \C^\times\\
	\delta &\longmapsto \psi\Big[\mathrm{tr}(-p^{-n \beta} z^\beta x \delta z^{-\beta} w_n )\Big].
\end{align*}
This is a group homomorphism by additivity of trace and $\psi$.
	
	\begin{lemma}\label{lem:F_x}
		\begin{enumerate}[(i)]
			\item There exists a finite index subgroup $B_\beta' \subset B_\beta$ such that $F_x$ is trivial on $B_\beta'$ for all $x \in \GL_n(\Qp)\cap M_n(\Zp)$. 
			\item For any fixed $x$, $F_x$ is the trivial function if and only if 
		\begin{equation}\label{eq:B_beta conditions}
		x_{n+1 - i, j} \in \left\{ \begin{array}{cc} p^{2\beta (n-i) + \beta} \Z_p & \text{ if } i > j \\ p^{2\beta (n-i)} \Z_p & \text{ if } i < j \end{array} \right.
		\end{equation}
	\end{enumerate}
	\end{lemma}
	\begin{proof}
		(i)  For $\delta$ sufficiently divisible by $p$, we have $F_x(\delta) = 1$ for all $x \in M_n(\Zp)$. 
		
		\medskip
		
		(ii) Writing out the trace explicitly, one sees
		\[
		\mathrm{tr}\Big[-p^{-n \beta} z^\beta x \delta z^{-\beta} w_n \Big] = -\sum_{i=1}^n \left( p^{\beta (1-2i)} x_{i, k} \left( \sum_{k < n+1 - i} c_{k, n+1 -i} + \sum_{k > n+1-i} p^{\beta} c_{k, n+1 - i} \right) \right) . 
		\]
		and uses the change of variables $i \mapsto n+1 - i$. If \eqref{eq:B_beta conditions} fails for some $(i,j)$, then $F_x$ will be non-trivial on the matrix $\delta$ which is zero apart from a 1 at $(i,j)$, so $F_x$ is not the trivial function. Conversely, if \eqref{eq:B_beta conditions} does hold, then the trace above is always integral and $F_x$ is trivial.
	\end{proof}
	
	Let $M_{\beta}' \subset \GL_n(\Qp)\cap M_n(\Zp)$ be the subset of matrices $x$ satisfying the conditions in \eqref{eq:B_beta conditions}. 
	
	\begin{corollary}\label{cor:average B}
		For any $\gamma \in A$, we have 
		\[
		\zeta(s,(u^{-1} t_p^{\beta}) \cdot W, \chi) = \eta(\det z^\beta) p^{n^2 \beta (s-1/2)} \chi(\operatorname{det} \gamma) \int_{M_\beta'} \psi\Big[ \mathrm{tr}(-p^{-n \beta} z^\beta x \gamma z^{-\beta} w_n ) \Big] \cdot I(x) dx.
		\]
	\end{corollary}
	\begin{proof}
		Using \eqref{eq:local zeta iwahori invariance} (in the first equality) and \eqref{eq:local zeta iwahori invariance 2} and \eqref{eq:trace add} (in the second), we have
		\begin{align*}
			\zeta&(s,(u^{-1}t_p^\beta)\cdot W, \chi) = \frac{1}{[B_\beta:B_\beta']}\sum_{\delta \in B_\beta/B_\beta'}\zeta(s,(u^{-1}t_p^\beta(\gamma + \delta))\cdot W, \chi) \\
			&= \eta (\det z^\beta) p^{n^2\beta(s-1/2)}\chi(\det \gamma) \int_{\GL_n(\Qp)\cap M_n(\Zp)} \psi\Big[-\mathrm{tr}(p^{-n\beta}z^{\beta}x\gamma z^{-\beta}w_n)\Big]\\
			&\hspace{160pt} \times \bigg[\frac{1}{[B_\beta:B_\beta']}\sum_{\delta \in B_\beta/B_\beta'} F_x(\delta)\bigg] I(x) dx.
		\end{align*}
	When $F_x$ is non-trivial, the square-bracketed term (hence the integrand) vanishes by character orthogonality; and when $F_x \equiv 1$, it is identically 1. We conclude since by Lemma \ref{lem:F_x}(ii), $F_x$ is trivial if and only if $x \in M_\beta'$. 
	\end{proof}
	
	\textbf{Step 2: Average over $A$}. Equip $A \cong (\Zp^\times)^n$ with the measure $d^{\times}A = \prod_{i=1}^n d^{\times}c_{i, i}$. 
	
	\begin{lemma}\label{lem:average A}
		We have 
		\[
		\chi(\operatorname{det}\gamma)\psi\Big[ \mathrm{tr}(-p^{-n \beta} z^\beta x \gamma z^{-\beta} w_n ) \Big] = \prod_{i=1}^n \chi( c_{i, i} )\psi\left( - p^{-(2\beta(n-i) + \beta)} x_{n+1 -i, i} c_{i,i} \right).
		\]
		Therefore
		\begin{align*} 
			\int_{\gamma \in A} \chi(\operatorname{det}\gamma)&\psi\Big[ \operatorname{tr}(-p^{-n \beta} z^\beta x \gamma z^{-\beta} w_n ) \Big] d^{\times}A\\
			 &= \prod_{i=1}^n \chi(-1)\int_{\Z_p^{\times}}  \chi( c_{i, i} )\psi\left(p^{-(2\beta(n-i) + \beta)} x_{n+1 -i, i} c_{i,i} \right) d^{\times} c_{i, i} \\
			&= \left\{ \begin{array}{cc} \frac{\chi(-1)^n\tau(\chi)^n}{p^{n\beta}(1-p^{-1})^n} \prod_{i=1}^n \chi(x'_{n+1-i, i})^{-1} & \text{ if } x_{n+1-i, i} \in p^{2\beta(n-i)} \Z_p^{\times} \ \forall i \\ 0 & \text{ otherwise } \end{array} \right.
		\end{align*}
		where $x'_{n+1-i, i} = x_{n+1-i, i}/p^{2\beta(n-i)}$ and $\tau(\chi)$ is the Gauss sum from \eqref{eq:gauss sum}.
	\end{lemma}
	\begin{proof}
		We have 
		\[
		\chi(\operatorname{det}\gamma) = \prod_{i=1}^n \chi( c_{i, i} ) 
		\]
		and
		\begin{align*} 
			\psi\Big[\operatorname{tr}(-p^{-n \beta} z^\beta x \gamma z^{-\beta} w_n ) \Big] &= \prod_{i=1}^n \psi\left( - p^{\beta(1-2i)} x_{i, n+1 - i} c_{n+1-i,n+1 -i} \right) \\
			&= \prod_{i=1}^n \psi\left( - p^{-(2\beta(n-i) + \beta)} x_{n+1 -i, i} c_{i,i} \right),
		\end{align*} 
		giving the first part. The rest follows from a simple change of variables.
	\end{proof}
	
	Let $M_{\beta} \subset M_\beta'$ be the subset where $x_{n+1-i, i} \in p^{2\beta(n-i)} \Z_p^{\times}$ for all $i$. Note that 
	\[
		M_{\beta} = w_n z^{2\beta} \Iwn(p^{\beta}).
	\]
	
	\begin{corollary}\label{cor:average A}
		We have 
		\[
		\zeta(s,(u^{-1}t_p^\beta) \cdot W, \chi) = \eta(\det z^\beta) p^{n^2 \beta (s-1/2)} \frac{\chi(-1)^n\tau(\chi)^n}{p^{n\beta}(1-p^{-1})^n} \int_{M_{\beta}} \prod_{i=1}^n \chi(x'_{n+1-i, i})^{-1} I(x) dx.
		\]
	\end{corollary}
	\begin{proof}
		This is similar to Corollary \ref{cor:average B}. Integrate the expression of that corollary over $A$, and reduce the support using Lemma \ref{lem:average A}. 
	\end{proof}
	
	Now using that fact that $M_{\beta} = w_n z^{2\beta} \Iwn(p^{\beta})$ we write $x = w_n z^{2\beta} x''$ and note that $\operatorname{det}x'' \equiv \prod_{i=1}^n x'_{n+1-i, i}$ modulo $p^\beta$. Making the change of variables, we see that
	\begin{align}
		\int_{M_\beta}  \prod_{i=1}^n& \chi(x'_{n+1-i, i})^{-1} I(x) dx\notag\\
		& = \int_{\Iwn(p^\beta)} W\matrd{w_nz^{2\beta}x''}{}{}{1} \prod_{i=1}^n \chi(x'_{n+1-i, i})^{-1}  \chi(\det w_nz^{2\beta} x'') |\det w_nz^{2\beta}|^{s-1/2}dx''\notag\\
		&= \chi(\det w_n)\cdot p^{-\beta n(n-1)(s-1/2)}\int_{\Iwn(p^\beta)} W\matrd{w_nz^{2\beta}}{}{}{1} dx''\notag\\
		&= \chi(\det w_n)\cdot p^{-\beta n(n-1)(s-1/2)}\cdot \mathrm{vol}(\Iwn(p^\beta)) \cdot W\matrd{w_nz^{2\beta}}{}{}{1}.\label{eq:W value}
	\end{align}
In the penultimate equality, we have used that
\[
\operatorname{det}x'' \equiv \prod_{i=1}^n x'_{n+1-i, i} \newmod{p^\beta},
\]
$\chi(p) = 1$, $|\det w_n| = 1$, and $|\det z| = p^{-n(n-1)/2}$.
Finally note that 
	\begin{equation}\label{eq:vol Iw}
	\operatorname{vol}(\Iwn(p^\beta)) = p^{-(\beta-1)\tfrac{n^2-n}{2}} \operatorname{vol}(\Iwn).
	\end{equation}
	Putting \eqref{eq:vol Iw} and \eqref{eq:W value} into Corollary \ref{cor:average A}, using that $\chi(-1)^n\chi(\det w_n) = \chi(\det(-w_n))$, completes the proof of Proposition \ref{prop:local zeta p}.
\end{proof}

	%%=====================================================
	%%
	%% 		SHALIKA REFINEMENTS
	%%
	%%=====================================================
	\section{Spin $p$-refinements}\label{sec:spin refinements}
	As highlighted in the introduction to Part II, we want to answer:
	\begin{quote}
		\emph{For which $p$-refinements $\tilde\pi$ is there $W \in \cS_{\psi}^\eta(\tilde\pi)$ with $\zeta(s,(u^{-1}t_p^\beta)\cdot W,\chi) \neq 0$?}
	\end{quote}
 Given Proposition \ref{prop:local zeta p}, this is equivalent to asking when there exists $W \in \tilde\pi$ and $\beta \geq 1$ such that $W\smallmatrd{w_nz^{2\beta}}{}{}{1} \neq 0$. We expect this to be true only for a special class of $p$-refinements, those that `interact well with the Shalika model'. 
 
 In this section, we begin to make this assertion rigorous. We define `spin' $p$-refinements as those `that come from $\mathrm{GSpin}_{2n+1}$', made precise in Proposition \ref{prop:spin refinement}. In later sections we will show that for any spin $p$-refinement $\tilde\pi$, there exists $W\in \cS_{\psi}^\eta(\tilde\pi)$ such that $W\smallmatrd{w_nz^{2\beta}}{}{}{1} \neq 0$. Our key application of spin $p$-refinements will ultimately be summarised in Corollary \ref{cor:L-value}.

	\begin{notation}
		Let $\cG \defeq \mathrm{GSpin}_{2n+1}$, taking the split form. As $\pi$ is a generic unramified principal series and admits a Shalika model, by \cite{AS06} it is the functorial transfer of an unramified principal series representation $\Pi$ of $\cG(\Qp)$. We shall describe $\Pi$ explicitly in Proposition \ref{prop:chi}.
	\end{notation}
	
	We first give a concrete definition of spin $p$-refinement. We justify it in \S\ref{sec:spin via gspin}, and give several equivalent formulations in Propositions \ref{prop:alpha spin} and \ref{prop:spin refinement}. Let $\tilde\pi = (\pi,\alpha)$ be a $p$-refinement of $\pi$ (as in \S\ref{sec:hecke p}), and for $1 \leq r \leq 2n-1$ write $\alpha_{p,r} \defeq \alpha(U_{p,r})$.
	
	\begin{definition}\label{def:spin-refinement}
		We say $\tilde\pi$ is a \emph{spin $p$-refinement} if $\alpha_{p,n+s} = \eta(p)^s\alpha_{p,n-s}$ for all $0 \leq s \leq n-1$.
	\end{definition}
	
	Recall here $\eta$ is the Shalika character (see page \pageref{shalika character}). We will show $\tilde\pi$ is spin if and only if $\alpha$ factors through an eigensystem occuring in $\Pi^{\Iwahori_{\cG}}$, for $\Iwahori_{\cG}$ the Iwahori subgroup of $\cG(\Zp)$ (see \S\ref{sec:spin via gspin}).

	\subsection{Conventions for Shalika models}\label{sec:shalika conventions}
		As $\pi$ is spherical, it can be written as a (normalised) induction from the upper-triangular Borel, i.e.\ there exists an unramified character $\UPS : T(F) \to \C^\times$ such that	$\pi = \Ind_B^G \UPS$. This $\UPS$ is well-defined up to conjugation by $\cW_G$. Moreover, we have:
	
	\begin{proposition}\cite[Prop.\ 1.3]{AG94}. \label{prop:AG 1} The unramified principal series $\Ind_B^G \UPS$ admits an $(\eta,\psi)$-Shalika model if and only if there is a decomposition $\{1,...,2n\} = X_1 \sqcup X_2$, where $\# X_j = n$, and a bijection $\nu : X_1 \to X_2$ such that $\UPS_i \UPS_{\nu(i)} = \eta$ for all $i \in X_1$.	
	\end{proposition}

	Given choices of a decomposition $X_1 \sqcup X_2$ and a bijection $\nu$, we can find an identification $\pi = \Ind_B^G \UPS$ with $\UPS_i\UPS_{\nu(i)} = \eta$. 

The natural choice of decomposition is $X_1 = \{1,...,n\}$, $X_2 = \{n+1,...,2n\}$. For this there will be \emph{two} natural choices of $\nu$, with their own advantages and disadvantages. 
	\begin{itemize}\s
		\item[--] In \cite{AG94}, $\nu$ is chosen so that $\nu(i) = n+i$. When $\UPS_i\UPS_{n+i} = \eta$, Ash--Ginzburg use this to define an explicit intertwining $\Ind_B^G \UPS \to \cS_\psi^\eta(\pi)$, which we describe in \S\ref{sec:AG}. 
		
		\item[--] In \cite{AS06}, $\nu$ is chosen so that $\nu(i) = 2n+1-i$. This is considerably more natural when discussing spin refinements, as will become clear later in this section. 
	\end{itemize}

	Since we will later use Ash--Ginzburg's explicit intertwining, we make the following choice:
	
	\begin{quote}
		\emph{In the rest of the paper, assume $\pi = \Ind_B^G \UPS$, where $\UPS_i\UPS_{n+i} = \eta$ for $1 \leq i \leq n$.}
	\end{quote}

Note that the Ash--Ginzburg and Asgari--Shahidi choices are interchanged by conjugation by $\tau \defeq \smallmatrd{1}{}{}{w_n}$. In particular, given the choice of $\UPS$ fixed above, we have $\UPS^\tau_{i}\UPS^{\tau}_{2n+1-i} = \eta$. This helps when it is more convenient to use Asgari--Shahidi's choice (see e.g.\ Remark \ref{rem:AG vs AS}).

	\subsection{Root systems for $\GL_{2n}$ and $\mathrm{GSpin}_{2n+1}$} \label{sec:structure gspin}
	Recall the space of algebraic characters and cocharacters of the torus $T \subset G = \GL_{2n}$ are given by
	\[
	X = \Z e_1 \oplus \Z e_2 \oplus \cdots \Z e_{2n}, \hspace{12pt} X^\vee = \Z e_1^* \oplus \Z e_2^* \oplus \cdots \Z e_{2n}^*.
	\]
	Write $\langle-,-\rangle_G$ for the natural pairing on $X \times X^\vee$. The corresponding root system is $A_{2n-1}$, with roots $R = \{\pm(e_i - e_j) : 1 \leq i < j \leq 2n\}$. The Weyl group $\cW_G = \mathrm{S}_{2n}$ acts by permuting the $e_i$, with longest Weyl element $w_{2n}$ the permutation that sends $e_i \mapsto e_{2n+1-i}$ for all $i$.  
	
	Let $X_0 \subset X$ be the space of \emph{pure characters} $X_0 = \{\lambda \in X: \exists \sw \in \Z \text{ such that } \lambda_i + \lambda_{2n-i+1} = \sw\}$, and let 
	\begin{equation}\label{eq:W_G^0}
	\cW_G^0 \defeq \{ \sigma \in \cW_G :  \forall \lambda \in X_0, \ \lambda^\sigma \in X_0\} \subset \cW_G.
	\end{equation}

	Now fix a standard upper Borel subgroup $\cB$ and maximal split torus $\cT$ in $\cG = \mathrm{GSpin}_{2n+1}$. This has rank $n+1$ \cite[Thm.\ 2.7]{Asg02}. We use calligraphic letters to denote objects for GSpin, and otherwise maintain the same notational conventions as before. 
	
	\begin{proposition}\label{prop:spin root system}
		The root system for $\cG$ is $(\cX, \cR, \cX^\vee, \cR^\vee)$, where
		\[
		\cX = \Z f_0 \oplus \Z f_1 \oplus \cdots \oplus \Z f_n, \hspace{12pt} \cX^\vee = \Z f_0^* \oplus \Z f_1^* \oplus \cdots \oplus \Z f_n^*,
		\] with roots $\cR = \{\pm f_i \pm f_j : 1 \leq i<j\leq n\} \cup \{f_i : 1 \leq i \leq n\}$ and  positive roots $\{f_i : 1 \leq i \leq n\} \cup \{f_i \pm f_j : 1 \leq i < j \leq n\}$. 	The Weyl group $\cW_{\cG}$ has size $2^n\cdot n!$, generated by permutations $\sigma \in S_n$ and sign changes $\mathrm{sgn}_i$, which act on roots and coroots respectively as (for $j \neq i$)
		\begin{equation}\label{eq:weyl action cG}
			\sigma f_0 = f_0, \ \ \sigma f_i = f_{\sigma(i)},\ \ \ \mathrm{sgn}_i f_0 = f_0 + f_i,\ \ \mathrm{sgn}_i(f_i) = -f_i, \ \ \mathrm{sgn}_j(f_i) = f_i,
		\end{equation}
		\[
		\sigma f_0^* = f_0^*, \ \ \sigma f_i^* = f_{\sigma(i)}^*,\ \ \ \mathrm{sgn}_i f_0^* = f_0^*,\ \ \mathrm{sgn}_i(f_i^*) = f_0^* -f_i^*,  \ \ \mathrm{sgn}_j(f_i^*) = f_i^*.
		\]
		In particular, $\cW_{\cG}$ is the semidirect product $\{\pm1\}^n \rtimes S_n$.
	\end{proposition}
	\begin{proof}
		The first part is \cite[Prop.\ 2.4]{Asg02}, and the second \cite[Lem.\ 13.2.2]{HS16}.
	\end{proof}
	
	Write $\langle-,-\rangle_{\cG}$ for the natural pairing on $\cX \times \cX^\vee$.

	\subsection{The maps $\jmath$ and $\jmath^\vee$}\label{sec:j}
	There is a natural injective map $\jmath : \cX \hookrightarrow X$ given by
	\begin{equation}\label{eq:f and e}
		f_i \longmapsto e_i - e_{2n-i+1} \text{ for } 1 \leq i \leq n,\qquad 
		f_0  \longmapsto e_{n+1} + \cdots + e_{2n}.\notag
	\end{equation}
	We may identify $\cX$ with cocharacters of $\mathrm{GSp}_{2n}$, and $X$ with cocharacters of $\GL_{2n}$. The map $\jmath$ is then the natural map on cocharacters induced by the inclusion $T_{\mathrm{GSp}_{2n}} \subset T$ of tori.

	\begin{proposition}\label{prop:local pure}
		We have $X_0 = \jmath(\cX)$.
	\end{proposition}
	\begin{proof}
		Any linear combination of the $\jmath(f_i)$ is a pure weight with purity weight $0$, and any such weight arises in this form; and scaling the purity weight to $\sw$ corresponds to adding $\jmath(\sw f_0)$. 
	\end{proof}

	\begin{proposition}\label{prop:weyl transfer}
		There is a map $\cW_{\cG} \to \cW_G$ of Weyl groups, which we also call $\jmath$, such that:
		\begin{itemize}\setlength{\itemsep}{0pt}
			\item[(i)] $\jmath$ induces an isomorphism $\cW_{\cG} \xrightarrow{\jmath} \cW_G^0 \subset \cW_G$;
			\item[(ii)] for all  $\sigma \in \cW_{\cG}$ and $\mu \in \cX$, we have $\jmath(\mu^\sigma) = \jmath(\mu)^{\jmath(\sigma)}$.
		\end{itemize}
	\end{proposition}
	\begin{proof}
		We know $\cW_{\cG} \cong \{\pm1\}^n \rtimes S_n$ is generated by sign changes and permutations in $\{f_1,...,f_n\}$, and $\cW_G = S_{2n}$ is permutations in $\{e_1,...,e_{2n}\}$. If $\sigma \in S_n \subset \cW_{\cG}$ is a permutation, then define
		\[
		\jmath(\sigma) = \smallmatrd{\sigma}{}{}{w_n\sigma w_n} : (e_1,...,e_n,e_{n+1},...,e_{2n}) \mapsto \left(e_{\sigma(1)}, ..., e_{\sigma(n)}, e_{2n-\sigma(n)+1}, ..., e_{2n-\sigma(1)+1}\right),
		\]
		and if $\sigma = \epsilon_i$ is the sign change at $i \in \{1,...,n\}$, define $\jmath(\epsilon_i)$ to be the permutation switching $e_i$ and $e_{2n-i+1}$. A simple check shows this induces a well-defined homomorphism.
		
		Suppose $\lambda \in X$ is pure with all the $\lambda_i$ distinct. Let $\sigma \in \cW_G = S_n$; if $\lambda^\sigma$ is pure, then $\sigma$ must preserve the relative positions of each pair $\{\lambda_i, \lambda_{2n-i+1}\}$. The only way to do this is to permute $i \in \{1,...,n\}$ or to switch $\lambda_i,\lambda_{2n-i+1}$. These are exactly the permutations in $\jmath(\cW_{\cG})$, giving (i).
		
		Part (ii) is a simple explicit check.
	\end{proof}
	
	\begin{corollary}
	There is a short exact sequence $1 \to \{\pm 1\}^n \to \cW_G^0 \to S_n \to 1$, which is split by $\jmath : S_n \to \cW_G^0$. The image of $\{\pm1\}^n$ is generated by the transpositions $(i,2n+1-i)$ for $1 \leq i \leq n$.
	\end{corollary}
	
	Dually, define also a map $\jmath^\vee: X^\vee \to \cX^\vee$ by sending $\nu \in X^\vee$ to
	\[
	\jmath^\vee(\nu) \defeq \sum_{i = 0}^{n} \big\langle \jmath(f_i), \nu\big\rangle_G \cdot f_i^*.
	\]
	Then for all $\mu \in X$, we have 
	\begin{equation}\label{eq:pairing jmath}
		\langle \mu, \jmath^\vee(\nu)\rangle_{\cG} = \langle \jmath(\mu), \nu\rangle_G
	\end{equation}
	by construction. (Again, this map arises from our explicit realisation of $T_{\mathrm{GSp}_{2n}} \subset T$). Also let $\jmath^\vee : \cW_G^0 \to \cW_{\cG}$ denote the inverse to $\jmath : \cW_{\cG} \cong \cW_G^0$.
	\begin{proposition}\label{prop:jmath vee eq}
		For all $\nu \in X^\vee$ and $\sigma \in \cW_G^0$, we have $\jmath^\vee(\nu^\sigma) = \jmath^\vee(\nu)^{\jmath^\vee(\sigma)}$.
	\end{proposition}
	\begin{proof}
		If $\sigma \in \cW_G^0$, then $\sigma = \jmath(\rho)$ for some $\rho \in \cW_{\cG}$, and $\jmath^\vee(\sigma) = \rho$. Then compute that
		\begin{align*}
			\jmath^\vee(\nu^\sigma) \defeq \sum_{i=0}^n \langle\jmath(f_i),\nu^\sigma\rangle_G f_i^* = \sum_{i=0}^n\langle\jmath(f_i)^{\jmath(\rho^{-1})},\nu\rangle_G f_i^* &= \sum_{i=0}^n\langle\jmath(f_i^{\rho^{-1}}),\nu\rangle_G f_i^*\\
			&= \sum_{i=0}^n\langle\jmath(f_i),\nu\rangle_G (f_i^*)^\rho = \jmath^\vee(\nu)^{\jmath^\vee(\sigma)},
		\end{align*}
		where the penultimate equality is a simple check.
	\end{proof}

	\subsection{Spin refinements via GSpin}\label{sec:spin via gspin}
	Via the Satake isomorphism, we may describe Hecke operators in terms of cocharacters. 
	Note $U_{p,r}$ is attached to the cocharacter 
	\begin{equation}\label{eq:cocharacter}
		\nu_{r} \defeq e_1^* + \cdots + e_r^* \in X^\vee,
	\end{equation}
	since $t_{p,r}^G \defeq t_{p,r} = \nu_{r}(p)$. Define $t_{p,r}^{\cG} \defeq \jmath^\vee(\nu_{r})(p) \in \cT(F)$; by definition, for $0 \leq s \leq n-1$ 
	\begin{equation}\label{eq:jmath on cocharacters}
		\jmath^\vee(\nu_{n-s}) = f_1^* + \cdots + f_{n-s}^*,\qquad \jmath^\vee(\nu_{n+s}) =  \jmath^\vee(\nu_{n-s}) + sf_0^*.
	\end{equation}
	Let $\Iwahori_{\cG} \subset \cG(\Zp)$ be the Iwahori subgroup for $\cG$, and for $1 \leq r \leq n$, let $\cU_{p,r} \defeq [\Iwahori_{\cG} \ t_{p,r}^{\cG} \ \Iwahori_{\cG}]$. Also let $\mathcal{V}_p \defeq [\Iwahori_{\cG}\ f_0^*(p)\  \Iwahori_{\cG}]$. From \eqref{eq:jmath on cocharacters}, we see $\cU_{p,n+s} = \cV_p^s \cdot \cU_{p,n-s}$.
	
	Let $\cH^{\cG}_{p} \defeq \Zp[\cV_p, \cU_{p,r}: 1 \leq r \leq n]$ be the Hecke algebra for $\cG$. Then $\jmath^\vee$ induces a map
	\[
	\cH_p \longrightarrow \cH_p^{\cG}, \qquad U_{p,n-s} \longmapsto \cU_{p,n-s}, \ \ U_{p,n+s} \longmapsto \cV_p^s \cdot \cU_{p,n-s},
	\]
	for each $0 \leq s \leq n-1$. Comparing with Definition \ref{def:spin-refinement}, we obtain:
	
	\begin{proposition}\label{prop:alpha spin}
		A $p$-refinement $(\pi,\alpha)$ is a spin $p$-refinement if and only if $\alpha$ factors through
		\[
		\cH_p \xrightarrow{\ \ \jmath^\vee \ \ } \cH_p^{\cG} \xrightarrow{ \ \ \alpha^{\cG} \ \ } \overline{\Q},
		\]
		for some character $\alpha^{\cG}$ with $\alpha^{\cG}(\cV_p) = \eta(p)$.
	\end{proposition}
	
	\begin{remark}
		We could add the operator $U_{p,2n} = [\Iw \ \mathrm{diag}(p,...,p)\  \Iw]$ to $\cH_p$; it acts by $\mathrm{diag}(p,...,p)$, so acts on $\pi$ by $\eta(p)^n$ (since $\pi$ has central character $\eta^n$). We would then have $\jmath^\vee(U_{p,2n}) = \cV_p^n$. In particular,  the requirement that $\alpha^{\cG}(\cV_p) = \eta(p)$ is natural.
	\end{remark}

	Recall $\pi$ on $G(\Qp)$ is the transfer of $\Pi$ on $\cG(\Qp)$, and:
	\begin{equation}\label{eq:satake}
	\text{\emph{From now on, we assume that the Satake parameter of $\pi$ is regular.}}
	\end{equation}
	 We now show in this case that if $(\pi,\alpha)$ is a spin refinement, then the system of eigenvalues $\alpha^{\cG}$ occurs in $\Pi^{\Iwahori_{\cG}}$.
	
	Recall in \S\ref{sec:shalika conventions} we fixed an unramified character $\theta$ of $T(\Qp)$ such that $\pi = \Ind_B^G \UPS$ and $\UPS_i\UPS_{n+i} = \eta$ (cf.\ \cite[(43)]{DJR18}, where $\UPS$ is denoted $|\cdot|^{(2n-1)/2}\lambda$). Recall $\tau = \smallmatrd{1_n}{}{}{w_n} \in \cW_G$, and that $\UPS^\tau$ satisfies the Asgari--Shahidi condition $\UPS_i^\tau\UPS_{2n+1-i}^\tau = \eta$. From \cite[p.177(i)]{AS06} and \cite[Prop.\ 5.1]{AS14}, we see:
	
	\begin{proposition}\label{prop:chi}
		There is an unramified character $\UPS_{\cG}$ of $\cT(\Qp)$ such that:
		\begin{itemize}\setlength{\itemsep}{0pt}
			\item[(i)] $\Pi = \Ind_{\cB}^{\cG}\UPS_{\cG}$ is a (normalised) parabolic induction,
			\item[(ii)] we have $\jmath(\UPS_{\cG}) = \UPS^\tau$.
		\end{itemize}
	\end{proposition}

	\begin{remark}\label{rem:AG vs AS}
 The $\tau$ is necessary as we are using the convention of Ash--Ginzburg (see \S\ref{sec:shalika conventions}). From this, we see that $\UPS^\tau$ is a more natural convention for GSpin computations. If we chose $\UPS$ as in Asgari--Shahidi, we could have removed $\tau$ from (ii) and henceforth in this section.
 
 Note that the Asgari--Shahidi condition $\UPS_i\UPS_{2n+1-i} = \eta$ is preserved by $\cW_G^0$, reflecting its natural place in the spin world. On the other hand the Ash--Ginzburg condition is preserved instead by $\tau^{-1}\cW_G^0\tau$. 
	\end{remark}

By \eqref{eq:delta ups} and \eqref{eq:satake}, our fixed choice of $\UPS$ (hence $\UPS^\tau$)  fixes a bijection 
	\[
	\Delta_{\UPS^\tau} : \{p\text{-refinements of }\pi\} \longrightarrow \cW_G.
	\]
	
	\begin{lemma}\label{lem:Shalika W_G^0}
		A $p$-refinement $\tilde\pi = (\pi,\alpha)$ is spin if and only if $\Delta_{\UPS^\tau}(\tilde\pi) \in \cW_G^0$. 
	\end{lemma}
	\begin{proof}
		Conjugating Proposition \ref{prop:AG 1} by $\tau$, we see for each $i$, we have $\UPS_i^\tau \cdot \UPS_{2n+1-i}^\tau = \eta$ as characters of $\Qp^\times$. Let $\sigma = \Delta_{\UPS^\tau}(\tilde\pi)$. By definition of $\cW_G^0$, we see if $\sigma \in \cW_G^0$ then 
		\begin{equation}\label{eq:sigma}
			\UPS_{\sigma(i)}^\tau \cdot \UPS_{\sigma(2n+1-i)}^\tau = \eta,
		\end{equation}
		whilst as the Satake parameter is regular, if $\sigma \notin \cW_G^0$, then \eqref{eq:sigma} fails for some $i$. Thus \eqref{eq:sigma} holds for all $i$ if and only if $\sigma \in \cW_G^0$.
		
From the explicit description of $\alpha_{p,r} = \alpha(U_{p,r})$ from Proposition \ref{prop:p-refinement}, we see that $\alpha_{p,n+s} = \eta(p)^s\alpha_{p,n-s}$ if and only if \eqref{eq:sigma} holds for all $i$. The result follows.
	\end{proof}

	A \emph{$p$-refinement of $\Pi$} is a tuple $\tilde\Pi = (\Pi, \alpha^{\cG})$, where $\alpha^{\cG} : \cH_{p}^{\cG} \to \overline{\Q}$ is a system of Hecke eigenvalues appearing in $\Pi^{\Iwahori_{\cG}}$. We say $\tilde\Pi$ is \emph{regular} if this system of eigenvalues appears in $\Pi^{\cG}$ without multiplicity, i.e.\ the generalised eigenspace is a line. As in Proposition \ref{prop:p-refinement}, after fixing the unramified character $\UPS_{\cG}$, such $p$-refinements correspond to elements $\sigma \in \cW_{\cG}$. 

The following is our main motivation for the definition of spin $p$-refinement.
	
	\begin{proposition}\label{prop:spin refinement}
		Suppose the Satake parameter of $\pi$ is regular, and let $\tilde\pi = (\pi,\alpha)$ be a $p$-refinement. Then $\tilde\pi$ is a spin $p$-refinement if and only if there exists a $p$-refinement $(\Pi,\alpha^{\cG})$ of $\Pi$ such that $\alpha = \alpha^{\cG} \circ \jmath^\vee$ as characters $\cH_{p} \to \overline{\Q}$.
	\end{proposition}
	
	\begin{proof}
		By Proposition \ref{prop:alpha spin}, $\tilde\pi$ is spin if and only if $\alpha$ factors through some $\alpha^{\cG}$; so suffices to show that in this case, the system $\alpha^{\cG}$ occurs in $\Pi^{\Iwahori_{\cG}}$. Let $\sigma = \Delta_{\UPS^\tau}(\tilde\pi)$. By Lemma \ref{lem:Shalika W_G^0}, $\sigma \in \cW_G^0$.
		
		Denote half the sum of the positive roots for $G$ and $\cG$ by
		\begin{equation}\label{eq:rho G}
			\rho_G = \left(\tfrac{2n-1}{2}, \tfrac{2n-3}{2}, \cdots, \tfrac{-(2n-3)}{2}, \tfrac{-(2n-1)}{2}\right), \qquad \rho_{\cG} = \tfrac{2n-1}{2}f_1  + \tfrac{2n-3}{2}f_2 + \cdots + \tfrac{1}{2}f_n.
		\end{equation}
		Note $\jmath(\rho_{\cG}) = \rho_G$. By rewriting the formulation of Proposition \ref{prop:p-refinement}, the $U_{p,r}$-eigenvalue of $\tilde\pi$ can be written as
		\[
			\alpha_{p,r} = q^{\langle \rho_G, \nu_{p,r}\rangle_G} p^{\langle\UPS^\tau,\nu_{p,r}^\sigma\rangle_G},
		\]
		where we identify $\UPS^\tau(\nu_{p,r}^\sigma(p)) = p^{\langle\UPS^\tau,\nu_{p,r}^\sigma\rangle_G}$ under the natural extension of $\langle-,-\rangle_G$. 
		
		Since $\sigma \in \cW_G^0$, by Proposition \ref{prop:weyl transfer} it is of the form $\jmath(w)$ for some $\omega \in \cW_{\cG}$. Let $\tilde\Pi = (\Pi,\tilde\alpha)$ be the $p$-refinement corresponding to $\omega$; then by considering the characteristic polynomial of $\cU_{p,r}$ on $\Pi_p^{\Iwahori_{\cG}}$ (see \cite[Cor.\ 3.16]{OST19}), we see that the $\cU_{p,r}$-eigenvalue attached to $\tilde\Pi$ is 
		\begin{align*}
			\tilde\alpha(\cU_{p,r}) = q^{\langle \rho_{\cG}, \jmath^\vee(\nu_{p,r})\rangle_{\cG}} p^{\langle\UPS_{\cG},\jmath^\vee(\nu_{p,r}^\omega)\rangle_{\cG}} &= q^{\langle \rho_{G}, \nu_{p,r}\rangle_{G}} p^{\langle\UPS_{\cG},\jmath^\vee(\nu_{p,r})^{\jmath^\vee(\sigma)}\rangle_{\cG}}\\
			&= q^{\langle \rho_{G}, \nu_{p,r}\rangle_{G}} p^{\langle\UPS^\tau,\nu_{p,r}^{\sigma}\rangle_{G}} = \alpha_{p,r},
		\end{align*}
		where in the second equality we have used $\jmath(\rho_{\cG}) = \rho_G$ with \eqref{eq:pairing jmath}, and in the third we have used Proposition \ref{prop:jmath vee eq} with \eqref{eq:pairing jmath}. In particular, $\tilde\alpha(\cU_{p,r}) = \alpha^{\cG}(\cU_{p,r})$ for all $r$.
		
		It remains to show $\tilde\alpha(\cV_{p}) = \alpha^{\cG}(\cV_{p})$. Note $\alpha^{\cG}(\cV_{p}) = \eta(p)$ by Proposition \ref{prop:alpha spin}. Also, $f_0^*(p)$ is central in $\cG(\Qp)$ by \cite[Prop.\ 2.3]{AS06}, and the central character of $\Pi$ is $\eta$ by p.178 \emph{op.\ cit}. Hence $\cV_{p}$ acts on $\Pi$ by $\eta(p)$. It follows that $\tilde\alpha(\cV_{p}) = \eta(p)$, and we conclude that $\tilde\alpha = \alpha^{\cG}$, as required.
	\end{proof}

	\begin{remark}
		We finally indicate how spin $p$-refinements relate to the notion of $Q$-regular $Q$-refinement in \cite[Def.\ 3.5]{DJR18}. This was defined to be an element $T \in \cW_G/\cW_H$, equivalent to a choice of $n$-element subset  $S_T \subset \{1,...,2n\}$, satisfying two conditions. Their condition (i) is our definition of regularity, and their condition (ii) guarantees that $T$ lies in the image of the composition $\cW_G^0 \hookrightarrow \cW_G \to \cW_G/\cW_H$. One sees that spin $p$-refinements are in bijection with $T \in \cW_G/\cW_H$ satisfying (ii) together with an ordering of $S_T$.
	\end{remark}

	\section{Shalika $p$-refinements}\label{sec:shalika refinements}

	We now define another class of $p$-refinements. Let $\tilde\pi = (\pi,\alpha)$ be a $p$-refinement. Recall we write $f \in \tilde\pi$ as shorthand for $f \in \pi^{\Iw} [\![U_{p,r} - \alpha(U_{p,r}) : 1 \leq r \leq 2n-1]\!]$ (and similarly for $W \in \cS_\psi^\eta(\tilde\pi)$). Recall that $z = \mathrm{diag}(p^{n-1},p^{n-2},...,p,1)$.

	\begin{definition}\label{def:shalika refinement}
		We say $\tilde\pi$ is a \emph{Shalika $p$-refinement} if there exist $W \in \cS_\psi^\eta(\tilde\pi)$ and $\beta \geq 1$ such that
		\begin{equation}\label{eq:shalika refinement}
		W\matrd{w_nz^{2\beta}}{}{}{1} \neq 0.
		\end{equation}
	\end{definition}

	Note that via Proposition \ref{prop:local zeta p}, this implies that the local zeta integral $\zeta(s,  (u^{-1}t_{\pri,n}^\beta)\cdot W, \chi)$ is non-vanishing for any smooth character $\chi : \Qp^\times \to \C^\times$ of conductor $p^\beta$. Accordingly this generalises the condition \cite[\S2.8, condition (C2)]{BDW20} required to construct a $p$-adic $L$-function for $\tilde\pi$ via the methods of \cite{BDW20}. 
	
The following connects this new definition to the previous section.
	
	\begin{expectation}\label{conj:shalika = spin}
		A refinement $\tilde\pi$ is a Shalika refinement if and only if it is a spin refinement.
	\end{expectation}
	
	Our main results of this section (Proposition \ref{prop:spin refinements are shalika} and Corollary \ref{cor:spin refinements are shalika}) show one direction of this when $\pi$ is regular: that all spin $p$-refinements are Shalika $p$-refinements. Moreover in this case we precisely compute the value \eqref{eq:shalika refinement}.

\begin{remark}
Our expectation is motivated by folklore conjectures on classical families in eigenvarieties. More precisely, in  \S\ref{sec:shalika families} we will show that through any non-critical slope Shalika refinement $\tilde\pi$, the $\GL_{2n}$-eigenvariety is \'etale over the pure weight space, and the component through $\tilde\pi$ contains a Zariski-dense set of symplectic points. The folklore conjecture says this should not be possible unless the family is transfer from $\mathrm{GSpin}_{2n+1}$, whence we expect $\tilde\pi$ to have been a spin refinement by the results of the previous section. We have elaborated on this in the sequel paper \cite{classical-locus}; in particular, the above expectation is generalised to arbitrary parabolics in Conjecture C \emph{op.\ cit}., and substantial results towards it are proved in Theorem D. As a special case, Expectation \ref{conj:shalika = spin} holds for non-critical slope refinements of regular weight.
\end{remark}

	\subsection{Explicit eigenvectors for principal series}\label{ss:intertwining}

If $\pi$ is regular, then $\tilde\pi = (\pi,\alpha)$ is a line; so one can use any non-zero eigenvector $W \in \cS_\psi^\eta(\tilde\pi)$ to test the condition for $\tilde\pi$ be a Shalika refinement. For this, we must write down explicit $\alpha$-eigenvectors in $\cS_\psi^\eta(\tilde\pi)$. First, we consider $\alpha$-eigenvectors in principal series representations.

Recall that in \S\ref{sec:shalika conventions} we fixed an identification $\pi = \Ind_B^G\UPS$, for $\UPS$ chosen as in Ash--Ginzburg (i.e.\ $\UPS_i \UPS_{n+i} = \eta$ for all $i$). Recall  also that our choice of $\UPS$ fixes a bijection $\Delta_\UPS : \{p\text{-refinements of }\pi\} $ $\isorightarrow \cW_G$ (see \eqref{eq:delta ups}). For $\sigma \in \cW_G$, let $\tilde\pi_\sigma = (\pi, \alpha_\sigma) \defeq \Delta^{-1}_\UPS(\sigma)$; every $p$-refinement is of the form $\tilde\pi_\sigma$ for some $\sigma$.

For a general $\sigma \in \cW_G$, it turns out to be non-trivial to write down an explicit $\alpha_\sigma$-eigenvector in $\Ind_B^G \UPS$. However, it is very easy to write down such an eigenvector in the (different, but isomorphic) principal series representation $\Ind_B^G \UPS^\sigma$. In particular, let
\[
	\fsigma \in \Ind_B^G \UPS^\sigma
	\]
	be the unique function that is: 
\begin{itemize}\s
\item Iwahori-invariant, 
\item supported on the big Bruhat cell $B(\Qp)\cdot w_{2n} \cdot \Iw$, and 
\item normalised so that $\fsigma(w_{2n}) = 1$. 
\end{itemize}

	\begin{proposition}	\label{prop:W_p eigenvector}
		 $\fsigma \in \Ind_B^G\UPS^\sigma$ is an $\alpha_\sigma$-eigenvector; i.e.\ for $r = 1,...,2n-1$, we have
		\[
			U_{p,r}\fsigma = \alpha_\sigma(U_{p,r}) \fsigma.
		\]
	\end{proposition}
	
	\begin{proof}
		For such $r$, let $P_r$ be the maximal standard parabolic subgroup with Levi $\GL_r \times \GL_{2n-r}$, and let $J_r$ be the associated parahoric subgroup. Parahoric decomposition gives $J_r = N_{P_r}(\Zp)\cdot (P_r^-(\Zp)\cap J_r)$, where $N_{P_r} \subset P_r$ is the unipotent radical. Intersecting with $\Iw$ shows $\Iw = N_{P_r}(\Zp) \cdot (P_r^-(\Zp)\cap\Iw)$. As $t_{p,r}^{-1}(P_r^-(\Zp)\cap\Iw)t_{p,r} \subset (P_r^-(\Zp)\cap\Iw) \subset \Iw$, we deduce 
		\[
		\Iw t_{p,r} \Iw = N_{P_r}(\Zp)(P_r^-(\Zp)\cap\Iw) \cdot t_{p,r} \cdot \Iw = N_{P_r}(\Zp) \cdot t_{p,r} \cdot \Iw,
		\] 
		and in particular we can decompose the double coset into single cosets via
		\begin{equation}\label{eq:single cosets}
			\Iw t_{p,r} \Iw = \bigsqcup_{m \in M_{r,2n-r}(\Zp)/M_{r,2n-r}(p\Zp)} \smallmatrd{1_r}{m}{0}{1_{2n-r}} \cdot t_{p,r}\cdot \Iw.
		\end{equation}
		We have Bruhat decomposition 
		\begin{equation}\label{eq:bruhat}
			\GL_n(\Qp) = \bigsqcup_{\rho \in W_{n}} B_n(\Qp)\cdot \rho \cdot \Iwn,
		\end{equation}
		so it suffices to compute $U_{p,r}\fsigma(\rho)$ for $\rho \in \cW_G$. By \eqref{eq:single cosets}, we have
		\[
		U_{p,r}\fsigma(\rho) = \sum_{m \in M_{r,2n-r}(\Zp)/M_{r,2n-r}(p\Zp)}\fsigma\left(\rho\smallmatrd{1_r}{m}{0}{1_{2n-r}}t_{p,r}\right).
		\]
		
		\begin{claim}
			$\rho\smallmatrd{1_r}{m}{0}{1_{2n-r}}t_{p,r} \in B(\Qp) \cdot w_{2n} \cdot \Iw$ if and only if $\rho = w_{2n}$ and $m \in M_{r,2n-r}(p\Zp)$.		
		\end{claim}
		\emph{Proof of claim:} Let $t_{p,r}' = w_{2n}t_{p,r}w_{2n} \in T(\Qp)$. Then
		\[
		B(\Qp)w_{2n}\Iw = B(\Qp)t_{p,r}'w_{2n}\Iw = B(\Qp)w_{2n}t_{p,r}\Iw = B(\Qp)w_{2n}\Iw^{r,-}t_{p,r},
		\] 
		where $\Iw^{r,-} = t_{p,r}\Iw t_{p,r}^{-1}$. Thus
		\[
		\rho \smallmatrd{1_r}{m}{0}{1_{2n-r}} t_{p,r} \in B(\Qp)w_{2n}\Iw \ \Longleftrightarrow  \rho \smallmatrd{1_r}{m}{0}{1_{2n-r}} \in B(\Qp)w_{2n}\Iw^{r,-}.
		\]
		Conjugating \eqref{eq:bruhat} by $t_{p,r}$, we obtain $\GL_{2n}(\Qp) = \bigsqcup_{\rho' \in \cW_G} B(\Qp)\rho' \Iw^{r,-}$. It follows immediately that $\rho\smallmatrd{1_r}{m}{0}{1_{2n-r}}$ is in the cell $w_{2n}\Iw^{r,-}$, and the claim follows. $\qed$\\
		
		We return to Proposition \ref{prop:W_p eigenvector}. The claim implies $U_{p,r}\fsigma(\rho) = 0$ unless $\rho = w_{2n}$, and
		\begin{align*}
			U_{p,r}\fsigma(w_{2n}) &= \fsigma\left(w_{2n}t_{p,r}\right) = \fsigma\left(t_{p,r}'w_{2n}\right)\\
			&= \delta_B^{1/2}\theta^\sigma(t_{p,r}') \fsigma(w_{2n}) = (\delta_B^{1/2}\theta^\sigma)^{w_{2n}}(t_{p,r})\fsigma(w_{2n}) = \alpha_\sigma(U_{p,r}) \fsigma(w_{2n}),
		\end{align*}
		where the last equality is the equation for $\alpha_\sigma(U_{p,r})$ in Proposition \ref{prop:p-refinement}.
	\end{proof}

\subsection{Local Shalika models \`a la Ash--Ginzburg}\label{sec:AG}
Let $\sigma \in \cW_G$, corresponding to the $p$-refinement $\tilde\pi_\sigma = (\pi,\alpha_\sigma)$. We've written down explicit $\alpha_\sigma$-eigenvectors $\fsigma \in \Ind_B^G \UPS^\sigma$. We know (abstractly) that as $\GL_{2n}(\Qp)$-representations we have
\begin{equation}\label{eq:abstract isomorphisms}
\Ind_B^G \UPS^\sigma \cong \Ind_B^G \UPS = \pi \cong \cS_\psi^\eta(\pi);
\end{equation}
Since these isomorphisms are Hecke equivariant, by Proposition \ref{prop:W_p eigenvector} the image of $\fsigma$ under any choice of isomorphisms \eqref{eq:abstract isomorphisms} lies in $\cS_\psi^\eta(\tilde\pi_\sigma)$. This gives an eigenvector with which we can test, via \eqref{eq:shalika refinement}, if $\tilde\pi_\sigma$ is a Shalika refinement. To do this we must make \eqref{eq:abstract isomorphisms} explicit.

We first describe the right-hand isomorphism, recalling \cite{AG94}. When $\UPS_i \UPS_{n+i} = \eta$, Ash--Ginzburg defined and studied an explicit intertwining $\cS_\psi^\eta$ of $\Ind_B^G \UPS$ into its Shalika model. This intertwining is the  reason we choose $\UPS$ as in Ash--Ginzburg (see \S\ref{sec:shalika conventions}).

For $f\in \Ind_B^G \UPS$ and $g \in \GL_{2n}(\Qp)$, as in \cite[(1.3)]{AG94} we let 
\begin{equation}\label{eq:AG}
	\cS_\psi^\eta(f)(g) \defeq \int_{\GL_n(\Zp)}\int_{M_n(\Qp)} f\left[\smallmatrd{}{1}{1}{}\smallmatrd{1}{X}{}{1}\smallmatrd{k}{}{}{k}g\right] \psi^{-1}(\mathrm{tr}(X))\eta^{-1}(\det(k)) dXdk.
\end{equation}

Ash--Ginzburg show (before Lemma 1.4) that $\cS_\psi^\eta$ converges absolutely for $\UPS$ in a space $\Omega$.

\begin{proposition}\label{prop:AG 2}
If $\Ind_B^G \UPS$ is regular, 
 $\cS_\psi^\eta(f)$ can be analytically continued to a non-zero intertwining 
 \[
 	\cS_\psi^\eta : \Ind_B^G \UPS  \longhookrightarrow \cS_{\psi}^\eta(\Ind_B^G\UPS) = \cS_\psi^\eta(\pi).
 \] 
\end{proposition}

\begin{proof}
This is proved across Lemmas 1.4--1.6 of \cite{AG94}.  For any $g \in \GL_{2n}(\Qp)$, they consider $f$ varying with $\UPS$ in a flat family, and thus make sense of $\cS_\psi^\eta(f)(g)$ as a function of $\UPS$. They then define a quantity $P(\UPS)$ and show that $P(\UPS)\cS_\psi^\eta(f)(g)$ is polynomial in $\UPS$. Moreover $P(\UPS)$ is non-zero when $\Ind_B^G \UPS$ is regular. Thus $\cS_\psi^\eta(f)$ can be analytically continued from $\UPS \in \Omega$ to arbitrary $\UPS$, as claimed.
 
 The map is easily seen to be $\GL_{2n}(\Qp)$-equivariant. Finally $\cS_\psi^\eta(f) \in \cS_\psi^\eta(\Ind_B^G\UPS)$ (justifying the notation) using \cite[p.72--73]{BFG92}.
\end{proof}

	\subsection{Spin $p$-refinements under intertwining maps} \label{sec:spin intertwining}

In Expectation \ref{conj:shalika = spin}, we stated that we expect Shalika and spin refinements are equivalent. Let, then, $\tilde\pi_\sigma$ be a spin $p$-refinement. Recall from \S\ref{sec:shalika conventions} that we have identified $\pi = \Ind_B^G \UPS$, but that we could have replaced $\UPS$  by any of its conjugates under $\cW_G$; we now fix a choice.

\begin{lemma}\label{lem:choice of UPS}
There exists an unramified character $\UPS : T(\Qp) \to \C^\times$ such that:
\begin{itemize}\s
	\item $\pi = \Ind_B^G \UPS$, 
	\item $\UPS_i\UPS_{n+i} = \eta$ for $1 \leq i \leq n$,
	\item and $\Delta_\UPS(\tilde\pi) = \tau = \smallmatrd{1}{}{}{w_n}$.
\end{itemize}
\end{lemma}
\begin{proof}
	First fix any $\UPS$ with $\UPS_i\UPS_{n+i} = \eta$. By Lemma \ref {lem:Shalika W_G^0}, we know 
\[
\Delta_{\UPS^\tau}(\tilde\pi_\sigma) = \Delta_\UPS(\tilde\pi_\sigma) \cdot \tau = \sigma \tau \in \cW_G^0, 
\]
so that $\sigma \in \cW_G^0\tau.$ After replacing $\UPS$ with its conjugate by $\tau^{-1}(\sigma\tau)^{-1}\tau \in \tau^{-1}\cW_G^0\tau$, we have $\Delta_{\UPS}(\tilde\pi) = \tau$; and this preserves the the Ash--Ginzburg property $\UPS_i\UPS_{n+i} = \eta$ by Remark \ref{rem:AG vs AS}. 
\end{proof} 

Henceforth we fix $\UPS$ as in Lemma \ref{lem:choice of UPS}. Motivated by \eqref{eq:abstract isomorphisms} and the discussion around it, we now compute an explicit intertwining
\begin{equation}\label{eq:M_wn}
M_{w_n} : \Ind_B^G \UPS^\tau \isorightarrow \Ind_B^G\UPS,
\end{equation}
after which we can combine with Proposition \ref{prop:AG 2} to obtain an explicit eigenvector $\cS_\psi^\eta \circ M_{w_n}(f^\tau)$ in the spin $p$-refinement $\cS_\psi^\eta(\tilde\pi_\tau)$.

Since we will study this using an inductive argument, it is convenient to consider a slightly more general setting. Let $\rho \in \cW_n$, and $\nu_\rho \defeq \smallmatrd{1}{}{}{\rho}\in \cW_{G}$. There is an isomorphism 
	\[
	M_\rho: \Ind_B^G \UPS^{\nu_{\rho}} \isorightarrow \Ind_B^G \UPS = \pi, 
	\]
	which is unique up-to-scalar  by Schur's Lemma, and which we now make precise on the big cell eigenvector $f^{\nu_\rho}$. Note $\tau = \nu_{w_n}$, hence the notation $M_{w_{n}}$ in \eqref{eq:M_wn}.

	\begin{definition}\label{def:F_w}
	For each $w, \sigma \in \cW_G$:
\begin{itemize}\s
\item Let $f_w^\sigma\in \Ind_B^G \UPS^\sigma$ be the
	unique $\Iw$-invariant function supported on $B(\Qp)\cdot w \cdot \Iw$ such that
	$f_w^\sigma(w)= 1$. 
	
\item For $\delta \in \cW_n$ (and $\nu_\delta = \smallmatrd{1}{}{}{\delta}$ as above), let  
	\[
F_{\delta}^\sigma \defeq f_{\smallmatrd{}{w_n}{\delta w_n}{}}^\sigma =f_{\nu_{\delta} w_{2n}}^\sigma \in \Ind_B^G \UPS^\sigma.
\]

\item If $\sigma = 1$, we drop the superscript and just write $F_\delta = F_\delta^1 \in \Ind_B^G\UPS = \pi$. 
\end{itemize}
	
\end{definition}

Note $F_{1_n}^\sigma = f_{w_{2n}}^\sigma = f^\sigma$, the big cell eigenvector from Proposition \ref{prop:W_p eigenvector}.

	\begin{lemma} \label{prop:intertwining}
		After possibly rescaling $M_{\rho} : \Ind_B^G \UPS^{\nu_\rho} \to \pi$, we have
		\begin{equation}\label{eq:intertwining}
			M_{\rho} (f^{\nu_\rho}) = F_{\rho} +  	\sum_{ \ell(\delta)<\ell(\rho)}c_{\delta} F_{\delta}, \qquad \text{ with } c_\delta \in \C. 
		\end{equation}
	Here $\ell(\delta)$ denotes the Bruhat length of any $\delta \in\cW_n$.
	\end{lemma}

	\begin{proof} 
		For any simple reflection $s\in \cW_G$ and $\sigma \in \cW_G$, we have an intertwining isomorphism
		\[
		M^{s}_\sigma : \Ind_B^G \UPS^\sigma \longrightarrow \Ind_B^G \UPS^{s \sigma}.
		\]
		By \cite[Thm. 3.4]{Cas80} (see \cite{DJ-parahoric} for more details) this can be normalised so that for
		any  $\delta \in \cW_n$, there is a constant $c_{\delta,s} \in \C$ (depending also on $\theta$) such that 
		\[
		M^{s}_\sigma(F_\delta^\sigma)= \begin{cases} F_{s \delta}^{s \sigma} +c_{\delta, s}F_{\delta}^{s \sigma}  &:  \ell(s\delta)=\ell(\delta)+1, \\ 
			p^{-1} F_{s \delta}^{s \sigma} +c_{\delta, s}F_{\delta}^{s\sigma}   &: \ell(s\delta)=\ell(\delta)-1.\end{cases}
		\]
		
		Writing $\rho = s_1\cdots s_r$,  $M_{\rho}$ is the composition 
		\[
			M_\rho = M^{s_r}_{s_{r-1}\cdots s_1 \rho} \circ\cdots \circ M^{s_1}_\rho.
		\]
		 The lemma is then obtained by induction on $\ell(\rho)$ via basic properties of Bruhat length. 
	\end{proof}

We finally map into the Shalika model. For $\delta \in \cW_n$, recall $\cS_\psi^\eta$ from \eqref{eq:AG}, and let 
\[
W_\delta \defeq \cS_\psi^\eta(F_\delta) \in \cS_\psi^\eta(\pi).
\]
 This is well-defined, since $F_\delta \in \Ind_B^G\UPS = \pi$ and we chose $\UPS$ as in Ash--Ginzburg.

\begin{proposition}\label{prop:spin eigenvector}
 Let $\tilde\pi = (\pi, \alpha)$ be a regular spin $p$-refinement, and write $\pi = \Ind_B^G\UPS$ as in Lemma \ref{lem:choice of UPS}. The $\alpha$-eigenspace $\cS_\psi^\eta(\tilde\pi)$ is spanned by an eigenvector of the form
\begin{equation}\label{eq:spin eigenvector}
	W_0 = W_{w_n} + \sum_{w_n \neq \delta \in \cW_n} c_\delta W_\delta \qquad \in \cS_\psi^\eta(\tilde\pi).
\end{equation}
\end{proposition}
	\begin{proof}
By Proposition \ref{prop:W_p eigenvector}, we know $f^\tau \in \Ind_B^G\UPS^\tau$ is an $\alpha$-eigenvector. By Lemma \ref{prop:intertwining}, its image in $\Ind_B^G$ under the intertwining $M_{w_n}$ has the form
\[
	M_{w_n}(f^\tau) = \bigg[F_{w_n} + \sum_{w_n \neq \delta \in \cW_n} c_\delta F_\delta\bigg] \ \ \in \Ind_B^G \UPS = \pi.
\]
By definition the image of this under $\cS_\psi^\eta$ is $W_0$ (from \eqref{eq:spin eigenvector}), which hence gives a non-zero eigenvector in $\cS_\psi^\eta(\tilde\pi)$. By regularity this space is a line, so this eigenvector spans.
\end{proof}

\subsection{Non-vanishing of the local zeta integral at Iwahori level}\label{sec:local zeta integrals}

From Proposition \ref{prop:spin eigenvector} we've written down an explicit eigenvector $W_0 \in \cS_\psi^\eta(\pi)$ in any spin $p$-refinement. We now deduce that all spin $p$-refinements are Shalika $p$-refinements, by showing $W_0\smallmatrd{w_nz^{2\beta}}{}{}{1} \neq 0$, recalling $z = \operatorname{diag}(p^{n-1}, p^{n-2}, \dots, 1)$. We also compute this value exactly, and hence -- via Proposition \ref{prop:local zeta p} -- complete the computation of the local zeta integral for $W_0$. 

We first show that $W_\delta\smallmatrd{w_nz^{2\beta}}{}{}{1} = 0$ unless $\delta = w_n$ (which means, by Proposition \ref{prop:local zeta p}, that the local zeta integral vanishes for $W_\delta$ unless $\delta = w_n$). To do so, we examine when the integrand of \eqref{eq:AG} lies in the support of $F_\delta \in \Ind_B^G(\UPS)$.

\begin{proposition}\label{prop:shalika cell support}
	Let $\delta \in \cW_n$, $X \in M_n(\Qp)$ and $k \in \GL_n(\Zp)$, and let $\beta \geq 1$. Then
		\begin{equation}\label{eq:support inclusion}
			\tbyt{}{1}{1}{} \tbyt{1}{X}{}{1} \tbyt{k}{}{}{k} \tbyt{w_n z^{2\beta}}{}{}{1} \in B_n(\Qp)\tbyt{}{w_n}{\delta w_n}{} \Iw
		\end{equation}
if and only if:
\begin{itemize}\s
	\item $\delta = w_n$ is the longest Weyl element,
	\item $k \in B_n(\Zp) w_n \Iwn$ and 
	\item $k^{-1}X \in w_nz^{2\beta}M_n(\Zp)$.
\end{itemize}

\end{proposition}

\begin{proof}
 Suppose \eqref{eq:support inclusion} holds, and write
\begin{align}\label{eq:support equality}
\tbyt{}{1}{1}{} \tbyt{1}{X}{}{1} \tbyt{k}{}{}{k} &\tbyt{w_n z^{2\beta}}{}{}{1}\\ 
	&= \tbyt{A}{B}{}{D} \tbyt{}{w_n}{\delta w_n}{} \tbyt{a}{b}{c}{d},\notag
\end{align}
where
\[
A,D \in B_n(\Qp), \ \ B \in M_n(\Qp), \ \ a,d \in \Iwn, \ \ c \in pM_n(\Zp), \ \ b \in M_n(\Zp).
\]
Expanding this out, we get the equality of matrices:
\[
\tbyt{}{k}{k w_n z^{2\beta}}{X k} = \tbyt{B \delta w_n a + A w_n c}{B \delta w_n b + A w_n d}{D \delta w_n a}{D \delta w_n b} .
\]
This implies the following:
\begin{enumerate}[(1)]\s
    \item $B = -Aw_n c a^{-1} w_n^{-1} \delta^{-1}$ (from the top left entry).
    \item $-Aw_n (c a^{-1} b - d ) = k$ ((1) and top right), whence $A \in B_n(\Z_p)$ and 
    \[
    k \in B_n(\Z_p) \cdot w_n \cdot \Iwn.
    \]
\item $D \delta w_n a = k w_n z^{2\beta}$ (bottom left) which implies 
\[
k^{-1}D = w_nz^{2\beta}a^{-1}w_n\delta^{-1} \in w_nz^{2\beta}\GL_n(\Zp).
\]
    \item $D \delta w_n bk^{-1} = X$ (bottom right), which by (3) implies
\[
k^{-1} X = k^{-1}D\cdot \delta w_nbk^{-1} \in w_nz^{2\beta}M_n(\Zp).
\]

\end{enumerate}

We treat the cases $\delta \neq w_n$ and $\delta = w_n$ separately.

\medskip

\textbf{Case 1: $\delta \neq w_n$.} Suppose there exist $X$ and $k$ such that \eqref{eq:support inclusion} holds. We will derive a contradiction. 

For $r \in \{1,...,2n-1\}$, let $P_r$ be the parabolic with Levi $\GL_r \times \GL_{n-r}$, with associated (opposite) parahoric subgroup $\overline{J}_r \subset \GL_{2n}(\Zp)$. 

\begin{claim}
If $\delta \neq w_n$, there exists $r \in \{1,...,2n-1\}$ such that
\[
B_n(\Q_p) \cdot \delta w_n \cdot \overline{J}_r \cap \; B_n(\Q_p) \cdot \overline{J}_r = \varnothing.
\]
\end{claim}

\emph{Proof of claim:} Let $r \defeq \mathrm{min}(i : \delta w_n(i) \neq i)$, which exists since $\delta w_n \neq 1$. Let $\cW_{r,n-r}$ be the Weyl group of $\GL_r\times \GL_{n-r}$; it is the subgroup of $\cW_n$ that preserves both $\{1,...,r\}$ and $\{r+1,...,n\}$. In particular, $\delta w_n \notin \cW_{r,n-r}$.

We have the opposite Bruhat decomposition
\[
G(\Qp) = \bigsqcup_{\sigma \in \cW_n/\cW_{r,n-r}} B(\Qp)\sigma\overline{J}_r.
\]
Since $\delta w_n \notin \cW_{r,n-r}$, the cells $B(\Qp)\delta w_n \overline{J}_r$ and $B(\Qp)\overline{J}_r$ are disjoint, giving the claim. 

\bigskip

For $r$ as in the claim, let $\mu$ be such that $z^{2\beta} = t_{p,r} \cdot \mu$, recalling $t_{p,r} = \operatorname{diag}(p, \dots, p,1, \dots, 1)$ with $r$ lots of $p$. We see that 
\[
t_{p,r} \operatorname{Iw}_n t_{p,r}^{-1} \subset \overline{J}_r.
\]
Note that the valuation of $\mu$ under any positive root is non-negative, so $\mu^{-1} \overline{N}_n(\Z_p) \mu \subset \overline{N}_n(\Z_p)$, where $\overline{N}_n$ is the lower triangular unipotent. 

We now analyse (3) from the list above. Multiply both sides by $t_{p,r}^{-1}$ to get
\[
D \delta w_n a t_{p,r}^{-1} = k w_n \mu .
\]
We know from (2) that we can write $k = \alpha \cdot w_n \cdot \beta,$ 
with $\alpha \in B_n(\Z_p)$ and $\beta \in \Iwn$. We have
\[
k w_n \mu = \alpha w_n \beta w_n \mu \in B_n(\Z_p) \cdot \overline{\Iwn} \cdot \mu .
\]
But by the Iwahori decomposition, we see that
\[
B_n(\Z_p) \cdot \overline{\Iwn} \cdot \mu \subset B_n(\Q_p) \cdot \overline{\Iwn} \subset B_n(\Q_p) \cdot \overline{J}_r.
\]
We also have
\[
kw_n\mu = D \delta w_n a t_{p,r}^{-1} \in B_n(\Q_p) \cdot \delta w_n t_{p,r}^{-1} \cdot  \overline{J}_r \subset B_n(\Q_p) \cdot \delta w_n \cdot \overline{J}_r
\]
as the element $\delta w_n$ normalises the torus.

We must therefore have
\[
kw_n\mu \in \left( B_n(\Q_p) \cdot \delta w_n \cdot \overline{J}_r \right) \cap \left( B_n(\Q_p) \cdot \overline{J}_r \right) \neq \varnothing,
\] 
a contradiction. In particular, there do not exist $X$ and $k$ such that \eqref{eq:support inclusion} holds if $\delta \neq w_n$.

\bigskip

\textbf{Case 2: $\delta = w_n$}. We have shown that if \eqref{eq:support inclusion} holds, then 
\begin{equation}\label{eq:k and X}
k \in B(\Zp) w_n \Iwn \qquad \text{and} \qquad k^{-1}X \in w_nz^{2\beta}M_n(\Zp).
\end{equation}
Conversely, suppose \eqref{eq:k and X}, and write $k = A w_n d$, with $A \in B_n(\Zp)$ and $d \in \Iwn$. Via the Iwahori decomposition, we may assume $d \in N_n(\Zp)$ is upper unipotent. 

    For $A$ and $d$ as above, set $B = 0$ and $D = A z^{2\beta} \in B_n(\Q_p)$. Also set $c = 0$ and $a = z^{-2\beta} w_n d w_n z^{2\beta} \in \Iwn$ (since $\beta \geq 1$). If we set $b = D^{-1}Xk \in M_n(\Qp)$, then \eqref{eq:support equality} holds. Clearly $\smallmatrd{A}{B}{0}{D} \in B_n(\Qp)$, so we are done if we can show $\smallmatrd{a}{b}{c}{d} \in \Iw$. This will hold if $b \in M_n(\Zp)$. Observe
    \begin{align*}
    	b = D^{-1}Xk = z^{-2\beta}A^{-1}Xk &= z^{-2\beta}w_nd k^{-1} Xk\\
    	&\in z^{-2\beta}w_nd \cdot w_nz^{2\beta}M_n(\Zp)k = aM_n(\Zp)k \subset M_n(\Zp),
    \end{align*}
where in the second step we substitute $D$, the third we substitute $A^{-1} = w_ndk$, in the third we use \eqref{eq:k and X}, and in the fourth we substite $a = z^{-2\beta}w_ndw_nz^{2\beta}$. Thus $\smallmatrd{a}{b}{c}{d} \in \Iw$, completing the proof.
\end{proof}

From the proof, we also see the following:

\begin{corollary}\label{cor:values when delta = w}
    Let $\Theta$ be any unramified character of $T(\Qp)$, extended trivially to $B(\Qp)$. If we have
    \begin{align}\label{eq:control borel}
    \matrd{}{1}{1}{}\matrd{1}{X}{}{1}\matrd{k}{}{}{k}\matrd{w_nz^{2\beta}}{}{}{1} = \cB\matrd{}{w_n}{1}{}\cI
    \in B_n(\Qp)\matrd{}{w_n}{1}{} \Iw,
    \end{align}
    then
    \[
    \Theta(\cB) = \Theta\matrd{1}{}{}{z^{2\beta}}.
    \]
\end{corollary}

\begin{proof}
   The proposition gave necessary and sufficient conditions for \eqref{eq:control borel}. When they hold, we wrote down explicit values of $\cB = \smallmatrd{A}{}{}{D}$ and $\cI = \smallmatrd{a}{b}{}{d}$. By definition $A \in B_n(\zp)$ and $D = Az^{2\beta}$. The result follows easily.
\end{proof}

We now put this all together. Recall that (without loss of generality, as at the start of \S\ref{sec:spin intertwining}) we have chosen $\UPS$ so that identification $\Delta_\UPS : \{\text{$p$-refinements}\} \isorightarrow \cW_G$ sends our fixed spin $p$-refinement $\tilde\pi$ to $\tau = \smallmatrd{1}{}{}{w_n}$. By Proposition \ref{prop:p-refinement}, the Hecke eigenvalue of $U_{p,r}$ on $\tilde\pi$ is
\[
	\alpha_{p,r} = \Big[\delta_B^{1/2}\UPS^\tau\Big]^{w_{2n}}(t_{p,r}).
\]
Note also that since $\tilde\pi$ is a spin $p$-refinement, by Definition \ref{def:spin-refinement} we have $\alpha_{p,n+s} = \eta(p)^s \alpha_{p,n-s}$  for $0 \leq s \leq n-1$. We also have the $U_p$-eigenvalue $\alpha_p = \alpha_{p,1} \cdots \alpha_{p,2n-1}$.

\begin{proposition}\label{prop:spin refinements are shalika}
	\begin{enumerate}[(i)]
		\item If $\delta \neq w_n$, then $W_\delta\smallmatrd{w_nz^{2\beta}}{}{}{1} = 0$.
		\item We have 
		\begin{align*}
			W_0& \matrd{w_nz^{2\beta}}{}{}{1}= W_{w_n}\matrd{w_nz^{2\beta}}{}{}{1}\\ &
			= \Upsilon''  \cdot p^{\beta n^2} \cdot \frac{\delta_B(t_p)^\beta}{\eta(\det z^\beta)}\cdot\left(\frac{\alpha_p}{\alpha_{p,n}}\right)^\beta\neq 0,
		\end{align*}
	where $\Upsilon'' = \vol\big[B_n(\Z_p) \cdot w_n \cdot \Iwn \big]$ is a constant independent of $\beta$.
	\end{enumerate}

\end{proposition}

\begin{proof}
	In \eqref{eq:AG}, we gave an integral representation of 
	\begin{align*}
		W_\delta&\smallmatrd{w_nz^{2\beta}}{}{}{1} = \cS_{\psi}^\eta(F_\delta)\smallmatrd{w_n z^{2\beta}}{}{}{1}\\
		&=  \int_{\GL_n(\Zp)}\int_{M_n(\Qp)} F_{\delta}\left[ \smallmatrd{}{1}{1}{}\smallmatrd{1}{X}{}{1}\smallmatrd{k}{}{}{k}\smallmatrd{w_n z^{2\beta}}{}{}{1}\right] \psi^{-1}(\mathrm{tr}(X))\eta^{-1}(\det(k)) dXdk.
	\end{align*}

	\medskip
	
	(i) If $\delta \neq w_n$, then Proposition \ref{prop:shalika cell support} shows the domain of $F_\delta$ in the integral is disjoint from the support of $F_\delta$, hence the integrand (thus the integral) vanishes.
	
	\medskip
	
	(ii) The first equality in (ii) follows from \eqref{eq:spin eigenvector} and (i). Now, if $\delta = w_n$, then Proposition \ref{prop:shalika cell support} means we can restrict the domain of the integral to $k \in B_n(\Zp)w_n\Iwn$ and $X \in kw_nz^{2\beta}M_n(\Zp)$. Moreover:
	\begin{itemize}\s
		\item If $k \in B_n(\Qp)w_n\Iwn$ and $X  \in kw_nz^{2\beta}M_n(\zp)$, then
		\begin{align*}
			F_{w_n}\left[ \smallmatrd{}{1}{1}{}\smallmatrd{1}{X}{}{1}\smallmatrd{k}{}{}{k}\smallmatrd{w_nz^{2\beta}}{}{}{1}\right] &=  (\delta_B^{1/2}\UPS)\Big[\smallmatrd{1}{}{}{z^{2\beta}}\Big]F_{w_n}\smallmatrd{}{w_n}{1}{}\\
			&=  (\delta_B^{1/2}\UPS)\Big[\smallmatrd{1}{}{}{z^{2\beta}}\Big],
		\end{align*}
		using Corollary \ref{cor:values when delta = w}, Iwahori-invariance of $F_{w_n}$, and Definition \ref{def:F_w}.
		
		\item Since $X \in M_n(\Zp)$, we have $\psi(\mathrm{tr}(X)) = 1$.
		
		\item $\eta(\det(k)) = 1$ as $\eta$ is unramified.
	\end{itemize}
	In particular, the integral collapses to
	\begin{align}\notag
		W_{w_n}&\smallmatrd{w_nz^{2\beta}}{}{}{1} = (\delta_B^{1/2}\UPS)\Big[\smallmatrd{1}{}{}{z^{2\beta}}\Big] \cdot \int_{B_n(\Zp)w_n\Iwn}\int_{kw_nz^{2\beta}M_n(\Zp)} dXdk\\
		&= (\delta_B^{1/2}\UPS)\Big[\smallmatrd{1}{}{}{z^{2\beta}}\Big]\cdot \operatorname{vol}\Big(B_n(\Z_p) \cdot w_n \cdot \Iwn \Big) \operatorname{vol}\Big(z^{2\beta}M_n(\zp)\Big),\label{eq:F_1 value 1}
	\end{align}
	where in the last step we made the change of variables $X \mapsto w_nk^{-1}Xkw_n$ before integrating. We easily see that
	\begin{equation}\label{eq:F_1 value 2}
		\operatorname{vol}\Big(z^{2\beta}M_n(\zp)\Big) = p^{\beta(n^2-n^3)} = p^{\beta n^2} \cdot \delta_B\smallmatrd{p^{n\beta}}{}{}{1}.
	\end{equation}
	
	Now, note that $w_{2n}\tau\smallmatrd{1}{}{}{z}\tau^{-1}w_{2n}^{-1} = \smallmatrd{z}{}{}{1} = t_{p,1}\cdots t_{p,n-1}$ (from \eqref{eq:t_p i}). In particular,
	\begin{align}
		(\delta_B^{1/2}\theta)\Big[\smallmatrd{1}{}{}{z^{2\beta}}\Big] &= 
		 \delta_B^{1/2}\Big[\smallmatrd{1}{}{}{z^{2\beta}}\Big] (\theta^\tau)^{w_{2n}}\Big[t_{p,1}^{2\beta} \cdots t_{p,n-1}^{2\beta}\Big]\notag \\
		 &= 	 \delta_B^{1/2}\Big[\smallmatrd{1}{}{}{z^{2\beta}}\Big](\delta_B^{-1/2})^{w_{2n}}\Big[t_{p,1}^{2\beta}\cdots t_{p,n-1}^{2\beta}\Big] \alpha_{p,1}^{2\beta} \cdots \alpha_{p,n-1}^{2\beta}\notag\\
		&= 	 \delta_B^{1/2}\Big[\smallmatrd{z^{2\beta}}{}{}{z^{2\beta}}\Big] \alpha_{p,1}^{\beta} \cdots \alpha_{p,n-1}^{\beta}\cdot [\eta(p)\alpha_{p,n+1}]^\beta \cdots [\eta(p)^{n-1}\alpha_{p,2n-1}]^\beta\notag\\
		 &= \eta(\det z^\beta)^{-1} \delta_B\Big[\smallmatrd{z^\beta}{}{}{z^{\beta}}\Big]\frac{\alpha_p^\beta}{\alpha_{p,n}^\beta}.\label{eq:F_1 value 3}
	\end{align}
Here in the second equality we use that $\delta_B^{w_{2n}} = \delta_B^{-1}$, in the third we use the definition of spin refinements (Definition \ref{def:spin-refinement}), and in the fourth that $\eta(p)\cdot \eta(p)^2 \cdots \eta(p)^{n-1} = \eta(\det z)$ and $\alpha_{p,1}\cdots\alpha_{p,n-1}\cdot\alpha_{p,n+1}\cdots\alpha_{p,2n-1} = \alpha_p/\alpha_{p,n}.$ Finally, we obtain (ii) after combining \eqref{eq:F_1 value 1}, \eqref {eq:F_1 value 2} and \eqref{eq:F_1 value 3}, as
\[
\delta_B(t_p)^\beta = \delta_B\smallmatrd{p^{n\beta}}{}{}{1}\delta_B\smallmatrd{z^\beta}{}{}{z^\beta}. \qedhere
\]
\end{proof}

\begin{corollary}
\label{cor:spin refinements are shalika}
Any regular spin $p$-refinement is a Shalika $p$-refinement.
\end{corollary}
\begin{proof}
In Proposition \ref{prop:spin refinements are shalika}, we started with an arbitrary regular spin $p$-refinement $\tilde\pi$, and exhibited in (ii) an eigenvector $W_0 \in \cS_\psi^\eta(\tilde\pi)$ with $W_0\smallmatrd{w_n z^{2\beta}}{}{}{1}\neq 0$. By definition, $\tilde\pi$ is thus a Shalika refinement.
\end{proof}

	\section{Interlude: an automorphic summary}\label{sec:running assumptions}
	
	Everything so far has been classical/automorphic in nature. The rest of the paper is concerned with a $p$-adic interpolation of the previous sections. For the benefit of the reader, we now collect our running automorphic assumptions in one place, and summarise our main classical results.
	
	Throughout the rest of the paper, we work with base field $\Q$, fix $K = \Iw \prod_{\ell\neq p}\GL_{2n}(\Z_\ell)$, and let $\pi$ be a RACAR of $\GL_{2n}(\A)$ of weight $\lambda$ such that: 
	
	\begin{conditions}\label{cond:running assumptions}
		\begin{itemize}\setlength{\itemsep}{0pt}
			\item[(C1)] $\pi$ admits a global $(\eta,\psi)$-Shalika model, for a Hecke character $\eta$;
			\item[(C2)] $\pi_{p}$ is spherical and admits a regular spin $p$-refinement $\tilde\pi_{p} = (\pi_{p}, \alpha)$;
			\item[(C3)] for each $\ell \neq p$, $\pi_\ell$ is spherical;
			\item[(C4)] for each $r = 1,...,2n-1$, letting $\alpha_{p,r} = \alpha(U_{p,r})$ and $\alpha_{p,r}^\circ = \lambda(t_{p,r})\alpha_{p,r}$, we have 
			\[
				v_p(\alpha_{p,r}^\circ) = \Big[v_p(\alpha_{p,r}) - \sum_{j=1}^r \lambda_{2n+1-j}\Big] < \lambda_r - \lambda_{r+1} + 1.
			\]
			We also write $\alpha_p = \alpha_{p,1} \cdots \alpha_{p,2n-1} = \alpha(U_p)$ and $\alpha_p^\circ = \alpha_{p,1}^\circ \cdots \alpha_{p,2n-1}^\circ$.
			
		\end{itemize}
	\end{conditions}
	In this case we write $\tilde\pi = (\pi,\alpha)$ and call it a \emph{$p$-refined RACAR satisfying (C1-4)}. 
	
	(C2) and (C3) imply that $\eta$ is everywhere unramified, so $\eta = |\cdot|^{\sw}$, where $\sw$ is the purity weight of $\lambda$. (C4) ensures $\tilde\pi$ is a \emph{non-critical slope} $p$-refinement, which we will explain in \S\ref{sec:nc slope}.
	
	If $\tilde\pi$ satisfies (C1-4), choose 
	\begin{equation}\label{eq:W_f}
		W_f = \otimes_\ell W_\ell \in \cS_{\psi_f}^{\eta_f}(\pi_f, E)
	\end{equation}
	as follows: for each $\ell\neq p$, let $W_\ell = W_\ell^\circ$ be the spherical test vector, and at $p$, let $W_p$ be the vector $W_0 \in \cS_{\psi_p}^{\eta_p}(\pi_p,E)$ from Proposition \ref{prop:spin eigenvector}. Define 
	\begin{equation}\label{eq:phi}
		\phi_{\tilde\pi}^\pm \defeq \Theta^\pm(W_f)/\Omega^\pm_\pi \in \hc{t}(S_K,\sV_\lambda^\vee(L))^\pm_{\tilde\pi},
	\end{equation}
	where $\Theta^\pm/\Omega_\pi^\pm$ is defined in \eqref{eq:Theta}, and $(-)_{\tilde\pi}$ is the localisation at $\m_{\tilde\pi}$ (see \S\ref{sec:hecke p}). Now recap:
	\begin{itemize}\s
		\item In Theorem \ref{thm:critical value}, we showed that for $\chi$ of conductor $p^{\beta'}$ and $\beta = \mathrm{max}(\beta',1)$,
		\[
			\cE_{B,\chi}^{j,\eta_0}\left(\phi_{\tilde\pi}^\pm\right) = \delta_B(t_B^{-\beta})\cdot\Upsilon_B\cdot\lambda(t_B^\beta)\cdot \zeta_j(W_\infty^\pm) \cdot \frac{L^{(p)}(\pi\otimes\chi,j+1/2)}{\Omega_\pi^\pm}\cdot \zeta_p\Big(j+\tfrac{1}{2},(u^{-1}t_B^\beta)\cdot W_p,\chi_p\Big).
		\]
		\item In Proposition \ref{prop:local zeta p}, we showed that if $\chi_p$ is ramified,
		\[
		\zeta_p\Big(j+\tfrac{1}{2},(u^{-1}t_p^{\beta}\Big) \cdot W_p, \chi_p) = \Upsilon' \cdot \eta(\det z^\beta) \cdot p^{-\beta \tfrac{n^2+n}{2}} \cdot p^{\beta nj} \cdot \tau(\chi)^n  \cdot \chi_p(\det(-w_n))  \cdot W_p\tbyt{w_n z^{2\beta}}{}{}{1}.
		\]
		
		\item In Proposition \ref{prop:spin refinements are shalika}, since $\tilde\pi$ is a spin $p$-refinement, we showed
		\[
		W_p\matrd{w_nz^{2\beta}}{}{}{1}= \Upsilon''  \cdot p^{\beta n^2} \cdot \frac{\delta_B(t_p)^\beta}{\eta(\det z^\beta)}\cdot\left(\frac{\alpha_p}{\alpha_{p,n}}\right)^\beta.
		\]
		Note that we showed this by making specific `good' identifications, namely fixing $\pi = \Ind_B^G \UPS$ for an appropriate $\UPS$ (chosen in Lemma \ref{lem:choice of UPS}) and then using the Shalika intertwining of \eqref{eq:AG}. However, by Multiplicity One for Shalika models \cite{Nia09, CS20}, we immediately deduce this for all regular spin $p$-refinements $\tilde\pi$ independent of any choices.
	\end{itemize}
Let $\gamma_{(pm)} \defeq \Upsilon_B \cdot \Upsilon' \cdot \Upsilon''$ (recall that $\Upsilon_B$ depends on the constant $\gamma$ appearing in \cite[(77)]{DJR18}, which in turn depends on the fixed integer $m$ in Definition \ref{def:auto cycle}). We deduce:

	\begin{corollary}\label{cor:L-value}
		Let $\chi$ be a finite order Hecke character of conductor $p^{\beta},$ with $\beta \in \Z_{\geq 1}$, and let $j \in \mathrm{Crit}(\lambda)$. If $(-1)^j\chi_\infty\eta_\infty(-1) = \mp 1$, then $\cE_{\chi}^{j,\eta_0}(\phi_{\tilde{\pi}}^\pm) = 0$. If $(-1)^j\chi_\infty\eta_\infty(-1) = \pm 1$, we have
		\[
		\cE_{B,\chi}^{j,\eta_0}\left(\phi_{\tilde\pi}^\pm\right) = \gamma_{(pm)}\cdot\chi_p(\det(-w_n)) \cdot \left(\frac{\alpha_p^\circ}{\alpha_{p,n}}\right)^\beta \cdot Q'(\pi,\chi,j)  \cdot
		\zetainfty \cdot \frac{L^{(p)}(\pi\otimes\chi,j+1/2)}{\Omega_\pi^\pm},
		\]
		where
		\[
			Q'(\pi,\chi,j) \defeq  p^{\beta\left(nj + \tfrac{n^2-n}{2}\right)}  \tau(\chi)^n.
		\]
	\end{corollary}
	\begin{proof}
		We use $\alpha_p^\circ = \lambda(t_p)\alpha_p$. The rest is book-keeping.
	\end{proof}
	
	Finally, we record one more relevant result:
	\begin{proposition}\label{prop:mult one}
		The $L$-vector space $\hc{t}(S_K,\sV_\lambda^\vee(L))^\pm_{\tilde\pi}$ is 1-dimensional, generated by $\phi_{\tilde\pi}^\pm$.
	\end{proposition}
	\begin{proof}
		The generalised eigenspace in $\pi_f^K$ where $\cH$ acts by $\psi_{\tilde\pi}$ is 1-dimensional; locally, at $\ell\neq p$ we have $\pi_\ell^{G(\Z_\ell)}$ is a line as $\pi_\ell$ is spherical, and at $p$, this is (C2). This line is generated by $W_f$ by construction. The result now follows from Hecke-equivariance of $\Theta^\pm$.
	\end{proof}

%%====================================================

\section{The local zeta integral at parahoric level}

\label{sec:local zeta unramified}

The local zeta integral we computed in Proposition \ref{prop:local zeta p} required the twisting character $\chi$ to be \emph{ramified}. This is similar to previous papers \cite{DJR18,BDW20} on this topic, where $p$-adic $L$-functions were only shown to have the required interpolation property at ramified characters. However, this excludes the trivial character, which is typically the most interesting one. We finish Part II by computing the local zeta integral again in a different way which allows us to also handle \emph{unramified} characters. Doing this at Iwahoric level appears to be very difficult. Instead, for this section only, we work at $Q$-\emph{parahoric} level, for $Q \subset \GL_{2n}$ the parabolic with Levi $\GL_n \times \GL_n$ (the setting treated in \cite{DJR18,BDW20}). This allows us to directly strengthen the results of \cite{DJR18,BDW20} (see Proposition \ref{prop:unramified interpolation} below). In \S\ref{sec:p-adic L-functions} we'll exploit the interpolation properties of $p$-adic $L$-functions to deduce the result at Iwahoric level, completing the present paper.

For compatibility with \cite{DJR18,BDW20}, we work over a general local field rather than just over $\Qp$. In particular, \emph{for this section only} we adopt the following notation.

\begin{notation}
		Let $F/\Qp$ be a finite extension, with ring of integers $\cO$ and maximal ideal $\pri = \varpi\cO$, and let $q \defeq \#\cO/\pri$. Let $\delta$ be such that $\varpi^\delta\cO$ be the different of $F/\Qp$. We let $J \subset \mathrm{GL}_{2n}(\mathcal{O})$ denote the parahoric subgroup associated with $Q$, i.e., all elements which land in $Q$ modulo $\mathfrak{p}$. 

		Let $\pi = \Ind_B^G \UPS$ be a generic unramified principal series representation of $\GL_{2n}(F)$ admitting an $(\eta,\psi)$-Shalika model, for $\eta : F^\times \to \C^\times$ a smooth character and $\psi : F \to \C^\times$ the usual additive character (e.g.\ \cite[\S4.1]{DJR18}). Note $\pi$ is spherical. We will assume that \eqref{eq:AG} gives a non-trivial intertwining $\cS_\psi^\eta : \Ind_B^G \UPS \hookrightarrow \cS_\psi^\eta(\pi)$ for an unramified character $\UPS$ in the Ash--Ginzburg convention (cf.\ Proposition \ref{prop:AG 2}).
		\end{notation}

	\subsection{Normalisation of local data}
	
Let $\tilde\pi = (\pi, \alpha)$ be a regular spin $p$-refinement of $\pi$ (to Iwahori level), normalised without loss of generality as in the proof of Proposition \ref{prop:spin eigenvector}. At parahoric level:
\begin{itemize}\s
	\item Rather than $\Iw$-invariant vectors we use $J$-invariant vectors.
	\item Attached to $\tilde\pi$ is a parahoric refinement $\tilde\pi^Q = (\pi,\alpha_{\pri,n})$, for 
	\begin{align}
		\alpha_{\pri,n} = \alpha(U_{\pri,n}) &= \Big[\delta_B^{1/2}\UPS\Big]^{w_{2n}}(t_{\pri,n})\notag\\
		&= \Big[\delta_B^{1/2}\UPS\Big]\smallmatrd{1_n}{}{}{\varpi 1_n} = q^{n^2/2}\UPS_{n+1}(\varpi)\cdots\UPS_{2n}(\varpi)\label{eq:unramified alpha}
	\end{align}
	(via Proposition \ref{prop:p-refinement}). This is a $Q$-regular $Q$-refinement in the sense of \cite[Def.\ 3.5]{DJR18}. (In the notation \emph{op.\ cit}., it corresponds to the set $\tau = \{n+1,...,2n\}$).
	\item For a vector $W_{\pri} \in \pi_{\pri}^J$, the relevant twisted local zeta integral arising in Theorem \ref{thm:critical value} is $\zeta_{\pri}(j+\tfrac{1}{2}, (u^{-1}t_{\pri, n}^\beta)\cdot W_{\pri}, \chi_{\pri})$.
\end{itemize}

\begin{definition}
	Let $F_0\in \Ind_B^G \UPS$ be the
	unique $J$-invariant function supported on $B(F)\cdot w_{2n}\cdot J$ such that
	\[
		F_0(w_{2n})= \alpha_{\pri,n}^\delta \cdot q^{-\delta n^2}.
		\]
		Let $W_0 = \cS_{\psi}^\eta(F_0) \in \cS_\psi^\eta(\pi)$ be the Shalika vector associated with $F_0$ by \eqref{eq:AG}. 
\end{definition}

This is the same choice of normalisation as \cite[\S3.3]{DJR18} (hence in \cite{BDW20}). By Lemma 3.6 of that paper, we have $W_0(t_{\pri,n}^{-\delta}) = 1$. In particular, $F_0 \in \tilde\pi^Q$ is non-zero, so generates this line.

\subsection{The local zeta integral}

Let $d^\times c$ be the Haar measure on $\cO^\times$ of total measure 1. Let $\chi : F^\times \to \C^\times$ be a finite order character of conductor $\pri^{\beta'}$, and denote its Gauss sum by 
\[
\tau(\chi) = \tau(\chi,\psi) = q^{\beta'}(1-q^{-1})\int_{\cO^{\times}} \chi(c\varpi^{-\beta'-\delta}) \psi(c\varpi^{-\beta'-\delta}) d^{\times} c.
\]
Our normalisation means that if $\chi$ arises from a Dirichlet character of $p$-power conductor, this recovers the usual Gauss sum \[	
\tau(\chi) = \sum_{c \in (\Z/p^{\beta'})^\times} \chi(c)e^{2\pi i c/p^{\beta'}}.
\]
 In the rest of this section, we prove:

\begin{proposition}\label{p:local-zeta}
	Let $\beta=\max(1, \beta')$. Then 
	\[
	\zeta\Big(s, (u^{-1}t_{\pri,n}^\beta) \cdot W_0, \chi\Big)
	=q^{\beta n\left(s-\tfrac{n}{2}\right)+\delta n\left(s - \tfrac{n}{2} - 1\right)} \cdot \chi(\det(-w_n)) \cdot Q(\pi,\chi,s)
	\]
	where 
	\[
	Q(\pi,\chi,s)=\left\{\begin{array}{cl} q^{-\beta n}\frac{q^n}{(q-1)^n} \cdot \tau(\chi)^n &: \chi \text{ ramified},\\
		\displaystyle{\frac{\chi(\varpi)^{-n(\delta+1)}}{(1-q)^n}\prod_{i=n+1}^{2n}
			\frac{1-\UPS_{i}\chi(\varpi)q^{1-s}}{1-\UPS_{i}\chi(\varpi)q^{-s}}} &: \chi \text{ unramified}.\end{array}\right.
	\]
\end{proposition}

\begin{proof}
	We first rewrite the local zeta integrals. For simplicity, we write $G_n = \GL_n(F)$, $B_n = B_n(F)$, and $M_n = M_n(F)$. Since $\eta$ is unramified, we have
	\begin{align*}
		\zeta(s, (u^{-1}t_{\pri,n}^\beta) \cdot W, \chi)&=\int_{G_n}W\left[\begin{pmatrix}
			h&\\&1\end{pmatrix}\begin{pmatrix} 1&-w_n\\&1\end{pmatrix} \begin{pmatrix}
			\varpi^\beta&\\&1\end{pmatrix}\right]\chi(\det h)|\det
		h|^{s-1/2}dh\\
		&\stackrel{\eqref{eq:AG}}{=}\int_{G_n}\int_{K_n}\int_{M_n}F_0\left[\begin{pmatrix}
			&1\\1&\end{pmatrix}\begin{pmatrix}
			1&X\\&1\end{pmatrix}\begin{pmatrix}
			kh&\\&k\end{pmatrix}\begin{pmatrix}
			1&-w_n\\&1\end{pmatrix}\begin{pmatrix}
			\varpi^\beta&\\&1\end{pmatrix}\right]\\& \hspace{60pt} \times \chi(\det h)|\det
		h|^{s-1/2}\eta^{-1}\big(\det k\big)\psi\big(-\tr X\big)\ dXdkdh\\
		&=\int_{G_n}\int_{K_n}\int_{M_n}F_0\left[\begin{pmatrix}
			&1\\1&\end{pmatrix}\begin{pmatrix}kh\varpi^\beta
			&\\&k\end{pmatrix}\begin{pmatrix}
			1&\varpi^{-\beta}(h^{-1}k^{-1}X k-w_n)\\&1\end{pmatrix}\right]\\&\hspace{60pt} \times \chi(\det h)|\det
		h|^{s-1/2}\psi\big(-\tr X\big) \ dXdkdh.
	\end{align*}
	Changing variables so that $kh\varpi^\beta$ becomes $h$ and $\varpi^{-\beta}(h^{-1}k^{-1}X k-w_n)$ becomes $X$ will not change $dh$ but changes $dX$ to $|\det h|^n dX$. The integral becomes 
	\begin{align*}
		&=\int_{G_n}\int_{K_n}\int_{M_n}F_0\left[\begin{pmatrix}
			&1\\1&\end{pmatrix}\begin{pmatrix}h
			&\\&k\end{pmatrix}\begin{pmatrix}
			1&X\\&1\end{pmatrix}\right]\\
		&\hspace{20pt}\times  \chi\Big(\det \varpi^{-\beta}k^{-1}h\Big)|\det
		\varpi^{-\beta}k^{-1}h|^{s-1/2}|\det h|^n\psi\Big(-\tr \big[hX k^{-1}+\varpi^{-\beta} hw_nk^{-1}\big]\Big)\ dXdkdh\\
		&=q^{\beta n(s-1/2)}\int_{G_n}\int_{K_n}\int_{M_n}F_0\left[\begin{pmatrix}
			&1\\1&\end{pmatrix}\begin{pmatrix}h
			&\\&k\end{pmatrix}\begin{pmatrix}
			1&X\\&1\end{pmatrix}\right]\\&\hspace{60pt} \times \chi\Big(\det \varpi^{-\beta}k^{-1}h\Big)|\det h|^{s+n-1/2}\psi\Big(-\tr \big[hX k^{-1}+\varpi^{-\beta} hw_nk^{-1}\big]\Big)\ dXdkdh
	\end{align*}
	
	Since $G_n=B_nK_n$ (with $B_n\cap K_n$ of volume 1) we may
	rewrite the above integral using $h=b\ell$ with $b\in B_n$ and
	$\ell \in K_n$ (so $|\det h|=|\det b|$):
	\begin{align*}
		&=q^{\beta n(s-1/2)}\int_{B_n}\int_{K_n}\int_{K_n}\int_{M_n}F_0\left[\begin{pmatrix}
			&1\\1&\end{pmatrix}\begin{pmatrix}b\ell
			&\\&k\end{pmatrix}\begin{pmatrix}
			1&X\\&1\end{pmatrix}\right]\\&\hspace{60pt} \times \chi\Big(\det
		\varpi^{-\beta}k^{-1}b\ell\Big)|\det b|^{s+n-1/2}\psi\Big(-\tr \big[b\ell X
		k^{-1}+\varpi^{-\beta} b\ell w_nk^{-1}\big]\Big)dXdkd\ell db\\
		&=q^{\beta n(s-1/2)}\int_{B_n}\int_{K_n}\int_{K_n}\int_{M_n}\delta_B^{1/2} \theta\left[\begin{pmatrix} 1&\\&b\end{pmatrix}\right] F_0\left[\begin{pmatrix}
			&1\\1&\end{pmatrix}\begin{pmatrix}\ell
			&\\&k\end{pmatrix}\begin{pmatrix}
			1&X\\&1\end{pmatrix}\right]\\&\hspace{60pt} \times  \chi\Big(\det
		\varpi^{-\beta}k^{-1}b\ell\Big)|\det b|^{s+n-1/2}\psi\Big(-\tr \big[b\ell X
		k^{-1}+\varpi^{-\beta} b\ell w_nk^{-1}\big]\Big)dXdkd\ell db  
	\end{align*}
	
	Using $J$-invariance of $F_0$, we see that 
	\[
	F_0\left[\begin{pmatrix}
		&1\\1&\end{pmatrix}\begin{pmatrix}\ell
		&\\&k\end{pmatrix}\begin{pmatrix}
		1&X\\&1\end{pmatrix}\right]=F_0\left[\begin{pmatrix}
		&1\\1&\end{pmatrix}\begin{pmatrix}
		1&\ell X k^{-1}\\&1\end{pmatrix}\right].
	\]
	The proof of \cite[Lemma 3.6]{DJR18} implies that this vanishes unless $X\in M_n(\mathcal{O})$, in which case it equals $F_0\left[\matrd{}{1}{1}{}\right]=F_0\left[w_{2n}\matrd{w_n}{}{}{w_n}\right] = F_0(w_{2n}) = \alpha_{\pri,n}^\delta q^{-\delta n^2}$. The above integral becomes
	\begin{align*}
		&=q^{\beta n(s-1/2)}\alpha_{\pri,n}^\delta q^{-\delta n^2}\int_{B_n}\int_{K_n}\int_{K_n}\delta_B^{1/2}\theta\left[\begin{pmatrix} 1&\\&b\end{pmatrix}\right]  \chi\Big(\det k^{-1}b\ell\varpi^{-\beta}\Big)\\
		&\hspace{60pt} \times |\det b|^{n+s-1/2}\psi\Big(-\tr \big[\varpi^{-\beta}b\ell w_n k^{-1}\big]\Big)\left(\int_{M_n(\mathcal{O})}\psi\Big(-\tr b\ell X k^{-1}\Big)dX\right)dkd\ell db
	\end{align*}
	Note that $\displaystyle \int_{M_n(\mathcal{O})}\psi(\tr AX)dX=0$
	unless $A\in \varpi^{-\delta}M_n(\mathcal{O})$, in which case the integral is $\vol(M_n(\mathcal{O}))=1$. We conclude that the above inside
	integral vanishes unless $k^{-1}b\ell\in
	\varpi^{-\delta}M_n(\mathcal{O})$, i.e., if $b\in
	B_n\cap \varpi^{-\delta}M_n(\mathcal{O})$. The integral becomes
	\begin{align*}
		&=q^{\beta n(s-1/2)} \alpha_{\pri,n}^\delta q^{-\delta n^2}\int_{B_n\cap \varpi^{-\delta}M_n(\mathcal{O})}\int_{K_n}\int_{K_n}\delta_B^{1/2}\theta\left[\begin{pmatrix} 1&\\&b\end{pmatrix}\right]\\
		&\hspace{80pt} \times \chi\Big(\det k^{-1}b\ell\varpi^{-\beta}\Big)|\det b|^{n+s-1/2}\psi\Big(-\tr \big[\varpi^{-\beta}b\ell w_nk^{-1}\big]]\Big)dkd\ell db.
	\end{align*}
	Changing variables so that $\ell w_n k^{-1}$ becomes $\ell$, and integrating out $k$, the integral becomes
	\begin{align*}
		&=q^{\beta n(s-1/2)} \alpha_{\pri,n}^\delta q^{-\delta n^2}\int_{B_n\cap \varpi^{-\delta}M_n(\mathcal{O})}\int_{K_n}\delta_B^{1/2}\theta\left[\begin{pmatrix} 1&\\&b\end{pmatrix}\right]\\
		&\hspace{80pt} \times \chi\Big(\det b\ell\varpi^{-\beta}w_n\Big)|\det b|^{n+s-1/2}\psi\Big(-\tr \big[\varpi^{-\beta}b\ell \big]]\Big)d\ell db.
	\end{align*}

	Since $G_n = \bigsqcup_{ \rho\in \cW_n}B_n  \rho \Iwahori_n$ we see $K_n=\bigsqcup_{\rho \in \cW_n} (B_n\cap K_n) \rho \Iwahori_n$. In fact, since $B_n \cap K_n = N_n(\cO)T_n(\cO)$ and $ \rho$ normalises $T_n(\cO)$, we may
	replace $B_n\cap K_n$ with the unipotent radical
	$N_n(\mathcal{O})$ of $B_n(\mathcal{O})$. Write
	$\ell = n \rho i$ with $n\in
	N_n(\mathcal{O})$, $ \rho\in \cW_n$ and $i\in \Iwahori_n$. Using the Iwahori decomposition $I_n = N_n^-(\varpi \mathcal{O}) B_n(\mathcal{O})$ we write $i = \overline{n}b_1$. The
	integral becomes
	\begin{align*}
		&=q^{\beta n(s-1/2)} \alpha_{\pri,n}^\delta q^{-\delta n^2} \sum_{\rho\in W_n}\int_{B_n\cap \varpi^{-\delta}M_n(\mathcal{O})}\int_{N_n(\mathcal{O})}\int_{N_n^{-}(\varpi \mathcal{O})}\int_{B(\mathcal{O})}\delta_B^{1/2}\theta\left[\begin{pmatrix} 1&\\&b\end{pmatrix}\right]\\
		&\hspace{80pt} \times \chi\Big(\det b\rho b_1\varpi^{-\beta}w_n\Big)|\det b|^{n+s-1/2}\psi\Big(-\tr \big[\varpi^{-\beta}b n \rho \overline{n}b_1 \big]]\Big)db_1d \overline{n} dn db.
	\end{align*}

	Changing variables so that $b_1bn$ becomes $b$ doesn't change the Haar measure, as $b_1\in B(\mathcal{O})$. Integrating out $n$ and $b_1$, the integral becomes
	\begin{align*}
		&=q^{\beta n(s-1/2)} \alpha_{\pri,n}^\delta q^{-\delta n^2}\sum_{\rho\in W_n}\chi(\det\rho w_{n}\varpi^{-\beta})\int_{B_n\cap \varpi^{-\delta}M_n(\mathcal{O})}\int_{N_n^{-}(\varpi \mathcal{O})}\delta_B^{1/2}\theta\left[\begin{pmatrix} 1&\\&b\end{pmatrix}\right]\\
		&\hspace{80pt} \times \chi\Big(\det b\Big)|\det b|^{n+s-1/2}\psi\Big(-\tr \big[\varpi^{-\beta}b  \rho \overline{n} \big]]\Big)d \overline{n}  db.
	\end{align*}
	
	Write 
	\[\overline{n}=\begin{pmatrix} 1&&0\\&\ddots&\\x_{ij}&&1\end{pmatrix}\ \ \ \ \ \ b=\varpi^{-\delta}\begin{pmatrix} t_1 & & u_{ij}\\
		&\ddots&\\
		0&& t_n\end{pmatrix},\]
	where $t_i\in \mathcal{O}\backslash\{0\}$,
	$u_{ij}\in \mathcal{O}$ and $x_{ij}\in \varpi\mathcal{O}$. In
	this case, $db = \prod|t_i|^{i-n-1}\prod dt_i\prod du_{ij}$ and $d \overline{n}=\prod dx_{ij}$.
	
	Fix $\rho\in W_n$. Writing $\rho\overline{n}=(m_{ij})$ we see that
	\[\tr b \overline{n}\rho = \varpi^{-\delta}\left(\sum_{i=1}^n
	m_{ii}t_i + \sum_{i<j}u_{ij}m_{ji}\right).\]
	If $\rho\neq 1$, there exist indices $i<j$ such that
	$m_{ji}=1$. Indeed, this is equivalent to there existing
	$j=\tau(i)>i$ where $\tau$ is the permutation of $\rho$. In this
	case, from the inner integral we may factor
	\[\int_{\mathcal{O}}\psi(-\varpi^{-\delta-\beta}m_{ji}u_{ij})du_{ij}=0,\]
	as $\beta\geq 1$.
	
	We conclude that all the terms of the sum vanish, except for the
	term $\rho=1$. We see that the zeta integral becomes 
	\begin{align*}
		&=q^{\beta n(s-1/2)} \alpha_{\pri,n}^\delta q^{-\delta n^2}\chi(\det w_{n})\int\cdots\int\prod_{i=1}^n\theta_{n+i}|\cdot|^{s-1}(\varpi^{-\delta}t_i)\chi(\varpi^{-\beta-\delta}t_i) \psi\Big(-\varpi^{-\beta-\delta}t_i \Big)\\
		&\hspace{60pt}\times \prod_{i<j}\psi\bigg[-\varpi^{-\delta-\beta}x_{ji}u_{ji}\bigg]\prod
		dx_{ij}dt_idu_{ij},
	\end{align*}
	where $x_{ij}\in \varpi \mathcal{O}, u_{ij}\in \mathcal{O},
	t_i\in \mathcal{O}\backslash\{0\}$.
	
	The integral $\displaystyle
	\int_{\mathcal{O}}\psi\bigg[\varpi^{-\delta-\beta}x_{ji}u_{ji}\bigg]d
	u_{ij}$ vanishes unless $x_{ji}\in \varpi^\beta \mathcal{O}$, in
	which case the integral is 1. Pulling out the resulting factor $\int \cdots \int \prod dx_{ij} =  \mathrm{vol}(N^-(\varpi^\beta\cO)) = q^{-\beta{n\choose 2}}$, the zeta integral is
	\begin{align*}
		&=q^{\beta n(s-1/2)-\beta\binom{n}{2}} \alpha_{\pri,n}^\delta q^{-\delta n^2}\chi(\det w_{n})\\
		&\hspace{60pt} \times \int\cdots\int\prod_{i=1}^n\theta_{n+i}|\cdot|^{s-1}(\varpi^{-\delta}t_i)\chi(\varpi^{-\beta-\delta}t_i) \psi\Big(-\varpi^{-\beta-\delta}t_i \Big)\prod
		dt_i.
	\end{align*}
	
	Now we consider separately the cases where $\chi$ is ramified or unramified.
	
	\bigskip
	
	\textbf{$\chi$ ramified:} When $\chi$ is ramified with conductor $\beta$ then
	\begin{align*}
		&\int\theta_{n+i}|\cdot|^{s-1}(\varpi^{-\delta}t_i)\chi(\varpi^{-\beta-\delta}t_i) \psi\Big(-\varpi^{-\beta-\delta}t_i \Big)dt_i\\
		&=\chi(-1)\sum_{k = 0}^\infty \UPS_{n+i}|\cdot|^{s-1}(\varpi^{k-\delta}) \int_{\cO^\times} \chi\Big(t_i\varpi^{k-\delta-\beta}\Big)\psi\Big(t_i\varpi^{k-\delta-\beta}\Big)dt_i\\
		&= \chi(-1) q^{\delta (s-1)}\cdot q^{-\beta }\frac{q}{q-1} \UPS_{n+i}(\varpi)^{-\delta}\tau(\chi),
	\end{align*}
	since the integral vanishes unless $k = 0$, whence it is the Gauss sum. Using $\chi(-1)^n\chi(\det w_n) = \chi(\det(-w_n))$, the zeta integral thus simplifies to
		\begin{align*}
		&q^{\beta n(s-1/2)-\beta\binom{n}{2}}\cdot \alpha_{\pri,n}^\delta q^{-\delta n^2}\cdot \chi(\det w_{n})\prod_{i=n+1}^{2n}\left[\chi(-1) \UPS_{i}(\varpi)^{-\delta}q^{\delta (s-1)}\cdot q^{-\beta }\frac{q}{q-1} \tau(\chi)\right]\\
		&= q^{\beta n(s-n/2)}\cdot \alpha_{\pri,n}^\delta  q^{-\delta n^2}\cdot \chi(\det(-w_n)) \cdot  q^{\delta n (s-1)} \cdot \bigg[\prod_{i=n+1}^{2n}\UPS_{i}(\varpi)^{-\delta}\bigg] \cdot Q(\pi, \chi, s).
		\end{align*}

Finally note that by \eqref{eq:unramified alpha}, we have $\prod_{i=n+1}^{2n}\UPS_i(\varpi)^{-\delta} = (q^{-n^2/2}\alpha_{\pri,n})^{-\delta}$. Simplifying yields
\begin{equation}\label{eq:finishing ramified}
	= q^{\beta n(s-n/2)} q^{\delta n(s - 1 -n/2)} \cdot\chi(\det(-w_n))\cdot Q(\pi,\chi,s),
\end{equation}
as required.

\medskip
		
	\textbf{$\chi$ unramified:}	If $\chi$ is unramified, then $\beta=1$ and, dropping indices for simplicity,
	\begin{align*}
		\int_{\mathcal{O}\backslash\{0\}}\theta|\cdot|^{s-1}&(t\varpi^{-\delta})\chi(t
		\varpi^{-\delta-1})\psi(-t\varpi^{-\delta-1})dt\\
		&=\chi(-1)\sum_{k=0}^\infty
		q^{-k-(k-\delta)(s-1)}
		\theta(\varpi)^{k-\delta}\chi(\varpi)^{k-\delta-1}\int_{\mathcal{O}^\times}\psi(t
		\varpi^{k-\delta-1})d^\times
		t\\
		&=\chi(-1)\frac{1}{1-q}q^{\delta(s-1)}\theta(\varpi)^{-\delta}\chi(\varpi)^{-\delta-1} + \sum_{k=1}^\infty
		q^{-k-(k-\delta)(s-1)}\theta(\varpi)^{k-\delta}\chi(\varpi)^{k-\delta-1}\\
		&=\chi(-1)q^{-1-(1-\delta)(s-1)}\theta(\varpi)^{1-\delta}\chi(\varpi)^{-\delta}\left(\sum_{k=0}^\infty
		\left(\frac{\theta
			\chi(\varpi)}{q^{s}}\right)^k+\frac{q^{s}\theta\chi(\varpi)^{-1}}{1-q}\right)\\
		&=\chi(-1)q^{-1-(1-\delta)(s-1)}\theta(\varpi)^{1-\delta}\chi(\varpi)^{-\delta}\left(\frac{1}{1-\frac{\theta\chi(\varpi)}{q^{s}}}+\frac{q^{s}\theta\chi(\varpi)^{-1}}{1-q}\right)\\
		&=\chi(-1)\frac{1}{1-q}q^{\delta(s-1)}\theta(\varpi)^{-\delta}\chi(\varpi)^{-\delta-1}\frac{1-\frac{\theta\chi(\varpi)}{q^{s-1}}}{1-\frac{\theta\chi(\varpi)}{q^{s}}},
	\end{align*}
	using that $\displaystyle \int_{\mathcal{O}^\times}\psi(t \varpi^{k-\delta -1})d^\times t$ is 0 if $k<0$, is 1 if $k > 0$, and is $1/(1-q)$ if $k=0$.
	
	Substituting this in, we find the zeta integral is 
			\begin{align*}
		&q^{\beta n(s-1/2)-\beta\binom{n}{2}} \cdot \alpha_{\pri,n}^\delta q^{-\delta n^2} \cdot \chi(\det w_{n})\prod_{i=n+1}^{2n}\left[\chi(-1)\frac{q^{\delta (s-1)}}{1-q}\UPS_{i}(\varpi)^{-\delta}\cdot  \chi(\varpi)^{-\delta -1} \frac{1-\frac{\theta\chi(\varpi)}{q^{s-1}}}{1-\frac{\theta\chi(\varpi)}{q^{s}}}\right]\\
		&= q^{\beta n(s-n/2)} \cdot \alpha_{\pri,n}^\delta q^{-\delta n^2} \cdot \chi(\det (-w_n)) \cdot q^{\delta n (s-1)} \cdot \bigg[\prod_{i=n+1}^{2n}\UPS_{i}(\varpi)^{-\delta}\bigg] \cdot Q(\pi, \chi, s).
	\end{align*}
	We conclude that this has the required shape via the same computations as in \eqref{eq:finishing ramified}.
\end{proof}

	\part{$p$-adic Interpolation}
	%%================================================================
	%%
	%%				OVERCONVERGENT COHOMOLOGY
	%%
	%%================================================================
	
	\section{Overconvergent cohomology}\label{sec:overconvergent cohomology}
	
	We recap the theory of overconvergent cohomology in $p$-adic families, as developed for example in \cite{Urb11, Han17}. All of this material is explained in depth \emph{op.\ cit}., so we are terse with details.

	%%===================================================
	\subsection{Weight spaces}\label{sec:weight spaces}
	Recall $X^*(T),$ $X^*_0(T),$ $X^*(H)$ and $X^*_0(H)$ from \S\ref{sec:algebraic weights}. If $K \subset G(\A_f)$ is an open compact subgroup, let $\overline{Z}_K$ be the $p$-adic closure of $Z_G(\Q)\cap K$. 
	
	\begin{definition}\label{def:weight space}
		\begin{itemize}\setlength{\itemsep}{0pt}
			\item[(i)] The \emph{weight space} for $G$ of level $K$ is $\sW_K^G \defeq \mathrm{Spf}(\Zp[\![ T(\Zp)/\overline{Z}_K]\!])^{\eta}.$ This is a rigid analytic space whose $L$-points, for $L \subset \C_p$ any extension of $\Q_p$, are given by
			\[ 
			\sW^{G}_K(L) = \Hom_{\mathrm{cont}}(T(\Z_p)/\overline{Z}_K,L^\times).
			\]
			\item[(ii)]	Any element of $\sW^G_K(L) \cap X_+^*(T)$ is called an \emph{algebraic weight}. 
			\item[(iii)] The \emph{pure weight space} $\sW_{K,0}^G$ is the Zariski-closure of $X_0^*(T)\cap \sW^G_K$ in $\sW^G_K$.
		\end{itemize}
	\end{definition}
	
	We view all of these weights as characters of $T(\Zp)$ trivial on $\overline{Z}_K$. If $\lambda$ is pure of weight $\sw$ and $z \in \overline{Z}_K$, then $\lambda(z) \subset \{\pm 1\}$, so if $\lambda$ is trivial on $\overline{Z}_K$, then this is also true of all weights in a neighbourhood of $\lambda$ in the pure weight space. Since the level subgroup will always be fixed, we will henceforth always drop it from notation.
	
	A weight $\lambda \in \sW^G$ decomposes as $\lambda = (\lambda_1,...,\lambda_{2n})$, where each $\lambda_i$ is a character of $\Zp^\times$. We see $\lambda \in \sW_0^G$ if and only if there exists $\sw_\lambda \in \mathrm{Hom}_{\mathrm{cts}}(\Zp^\times, L^\times)$ such that $\lambda_i\cdot\lambda_{2n+1-i} = \sw_\lambda$ for all $1 \leq i \leq n$. The space $\sW_0^G$ has dimension $n+1$ (corresponding to changing $\lambda_1,...,\lambda_n, \sw_\lambda$).

	If $\Omega \subset \sW^G_0$ is an affinoid in the pure weight space, then $\Omega$ is equipped with a character $\chi_\Omega : T(\Zp) \to \cO_\Omega^\times$ such that for any $\lambda \in \Omega(L)$, the composition $T(\Zp) \xrightarrow{\chi}\cO_\Omega^\times \xrightarrow{\mathrm{sp}_\lambda} L^\times$ is the character attached to $\lambda$, where $\mathrm{sp}_\lambda$ is evaluation at $\lambda$. Moreover, write $\chi_\Omega = (\chi_{\Omega,1},...,\chi_{\Omega,2n})$, where each $\chi_{\Omega,i}$ is a character of $\Zp^\times$; then since $\Omega$ is pure, there exists a character 
	\begin{equation}\label{eq:sw_omega}
		\sw_\Omega : \Zp^\times\to \cO_\Omega^\times \hspace{12pt}\text{such that  } \sw_\Omega(z) = \chi_{\Omega,i}(z)\cdot \chi_{\Omega,2n+1-i}(z) \ \forall z \in \Zp^\times, 1 \leq i \leq n.
	\end{equation}

	%%==============================================
	\subsection{Algebraic and analytic induction}
	
	For $\lambda \in X_0^*(T)$, recall $V_\lambda$ is the algebraic representation of $G$ of highest weight $\lambda$. The $L$-points of $V_\lambda$ can be described explicitly as the algebraic induction, whose points are algebraic functions $f : G(\Qp) \to L$ such that
	\begin{equation}\label{eq:alg induction}
		f(n^-tg) = \lambda(t)f(g) \ \ \ \forall n^- \in \overline{N}(\Qp), t \in T(\Qp), g \in G(\Qp).
	\end{equation}
	The action of $\gamma \in G(\Qp)$ on $f \in V_\lambda(L)$ is by $(\gamma \cdot f)(g) \defeq f(g \gamma)$. 
	
	As $G(\Zp)$ is Zariski-dense in $G(\Qp)$, we can identify $V_\lambda(L)$ with the set of algebraic $f : G(\Zp) \to L$  satisfying \eqref{eq:alg induction}. We have an integral subspace $V_\lambda(\cO_L)$ of $f$ such that $f(G(\zp)) \subset \cO_L$, and we let $V_\lambda^\vee(\cO_L) = \mathrm{Hom}_{\cO_L}(V_\lambda(\cO_L),\cO_L)$.\medskip
	
	If $I$ is a $p$-adic Lie group and $R$ a $\Qp$-algebra, let $\cA(I,R)$ denote the space of locally analytic functions $I \to R$. Let $\Iwahori_G \subset G(\Zp)$ be the Iwahori subgroup, and $\Omega \subset \sW^G_0(L)$ an affinoid, with attached character $\chi_\Omega$; we allow $\Omega = \{\lambda\}$, whence $\chi_\Omega = \lambda$. Recall the analytic induction spaces:
	
	\begin{definition}\label{def:A_Omega}
		Let $\cA_{\Omega}$ be the space of $f \in \cA(\Iwahori_G,\cO_\Omega)$ such that
		\begin{equation}\label{eq:iwahori induction}
			f(n^-t g) = \chi_\Omega(t)f(g)\text{ for all }n^- \in \overline{N}(p\Zp),\ t \in T(\Zp),\text{ and }g \in \Iwahori_G.
		\end{equation}
	\end{definition}
	Via Iwahori decomposition and \eqref{eq:iwahori induction}, restriction to $N(\Zp)$ identifies $\cA_{\Omega}(L)$ with $\cA(N(\Zp),\cO_\Omega)$. 
	
	As any $f \in V_\lambda(L)$ is determined by its restriction to the Zariski-dense subgroup $\Iwahori_G$, we see $V_\lambda(L)$ is the (finite Banach) subspace of $f \in \cA_\lambda$ that are algebraic on $\Iwahori_G$ (e.g.\ \cite[\S3.2.8]{Urb11}).

	\begin{definition}\label{def:D_Omega}
		Define $\cD_{\Omega} \defeq \Hom_{\mathrm{cont}}(\cA_{\Omega}, \cO_\Omega)$, a compact Fr\'echet $\cO_\Omega$-module.
	\end{definition}
	
	If $\Sigma \subset \Omega$ is a closed affinoid, then $\cD_{\Omega} \otimes_{\cO_{\Omega}}\cO_{\Sigma} \cong \cD_{\Sigma}$.

	%%==============================================
	\subsection{Hecke actions and slope decompositions}\label{sec:slope-decomp}\label{sec:slope-decomp 2}
	Let $K \subset G(\A_f)$ be an open compact subgroup such that $K_p\subset \Iwahori_G$. 
	Let $\Omega \subset \W$ be an affinoid. Recall $t_p = \mathrm{diag}(p^{2n-1},p^{2n-2},...,p,1)$, and let $\Delta_p \subset G(\Q_p)$ be the semigroup generated by $\Iw$ and $t_p$. There is a left-action of $\Delta_p$ on $\cD_\Omega$ as follows:
	\begin{itemize}\setlength{\itemsep}{0pt}
		\item[--] $k \in \Iw$ acts on $f \in \cA_\Omega$ by $(k * f)(g) \defeq f(gk)$, inducing a dual left action on $\mu \in \cD_\Omega$ by $(k * \mu)(f(g)) \defeq \mu(f(gk^{-1}))$.
		\item[--] $t_p$ acts on the left on $B(\Zp)$ by $t_p * b \defeq t_pbt_p^{-1}$. Since any $f \in \cA_\Omega$ is uniquely determined by its restriction to $B(\Zp)$, this induces a left action of $t_p$ on $\mu \in \cD_\Omega$ by $(t_p * \mu)(f(b)) = \mu(f(t_pbt_p^{-1}))$.
	\end{itemize}

	As $\bigcap_{i\geqslant 0} t_p^{i} N(\Zp) t_p^{-i} = 1,$ we have $t_p \in T^{++}$ in the notation of \cite[\S2]{Han17}. Thus we get an $\cO_\Omega$-linear controlling operator $U_p^\circ \defeq [K_p t_p K_p]$ on the cohomology groups $\hc{\bullet}(S_K,\sD_\Omega)$. Up to shrinking $\Omega$, the $\cO_{\Omega}$-module $\hc{\bullet}(S_K,\sD_{\Omega})$ admits a slope decomposition with respect to $U_p^\circ$  (see \cite[Def.~2.3.1]{Han17}). 
	For  $h \in \Q_{\geqslant 0}$ we let  $\hc{\bullet}(S_K,\sD_{\Omega})^{\leqslant h}$ denote the subspace of elements of slope at most $h$, and note that it is an $\cO_{\Omega}$-module of finite type.

	\begin{remark}\label{rem:integral normalisation}
		The operator $U_p^\circ$ preserves the integral structure $\hc{\bullet}(S_K,\sD_\lambda(\cO_L)) \subset \hc{\bullet}(S_K,\sD_\lambda(L))$. We also have a $*$-action of $\Delta_p$ on $V_\lambda^\vee(L)$, defined identically, giving an operator $U_p^\circ$ on $\hc{\bullet}(S_K,\sV_\lambda^\vee(L))$ that preserves its natural integral subspace. If $U_p$ denotes the automorphic Hecke operator from \S\ref{sec:hecke p}, one may check $U_p^\circ = \lambda(t_p) \cdot U_p$. This is all explained in the remark in \cite[\S3.5]{BDW20}.
	\end{remark}

	%%==============================================
	\subsection{Non-critical slope refinements}\label{sec:nc slope}
	Let $\lambda \in X_0^*(T)$ be a pure dominant integral weight, $K$ as above, and $L/\Q_p$ a finite extension. The natural inclusion of $V_\lambda(L) \subset \cA_\lambda(L)$ induces dually a surjection $r_\lambda : \cD_\lambda(L) \longrightarrow V_\lambda^\vee(L)$, which is equivariant for the $*$-actions of $\Delta_p$. This induces a map
	\begin{equation}\label{eqn:specialisation}
		r_\lambda : \hc{\bullet}(S_K, \sD_\lambda(L)) \longrightarrow \hc{\bullet}(S_K,\sV_\lambda^\vee(L)),
	\end{equation}
	equivariant for the $*$-actions of $\Delta_p$ (hence $U_p^\circ$) on both sides.
	
	\begin{definition}\label{def:non-Q-critical}
		Let $\tilde\pi = (\pi,\alpha)$ be a $p$-refined RACAR of $G(\A)$ of weight $\lambda$. We say $\tilde\pi$ is \emph{non-critical} if $r_\lambda$ restricts to an isomorphism
		\begin{align*}
			r_\lambda : \hc{\bullet}(S_{K}, \sD_\lambda(L))\locpi  \isorightarrow&\  \hc{\bullet}(S_{K},\sV_\lambda^\vee(L))\locpi
		\end{align*}
		of generalised eigenspaces. 
		We say $\tilde\pi$ is \emph{strongly non-critical} if this is also true for $\h^{\bullet}$ (i.e., if $\tilde\pi$ is non-critical for $\h^\bullet$ \emph{and} for $\hc{\bullet}$ as in \cite[Rem.~4.6]{BW20}). 
	\end{definition}

	\begin{definition}\label{def:non-critical slope}
		Let $\tilde\pi$ be a $p$-refinement of $\pi$.  For $1 \leq r \leq 2n-1$, let $\alpha_{p,r}^\circ = \lambda(t_{p,r}) \alpha_{p,r}$, the corresponding $U_{p,r}^\circ$-eigenvalue. We say $\tilde\pi$ has \emph{non-critical slope} if for each $1 \leq r \leq 2n-1$, we have
		\[	 
		v_p\big(\alpha_{p,r}^\circ\big) < \lambda_{r} - \lambda_{r+1} + 1.
		\]
	\end{definition}

	\begin{theorem}[Classicality] \label{thm:control}
		If $\tilde\pi$ has non-critical slope, then it is strongly non-critical.
	\end{theorem}
	\begin{proof}
		This is \cite[Thm.~4.4, Rem.~4.6]{BW20}, explained in Examples 4.5 \emph{op.\ cit}.
	\end{proof}

	%%================================================================
	%%
	%%				BRANCHING LAWS
	%%
	%%================================================================
	
	\section{$p$-adic interpolation of branching laws}\label{sec:branching laws}
	A \emph{branching law} describes how an irreducible representation of $G$ decomposes upon restriction to $H$. Of particular interest to us is the branching law given by Lemma \ref{lem:branching law 2}, rephrased below in Lemma \ref{lem:branching law}, which provides a representation-theoretic interpretation of the Deligne-critical range. We now give a $p$-adic interpolation of this branching law.

	\subsection{Classical branching laws, revisited}\label{sec:classical branching laws}
	Let $\lambda \in X_0^*(T)$ be a pure algebraic weight, with purity weight $\sw$. By dualising, we get the following equivalent formulation of Lemma \ref{lem:branching law 2}:
	
	\begin{lemma}\label{lem:branching law}
		For $j \in \Z$, we have $j \in \mathrm{Crit}(\lambda)$ if and only if $\mathrm{dim}\  \mathrm{Hom}_{H(\Zp)}(V_\jw^H, V_\lambda) = 1$, that is if and only if $V_{\lambda}|_{H(\Zp)}$ contains $V_\jw^H$ with multiplicity 1.
	\end{lemma} 
	We now give a conceptual description of the generators of $V_\jw^H \subset V_{\lambda}|_{H(\Zp)}$. For $n \geq 2$, define weights
	\begin{align}\label{eq:alpha_i}
		\alpha_{1} = (1,0,...,0,-1), \ \ \alpha_{2} = (1,1,0,...,0,-1,-1)&,\ \  ...,\ \  \alpha_{n-1} = (1,...,1,0,0,-1,...,-1),\notag\\
		\alpha_{0} = (1,...,1,1,...,1), \hspace{12pt} &\alpha_{n} = (1,...,1,0,...,0).
	\end{align}
For $n=1$, we define $\alpha_0 = (1,1)$ and $\alpha_1 = (1,0)$. 
	For any $n$, if $\lambda = (\lambda_1,...,\lambda_{2n}) \in X_0^*(T)$ is a dominant pure algebraic weight, then we easily see that 
	\begin{equation}\label{eq:lambda in alpha}
		\lambda = (\lambda_1-\lambda_2)\alpha_{1} + (\lambda_2 - \lambda_3)\alpha_{2}+ \cdots + (\lambda_n - \lambda_{n+1})\alpha_{n} + \lambda_{n+1}\alpha_{0},
	\end{equation}
	where each coefficient is a non-negative integer except perhaps $\lambda_{n+1}$, which can be negative.
	
	\begin{notation}\label{not:v_i}
		\begin{itemize}\setlength{\itemsep}{0pt}
			\item[(i)] If $1 \leq i \leq n-1$, then $\mathrm{Crit}(\alpha_{i}) = \{0\}$. Since the purity weight of $\alpha_{i}$ is $0$, by Lemma \ref{lem:branching law} there is a non-zero vector $v_{(i)}\in V_{\alpha_{i}}(\Qp)$ upon which the action of $H(\Zp)$ is trivial, and $v_{(i)}$ is unique up to $\Qp^\times$-multiple.
			\item[(ii)] We have $\mathrm{Crit}(\alpha_{n}) = \{-1,0\}$, and the purity weight of $\alpha_{n}$ is 1. By Lemma \ref{lem:branching law}, there exist non-zero vectors $v_{(n), 1}, v_{(n),2} \in V_{\alpha_{n}}(\Qp)$  such that the action of $(h_1,h_2) \in H(\Zp)$ on $v_{(n),i}$ is by $\det(h_i)$. Again, these vectors are unique up to $\Qp^\times$-multiple.
			\item[(iii)] The space $V_{\alpha_{0}}$ is a line, with basis $v_{(0)}(g) = \mathrm{det}(g)$. \medskip
		\end{itemize}
	\end{notation}

	All of the $v_{(i)}$ are explicit algebraic functions $G(\Zp) \to \Qp$, and we can multiply them in the ring $\Qp[G]$. The following gives an explicit branching law from such a product.
	
	\begin{proposition}\label{prop:product for v lambda j}
		Let $\lambda \in X_0^*(T)$ be a dominant pure algebraic weight and $j \in \mathrm{Crit}(\lambda)$. Then
		\[
		v_{\lambda,j} \defeq (v_{(1)}^{\lambda_1-\lambda_2}) \cdot (v_{(2)}^{\lambda_2-\lambda_3}) \cdots (v_{(n-1)}^{\lambda_{n-1}-\lambda_n}) \cdot (v_{(n),1}^{-\lambda_{n+1}-j}) \cdot (v_{(n),2}^{\lambda_n + j}) \cdot (v_{(0)}^{\lambda_{n+1}})  
		\]
		is a generator of the line $V_\jw^H(\Qp)$ inside the $H(\Zp)$-representation $V_\lambda(\Qp)|_{H(\Zp)}$.
	\end{proposition}
	\begin{proof}
		This product is algebraic, as the $v_{(i)}$ and $v_{(n),i}$ are algebraic by construction, and:
		\begin{itemize}\setlength{\itemsep}{0pt}
			\item[--] for $1 \leq i \leq n-1$, we have $\lambda_{i} - \lambda_{i+1}\geq 0$, so $v_{(i)}^{\lambda_i - \lambda_{i+1}}$ is algebraic;
			\item[--] since $\lambda_{n} \geq -j \geq \lambda_{n+1}$, we have $v_{(n),1}^{-\lambda_{n+1}-j}$ and $v_{(n),2}^{\lambda_n+ j}$ are algebraic;
			\item[--] and $v_{(0)}^{\lambda_{n+1}} = \det^{\lambda_{n+1}}$ is algebraic.
		\end{itemize} 
		Thus their product is algebraic. If $t \in T(\Zp)$, then by \eqref{eq:lambda in alpha} we see
		\[
		\lambda(t) = \alpha_1^{\lambda_1-\lambda_2}(t) \cdots \alpha_{n-1}^{\lambda_{n-1}-\lambda_n}(t) \cdot \alpha_{n}(t)^{-\lambda_{n+1}-j} \cdot \alpha_{n}^{\lambda_n + j}(t) \cdot \alpha_0^{\lambda_{n+1}}(t),
		\]
		so for $n^- \in \overline{N}(\Zp)$ and $g \in G(\Zp)$, we see $v_{\lambda,j}(n^-tg) = \lambda(t)v_{\lambda,j}(g)$ by multiplying together the analogous relations for the $v_{(i)}$. Finally, by Notation \ref{not:v_i} we see $(h_1,h_2) \in H(\Zp)$ acts on $v_{\lambda,j}$ by
		\[
		\det(h_1)^{-\lambda_{n+1}- j} \cdot \det(h_2)^{\lambda_n + j} \cdot \det(h_1h_2)^{\lambda_{n+1}} = \det(h_1)^{-j} \det(h_2)^{\sw+ j},	
		\]
		as required (using $\sw = \lambda_n + \lambda_{n+1}$ in the final step).
	\end{proof}

	\subsection{Support conditions for branching laws}\label{sec:support conditions}
	 Let $w_n$ denote the antidiagonal $n\times n$ matrix with $(w_n)_{ij} = \delta_{i, n+1-j}$, and recall
	\begin{equation}\label{def:xi}
		u = \smallmatrd{1_n}{w_n}{0}{1_n} \in G(\zp).
	\end{equation}
	For $\beta \in \Z_{\geq 1}$, let
	\[
	N^\beta(\Zp) \defeq N(p^\beta\Zp)\cdot u = \left\{n \in N(\Zp) : n \equiv \smallmatrd{1_n}{w_n}{0}{1_n} \newmod{p^\beta}\right\}.
	\]
	Note this is \emph{not} a subgroup of $N(\Zp)$. We also emphasise that $N^\beta(\Zp)$ is not the set of $\Zp$-points of an algebraic group, and hope the notation does not cause confusion. 
	
	Let $\lambda \in X_0^*(T)$ be a pure dominant algebraic weight, let $j \in \mathrm{Crit}(\lambda)$, and let $v_{\lambda,j}$ be a generator of the line $V_\jw^H$ inside the $H(\Zp)$-representation $V_\lambda(\Qp)$, viewed as an explicit algebraic function $G(\Zp) \to \Qp$. The key examples we consider are $\lambda = \alpha_i$ from \eqref{eq:alpha_i}, with $v_{\lambda,j} = v_{(i)}$ from Notation \ref{not:v_i}. The aim of this subsection is to prove:

	\begin{proposition}\label{prop:support in N1}
		After possibly rescaling $v_{\lambda,j} \in V_\jw^H \subset V_\lambda(\Qp)\big|_{H(\Zp)}$, we have 
		\[
			v_{\lambda,j}(N^\beta(\Zp)) \subset 1+p^\beta\Zp \qquad \text{for all}\ \beta \geq 1.
		\]
	\end{proposition}
	
	Recall $G_n = \GL_n$ and its subgroups $B_n$, $\overline{B}_n$, $N_n$, and $\overline{N}_n$ from \S\ref{sec:notation}. Let $\Iwn$ be the Iwahori subgroup of $G_n(\Zp)$, and let $J^\beta \defeq \{g \in G_n(\zp):  g\newmod{p^\beta} \in T_n(\Z/p^\beta)\} \subset \Iwn$. By the Iwahori factorisation, any element $1_n + p^\beta Y \in J^\beta$ has an Iwahori factorisation $1_n + p^\beta Y = RS$, where $R \in N_n(\Zp)$ and $S \in \overline{B}_n(\Zp)$. 
	
	\begin{lemma}\label{lem:big cell}
		If $X \in M_n(\Zp)$, there exist $\overline{B} \in \overline{B}_{2n}(\Zp), h \in H(\Zp)$ with $\overline{B}, h \equiv 1_{2n} \newmod{p^\beta}$ and
		\[
		\smallmatrd{1_n}{w_n + p^\beta X}{0}{1_n} = \overline{B} \cdot u \cdot h.
		\]
	\end{lemma}
	
	\begin{proof}
		Fix $R \in B_n(\Zp), S \in \overline{B}_n(\Zp)$ such that $1 + p^\beta w_nX = RS.$ A simple check shows $R,S \equiv 1 \newmod{p^\beta}$. We see 
		\[
		\smallmatrd{w_nRw_n}{}{}{S^{-1}} \in \overline{B}_{2n}(\Zp), \ \ \ \ \ \ \smallmatrd{(w_nRw_n)^{-1}}{}{}{S} \in H(\Zp),
		\]
		and 
		\[
		\smallmatrd{w_nRw_n}{}{}{S^{-1}}\cdot \smallmatrd{1_n}{w_n}{0}{1_n} \cdot \smallmatrd{(w_nRw_n)^{-1}}{}{}{S} = \smallmatrd{1_n}{w_n+p^nX}{0}{1_n},
		\]	
		which has the claimed form.
	\end{proof}
	
	Now let $v_{\lambda,j}$ be as in Proposition \ref{prop:support in N1}. Since $v_{\lambda,j}$ transforms by $\lambda(\overline{B}) \in \Zp^\times$ under left translation by $\overline{B}\in \overline{B}_{2n}(\Zp)$, and transforms like $\det(h_1)^{-j}\det(h_2)^{\sw+j} \in \Zp^\times$ under right translation by $h = (h_1,h_2) \in H(\zp)$, we see
	\[
	v_{\lambda,j}\big(\overline{B}_{2n}(\Zp) \cdot u \cdot H(\Zp)\big) \subset \Zp^\times \cdot v_{\lambda,j}(u).
	\]
	Suppose $v_{\lambda,j}(u) = 0$; then $v_{\lambda,j}$ vanishes on the cell $\overline{B}_{2n}(\Zp) \cdot u\cdot H(\Zp)$. This cell is open and dense in $G(\Zp)$ (e.g. \cite[\S5.1.3]{loeffler-parabolics}), forcing $v_{\lambda,j} = 0$, which contradicts our assumptions. Thus $v_{\lambda,j}(u) \neq 0$, and we are free to rescale it by an element of $\Qp^\times$ so that $v_{\lambda,j}(u) = 1$. Further if $\overline{B}, h \equiv 1_{2n} \newmod{p^\beta}$, then $\lambda(\overline{B}), \det(h_1)^{-j}\det(h_2)^{\sw+j} \equiv 1 \newmod{p^\beta}$. Combining all of this with Lemma \ref{lem:big cell}, we deduce that for any $X \in M_n(\Zp)$, we have
	\begin{equation}\label{eq:cell to Zp^times}
		v_{\lambda,j}\left(\smallmatrd{1_n}{w_n + p^\beta X}{0}{1_n}\right) \in  1 + p^\beta\Zp.
	\end{equation}
	
	\begin{proof} \emph{(Proposition \ref{prop:support in N1}).}
		A general element of $N^\beta(\Zp)$ looks like $n = \smallmatrd{A}{w_n + p^\beta Y}{0}{B}$, where $A,B \in N_n(\Zp)$ with $A \equiv B \equiv 1_n \newmod{p^\beta}$ and $Y \in M_n(\Zp)$. Letting
		\[
		X \defeq \left(Y- w_n\left(\tfrac{B-1}{p^\beta}\right)\right)B^{-1} \in M_n(\Zp),
		\]
		we have 
		\[
		\smallmatrd{A}{w_n + p^\beta Y}{0}{B} = \smallmatrd{1_n}{w_n+p^\beta X}{0}{1_n}\smallmatrd{A}{}{}{B}.
		\]
		Then $v_{\lambda,j}(n) = \det(A)^{-j}\det(B)^{\sw+j} \cdot v_{\lambda,j}\left(\smallmatrd{1_n}{w_n+pY}{0}{1_n}\right) \in 1 + p^\beta\Zp$ by \eqref{eq:cell to Zp^times}.
	\end{proof}

	\subsection{$p$-adic interpolation of branching laws}\label{sec:p-adic interpolation branching}

	Now we $p$-adically interpolate. Assume the choices $v_{(0)}, v_{(1)},...,v_{(n-1)},v_{(n),1},v_{(n),2}$ in Notation \ref{not:v_i} were all normalised so $v_{(i)}(u) = 1$, whence $v_{(i)}(N^1(\Zp)) \subset 1+p\Zp \subset \Zp^\times$, as in Proposition \ref{prop:support in N1} (and its proof).
	
	If $R$ is a $\Qp$-algebra and $\chi = (\chi_1,...,\chi_{2n}): T(\Zp) \to R^\times$ is a character, then define a function
	\begin{align}\label{eq:v circ}
		w_\chi : N^1(\Zp) &\longrightarrow R^\times,\\
		g &\longmapsto v_{(0)}(g)^{\chi_{n+1}}  \cdot \left(\prod_{i=1}^{n-1} v_{(i)}(g)^{\chi_i - \chi_{i+1}}\right) \cdot v_{(n),1}(g)^{-\chi_{n+1}} \cdot v_{(n),2}(g)^{\chi_n},\notag
	\end{align}
	where if $x \in \Zp^\times$ we write $x^{\chi_i}$ as shorthand for $\chi_i(x)$. If $\beta \geq 1$, let
	\begin{equation}\label{eq:Iw 1}
		\Iw^\beta \defeq \overline{N}(p\zp)\cdot T(\zp) \cdot N^\beta(\zp) \subset \Iw,
	\end{equation}
	which again is not a subgroup. We may extend $w_\chi$ to a function $w_\chi :  \Iw^1 \to R^\times$ via \eqref{eq:iwahori induction}:
	\begin{equation}\label{eq:iwahori induction 2}
		w_\chi(\overline{n} \cdot t \cdot n) = \chi(t) \cdot w_\chi(n), \hspace{12pt} \overline{n} \in \overline{N}(p\Zp), t \in T(\zp), n \in N^1(\Zp).
	\end{equation}
	For $\beta \geq 1$, let
	\begin{equation}\label{eq:Iw_H^1}
		\IwH^\beta \defeq H(\Zp) \cap u^{-1}\Iw^\beta.
	\end{equation}
	If $g \in \Iw^1$ and $h \in \IwH^1$, then a simple check shows $gh \in \Iw^1$, so we can consider $w_\chi(gh)$. The following will be important in the sequel: recalling $L_p^{B,\beta}$ from Definition \ref{def:auto cycle}, note that for appropriate choices of $K_p$ (e.g. if $K_p = \Iw$) we have $L_p^{B,\beta} \subset \IwH^\beta$. 
	
	\begin{lemma} \label{lem:w_kappa action}
		If $h = (h_1,h_2) \in \IwH^1$, then $w_\chi(gh) = \det(h_2)^{\chi_{n}+\chi_{n+1}} \cdot w_\chi(g).$
	\end{lemma} 
	\begin{proof}
		By definition, $w_\chi(gh)$ is a product of terms involving $v_{(i)}(gh)$, $v_{(n),i}(gh)$. By construction $v_{(i)}(gh) = v_{(i)}(g)$ for $1 \leq i \leq n-1$, and $v_{(0)}(gh) = \det(h_1)\det(h_2) v_{(0)}(g)$, and $v_{(n),i}(gh) = \det(h_i)v_{(n),i}(g)$ for $i = 1,2$. We get factors of $\det(h_1)^{\chi_{n+1}}\det(h_2)^{\chi_{n+1}}$ from $v_{(0)}$, $\det(h_1)^{-\chi_{n+1}}$ from $v_{(n),1}$, and $\det(h_2)^{\chi_n}$ from $v_{(n),2}$, so 
		\[
		w_\chi(gh) = \det(h_1)^{\chi_{n+1}}\det(h_2)^{\chi_{n+1}} \cdot \det(h_1)^{-\chi_{n+1}} \cdot \det(h_2)^{\chi_n} \cdot w_\chi(g) = \det(h_2)^{\chi_n + \chi_{n+1}} \cdot w_\chi(g). \qedhere
		\]
	\end{proof}
	
	\begin{remark}
		The definition of $w_\chi$ is heavily motivated by Proposition \ref{prop:product for v lambda j}. Indeed we see that if $\lambda \in X_0^*(T)$ is a pure dominant algebraic weight and $j \in \mathrm{Crit}(\lambda)$, then for any $g \in N^1(\Zp)$ we have
		\begin{equation}\label{eq:explicit v_lambda,j}
			v_{\lambda,j}(g) = w_\lambda(g) \cdot \left(\frac{v_{(n),2}(g)}{v_{(n),1}(g)}\right)^j.
		\end{equation}
	\end{remark}
	
	To $p$-adically vary branching laws, we take $R = \cO_\Omega$, for $\Omega \subset \sW_0^G$ an affinoid in the pure weight space, equipped with a character $\chi_\Omega : T(\Zp) \to \cO_\Omega^\times$ as in \S\ref{sec:weight spaces}. 
	
	\begin{definition}\label{def:v_Omega}
		Let $f \in \cA(\Zp^\times,\cO_\Omega)$ be a locally analytic function. Define a function $v_\Omega(f) : N(\Zp) \to \cO_\Omega$ by
		\[
		v_\Omega(f)(g) \defeq \left\{\begin{array}{cc} w_{\chi_\Omega}(g)\cdot f\left(\tfrac{v_{(n),2}(g)}{v_{(n),1}(g)}\right) &: g \in N^1(\Zp),\\
			0 &: \text{otherwise}.
		\end{array}\right. 
		\]
		This is well-defined by the normalisations fixed at the start of \S\ref{sec:p-adic interpolation branching}, by Proposition \ref{prop:support in N1}, and the definition \eqref{eq:v circ}. Under the transformation law \eqref{eq:iwahori induction} and Iwahori decomposition, this extends to a unique element $[v_\Omega(f) : \Iwahori_G \to \cO_\Omega] \in \cA_\Omega$, with support on $\Iw^1$ from \eqref{eq:Iw 1}.
	\end{definition}
	
	We allow $\Omega = \{\lambda\}$ to be a point, in which case the above defines, for any $f \in \cA(\Zp^\times,L)$, a function $v_\lambda(f) : \Iwahori_G \to L$ in $\cA_\lambda$.
	
	\begin{lemma}\label{lem:v_lambda interpolates}
		If $\lambda \in X_0^*(T)$ and $g \in \Iw^1$, then for all $j \in \mathrm{Crit}(\lambda)$ we have 
		\[
		\big[v_\lambda(z\mapsto z^j)\big](g) = v_{\lambda,j}(g).
		\]
	\end{lemma}
	\begin{proof}
		For $g \in N^1(\Zp)$, this follows by combining the definition of $v_\lambda\textsc{}$ with \eqref{eq:explicit v_lambda,j}. Both sides satisfy the same transformation law \eqref{eq:iwahori induction 2} to extend to $\Iw^1$, so we have equality on the larger group too.
	\end{proof}
	
	Recall $\sw_\Omega : \Zp^\times \to \cO_\Omega^\times$ from \eqref{eq:sw_omega}. Let $h = (h_1,h_2) \in H(\Zp)$ act on $f \in \cA(\Zp^\times,\cO_\Omega)$ by
	\begin{equation}\label{eq:H(Zp) action on A(Zp*)}
		(h*f)(z) = \det(h_2)^{\sw_\Omega} \cdot f\left(\tfrac{\det(h_2)}{\det(h_1)} \cdot z\right).
	\end{equation}
	It follows that $\IwH^1 \subset H(\Zp)$, as defined in \eqref{eq:Iw_H^1}, acts on $\cA(\Zp^\times,\cO_\Omega)$. It also acts on $\cA_\Omega$ via its embedding into $\Iwahori_G$. 
	
	\begin{lemma}\label{lem:v_Omega equivariant}
		The map $v_{\Omega} : \cA(\zp^\times,\cO_\Omega) \to \cA_\Omega$ is $\IwH^1$-equivariant.
	\end{lemma}
	\begin{proof}
		Let $h = (h_1,h_2) \in \IwH^1$ and $g \in N^1(\Zp)$. First note that by \eqref{eq:sw_omega} we know $\sw_\Omega = \chi_{\Omega,n} + \chi_{\Omega,n+1}$; so $\det(h_2)^{\sw_\Omega}w_{\chi_\Omega}(g) = w_{\chi_\Omega}(gh)$ by Lemma \ref{lem:w_kappa action}.
		
		Now let $f \in \cA(\Zp^\times,\cO_\Omega)$, and compute
		\begin{align}\label{eq:equivariant 1}
			v_\Omega(h*f)(g) &= w_{\chi_\Omega}(g)\cdot (h*f)\left(\frac{v_{(n),2}(g)}{v_{(n),1}(g)}\right)\\
			&= \det(h_2)^{\sw_\Omega} w_{\chi_\Omega}(g)\cdot f\left(\frac{\det(h_2)\cdot v_{(n),2}(g)}{\det(h_1)\cdot v_{(n),1}(g)}\right)= w_{\chi_\Omega}(gh)\cdot f\left(\frac{v_{(n),2}(gh)}{ v_{(n),1}(gh)}\right).\notag
		\end{align}
		Since $gh \in \Iw^1$, we may write $gh = b'g'$, with $b' \in \overline{N}(p\Zp)\cdot B(\Zp)$ and $g' \in N^1(\Zp)$. Then 
		\begin{align} \label{eq:equivariant 2}
			[h*&v_\Omega(f)](g) = v_\Omega(f)(gh) = v_\Omega(f)(b'g') = \chi_\Omega(b') \cdot v_\Omega(f)(g')\\
			&= \chi_\Omega(b')w_{\chi_\Omega}(g')  f\left(\frac{v_{(n),2}(g')}{v_{(n),1}(g')}\right) = w_{\chi_\Omega}(b'g')  f\left(\frac{v_{(n),2}((b')^{-1}gh)}{v_{(n),1}((b')^{-1}gh)}\right) = w_{\chi_\Omega}(gh)  f\left(\frac{v_{(n),2}(gh)}{v_{(n),1}(gh)}\right),\notag
		\end{align}
		where in the last equality we use that $v_{(n),i}((b')^{-1}gh) = \alpha_n(b')^{-1}v_{(n),i}(gh)$ for both $i = 1,2$. Combining \eqref{eq:equivariant 1} and \eqref{eq:equivariant 2} yields $v_\Omega(h*f)(g) = [h*v_\Omega(f)](g),$	as required.
	\end{proof}

	\subsection{Branching laws for distributions}
	
	The overconvergent cohomology groups we consider have coefficients in $\cD_\Omega$, not $\cA_\Omega$, so we now dualise the above. Finally, we collate everything we have proved in the main result of this section (Proposition \ref{prop:distribution commute}).\medskip
	
	By Lemma \ref{lem:branching law 2}, for $\lambda \in X_0^*(T)$ we have $j \in \mathrm{Crit}(\lambda)$ if and only if $\mathrm{Hom}_{H(\Zp)}(V_\lambda^\vee, V^H_{(j,-\sw-j)})$ is a line; and moreover, the choices made in \S\ref{sec:classical branching laws} fix a generator
	\begin{equation}\label{eq:kappa lambda j}
		\kappa_{\lambda,j} : V_\lambda^\vee(\Qp) \to V^H_{(j,-\sw-j)}(\Qp) \isorightarrow \Qp, \hspace{12pt} \mu \mapsto \mu(v_{\lambda,j}).
	\end{equation}
	This is $H(\Zp)$-equivariant, as if $h = (h_1,h_2) \in H(\Zp)$, then
	\begin{equation}\label{eq:kappa_lambda,j equivariant}
		(h*\mu)(v_{\lambda,j}) = \mu(h^{-1}*v_{\lambda,j}) = \det(h_1)^{j}\det(h_2)^{-\sw-j} \mu(v_{\lambda,j}).
	\end{equation}
	We can base-extend this to any extension $L/\Qp$ and consider $\kappa_{\lambda,j}$ as a map $V_\lambda^\vee(L) \to L$.\medskip
	
	Similarly, we can dualise the map $v_\Omega : \cA(\zp^\times,\cO_\Omega) \to \cA_\Omega$ from Definition \ref{def:v_Omega} to get
	\begin{equation}\label{eq:kappa omega}
		\kappa_\Omega : \cD_\Omega \longrightarrow \cD(\zp^\times,\cO_\Omega), \hspace{12pt} \mu \mapsto [f \mapsto \mu(v_\Omega(f))] \text{ for } f \in \cA(\zp^\times,\cO_\Omega).
	\end{equation}

	\begin{definition}\label{def:distribution support}
		Let $\beta \geq 1$. 
		\begin{itemize}\s
			\item[(1)] We say that $f \in \cA_\Omega$  has  \emph{support on $\Iw^\beta$} if $f(g) = 0$ for all $g \notin \Iw^\beta$. By the definition of $\cA_\Omega$, this is equivalent to asking $f(n) = 0$ for all $n \in N(\Zp)\backslash N^\beta(\Zp)$.
			
			Similarly, we say $f$ has \emph{support away from $\Iw^\beta$} if $f(g) = 0$ for all $g \in \Iw^\beta$, or equivalently if $f(n) = 0$ for all $n \in N^\beta(\Zp)$.
			
			\item[(2)] Let $\cA_\Omega^\beta \subset \cA_\Omega$  be the subspace of functions supported on $\Iw^\beta$. 
			\item[(3)] If $f \in \cA_\Omega$, define a function $f|_{\Iw^\beta} \in \cA_\Omega^\beta \subset \cA_\Omega$ by multiplying $f$ by the indicator function of (the closed and open subset) $\Iw^\beta$; explicitly,
			\[
				f|_{\Iw^\beta}(g) \defeq \left\{\begin{array}{cl} f(g) &: g \in \Iw^\beta\\
					0 &: \text{otherwise}.\end{array}\right.
			\]
			This is well-defined: it is locally analytic on $N(\Zp)$ as $N^\beta(\Zp)$ is closed and open in $N(\Zp)$, it defines an element of $\cA_\Omega$ by Iwahori decomposition and the definition of $\Iw^\beta$, and by definition it has support on $\Iw^\beta$.
			\item[(4)] We say that $\mu \in \cD_\Omega$ has \emph{support on $\Iw^\beta$} if $\mu(f) = \mu(f|_{\Iw^\beta})$ for all $f \in \cA_\Omega$. Let $\cD_\Omega^\beta \subset \cD_\Omega$ be the subspace of distributions supported on $\Iw^\beta$.
		\end{itemize}
		We similarly write $\cA_\lambda^\beta$, $\cD_\lambda^\beta$, etc.\ when $\Omega = \{\lambda\}$ is a point.
	\end{definition}

The following is a useful alternative criterion for support.

\begin{lemma}\label{lem:support alternative}
	Let $\mu \in \cD_\Omega$. We have 
	\[
		\Big[\mu \in \cD_\Omega^\beta\Big] \iff \Big[\mu(f) = 0\text{ for all }f \in \cA_\Omega\text{ with support away from $\Iw^\beta$}\Big].
	\]
\end{lemma}

\begin{proof}
	Any $f \in \cA_\Omega$ can be written as $f = f|_{\Iw^\beta} + f|_{\Iw\backslash \Iw^\beta}$, where $f|_{\Iw\backslash\Iw^\beta}$ has support away from $\Iw^\beta$. Then $\mu(f) = \mu(f|_{\Iw^\beta}) + \mu(f|_{\Iw\backslash\Iw^\beta})$. The two conditions are now easily seen to be equivalent.
\end{proof}

	There is a natural map $s_\lambda : V_\lambda(L) \to \cA_\lambda^\beta(L)$ given by 
	\[
		s_\lambda(f)(g) = \begin{cases}f(g) &: g \in \Iw^\beta\\0 &: \text{ else}.\end{cases}
	\]
	 Abusing notation, for any $\beta$ we continue to write $r_\lambda : \cD_\lambda^\beta \to V_\lambda^\vee$ for its dual.
	
	\begin{proposition}\label{prop:distribution commute}
		For each classical $\lambda \in \Omega$ and each $j \in \mathrm{Crit}(\lambda)$, the following diagram commutes:
		\[
		\xymatrix@C=20mm{
			\cD_\Omega^1 \ar[r]^-{\mathrm{sp}_\lambda}\ar[d]^-{\kappa_\Omega}&  	
			\cD_\lambda^1(L) \ar[r]^-{r_\lambda} \ar[d]^-{\kappa_\lambda}& 
			V_\lambda^\vee(L) \ar[d]^-{\kappa_{\lambda,j}}\\
			\cD(\Zp^\times, \cO_\Omega)  \ar[r]^-{\mathrm{sp}_\lambda} &
			\cD(\Zp^\times,L) \ar[r]^-{\mu \mapsto \mu(z^j)}  & 
			L.
		}
		\]
	\end{proposition}
	\begin{proof}
		The first square commutes directly from the definitions and the fact that $\mathrm{sp}_\lambda \circ \chi_{\Omega} = \lambda$.
		
		To see the second square commutes, let  $\mu \in \cD_\lambda^1(L)$; then
		\begin{align*}
			\kappa_\lambda(\mu)(z \mapsto z^j) = \int_{\Zp^\times} z^j \cdot d\kappa_\lambda(\mu) &= \int_{\Iw} v_\lambda(z^j) \cdot d\mu = \int_{\Iw^1} v_\lambda(z^j) \cdot d\mu = \int_{\Iw^1} v_{\lambda,j} \cdot d\mu\\
			&= \int_{\Iw} v_{\lambda,j} \cdot d\mu = \int_{\Iw} v_{\lambda,j} \cdot dr_{\lambda}(\mu) = \kappa_{\lambda,j}\circ r_\lambda(\mu),
		\end{align*}
		where for clarity we write $\int_X f \cdot d\mu$ for $\mu(f|_X)$, interpreted suitably for each term. Then the second equality is by definition of $\kappa_\lambda$, the third and fifth equalities follow as $\mu$ is supported on $\Iw^1$, the fourth equality is Lemma \ref{lem:v_lambda interpolates}, the sixth equality is by definition of $r_\lambda$ as $v_{\lambda,j} \in V_\lambda$, and the seventh equality is the definition of $\kappa_{\lambda,j}$.
	\end{proof}
	
	Recall the $H(\Zp)$-action on $\cA(\Zp^\times,\cO_\Omega)$ from \eqref{eq:H(Zp) action on A(Zp*)}. We equip $\cD(\Zp^\times,\cO_\Omega)$ with the left dual action. In the diagram of Proposition \ref{prop:distribution commute}, recall the action of $H(\Zp)$ on the bottom-right term is by $\det(h_1)^{j}\det(h_2)^{-\sw-j}$. Finally we note:
	
	\begin{lemma}\label{lem:big diagram equivariant}
		Every map in the diagram of Proposition \ref{prop:distribution commute} is $\IwH^1$-equivariant.
	\end{lemma} 
	\begin{proof}
		For $\mathrm{sp}_\lambda$ and $r_\lambda$ this follows straight from the definition; for $\kappa_\Omega$ and $\kappa_\lambda$ this is Lemma \ref{lem:v_Omega equivariant}; for $\kappa_{\lambda,j}$ this is \eqref{eq:kappa_lambda,j equivariant}; and for evaluation at $z^j$, this follows since for $\mu \in \cD(\Zp^\times,L)$ and $h \in H(\Zp)$, we have $(h * \mu)(z^j) = \mu(h^{-1} * z^j) = \mu(\det(h_1)^{j}\det(h_2)^{-\sw-j}\cdot z^j).$
	\end{proof}

	\section{Distribution-valued evaluation maps}\label{sec:distribution valued evaluations}
	We now define distribution-valued evaluation maps on the overconvergent cohomology by combining the $p$-adic branching laws of \S\ref{sec:branching laws} with the abstract evaluation maps of \S\ref{sec:abstract evaluation maps}.

	\subsection{Combining evaluations and branching laws}\label{sec:distribution evaluations}

	Recall $H(\Zp)$ acts on $\cA(\Zp^\times,\cO_\Omega)$ by \eqref{eq:H(Zp) action on A(Zp*)}. We have an isomorphism 
	\begin{equation}\label{eq:Zptimes adeles}
		\Zp^\times \cong \Cl(p^\infty) \defeq  \Q^\times \big\backslash \A^\times\big/\prod_{\ell\neq p}\Z_\ell^\times \R_{>0},
	\end{equation} 
	and we extend the $H(\Zp)$-action on $f \in \cA(\Zp^\times,\cO_\Omega)$ to an action of $h = (h_1,h_2) \in H(\A)$ by
	\begin{equation}\label{eq:H(A) action A(Zp*)}
		(h * f)(z) = \chi_{\cyc}(\det(h_2))^{\sw_\Omega} \cdot f\left(\tfrac{\det(h_{2})}{\det(h_{1})} \cdot z\right),
	\end{equation}
	where translation of $z \in \Zp^\times$ under $\det(h_2)\det(h_1)^{-1} \in \A^\times$ is defined by lifting $z$ to $\A^\times$ under \eqref{eq:Zptimes adeles}, translating, and projecting back to $\Zp^\times$. Note $H_\infty^\circ$ and $H(\Q)$ again act trivially, and that any subgroup of $H(\widehat{\Z})$ acts through projection to $H(\Zp)$.
	
	\begin{lemma}\label{lem:kappa omega equivariant}
		If $\beta \geq 1$, then the map $\kappa_{\Omega}$ from \eqref{eq:kappa omega} is $L_\beta$-equivariant.
	\end{lemma}
	\begin{proof}
		Recall that $\kappa_\Omega$ is the dual of $v_\Omega : \cA(\Zp^\times,\cO_\Omega) \to \cA_\Omega$, so it suffices to prove $v_\Omega$ is $L_\beta$-equivariant. Recall $\IwH^1$ from \eqref{eq:Iw_H^1}, and that $v_\Omega$ is $\IwH^1$-equivariant by Lemma \ref{lem:v_Omega equivariant}. A simple check shows that $L_p^\beta \subset \IwH^1$, so $v_\Omega$ is $L_p^\beta$-equivariant. But the $L_\beta$-action on both terms factors through projection to $L_p^\beta$, since $L_\beta \subset H(\A)$ (for $\cA(\Zp^\times,\cO_\Omega)$) and by definition (for $\cA_\Omega$).
	\end{proof}

	To define our distribution-valued evaluations, we take strong motivation from Definition \ref{def:classical evalution delta}, adapting it with the Borel $B$ in place of the parabolic $Q$. Let $\Omega \subset \sW_0^G$ be an affinoid in the pure weight space; we allow $\Omega = \{\lambda\}$ a single weight. Lemma \ref{lem:kappa omega equivariant} allows us to make the following definition:
	
	\begin{definition}\label{def:oc evalution delta}
		Let $\beta \in \Z_{\geq 0}$ and $\delta \in H(\A_f)$, representing $[\delta] \in \pi_0(X_\beta)$. The \emph{overconvergent evaluation map of level $p^\beta$ at $[\delta]$} is the map
		\[
		\mathrm{Ev}_{B,\beta,[\delta]}^{\Omega} \defeq \mathrm{Ev}_{B,\beta,[\delta]}^{\cD_{\Omega}, \kappa_{\Omega}} = \delta * \left[\kappa_\Omega \circ \mathrm{Ev}_{B,\beta,\delta}^{\cD_\Omega}\right] : \hc{t}(S_K,\cD_\Omega) \longrightarrow \cD(\Zp^\times,\cO_\Omega).
		\]
	\end{definition}
	
	This is well-defined and independent of the choice of $\delta$ representing $[\delta]$ by Proposition \ref{prop:ind of delta}. Here we also use that $H(\Q)$ and $H_\infty^\circ$ act trivially on $\cD(\Zp^\times,\cO_\Omega)$.
	
	\begin{proposition}\label{prop:evs main commutative diagram}
		Suppose $\beta \geq 1$. Then for every $\lambda \in \Omega$, and $j \in \mathrm{Crit}(\lambda)$, we have a commutative diagram
		\[
		\xymatrix@C=18mm{
			\hc{t}(S_K,\sD_\Omega) \ar[d]^-{\mathrm{Ev}_{B,\beta,[\delta]}^\Omega}\ar[r]^-{\mathrm{sp}_\lambda} & \hc{t}(S_K,\sD_\lambda(L)) \ar[d]^-{\mathrm{Ev}_{B,\beta,[\delta]}^\lambda}\ar[r]^{r_\lambda} &    \hc{t}(S_K,\sV_\lambda^\vee(L))  \ar[d]^-{\cE_{B,\beta,[\delta]}^{j,\sw}}\\
			\cD(\Zp^\times,\cO_\Omega)\ar[r]^{\mathrm{sp}_\lambda} &
			\cD(\Zp^\times,L)\ar[r]^{\mu \mapsto \mu(z^j)} &
			L}
		\]
	\end{proposition}
	\begin{proof}
		First note that applying Lemma \ref{lem:pushforward} first to $\cD_\Omega \xrightarrow{\mathrm{sp}_\lambda} \cD_\lambda(L)$ and then $\cD_\lambda(L) \xrightarrow{r_\lambda} V_\lambda^\vee(L)$, we have a commutative diagram
		\begin{equation}\label{eq:functorial}
			\xymatrix@C=18mm{
				\hc{t}(S_K,\sD_\Omega) \ar[d]^-{\mathrm{Ev}_{B,\beta,\delta}^{\cD_\Omega}}\ar[r]^-{\mathrm{sp}_\lambda} & \hc{t}(S_K,\sD_\lambda(L)) \ar[d]^-{\mathrm{Ev}_{B,\beta,\delta}^{\cD_\lambda}}\ar[r]^{r_\lambda} &    \hc{t}(S_K,\sV_\lambda^\vee(L))  \ar[d]^-{\mathrm{Ev}_{B,\beta,\delta}^{V^{\vee}_\lambda}}\\
				(\cD_\Omega)_{\Gamma_{\beta,\delta}}\ar[r]^{\mathrm{sp}_\lambda} &
				\cD_\lambda(L)_{\Gamma_{\beta,\delta}}\ar[r]^{r_\lambda} &
				V_\lambda^\vee(L)_{\Gamma_{\beta,\delta}}.
			}
		\end{equation}
		Recall $\cD_\Omega^1 \subset \cD_\Omega$ from Definition \ref{def:distribution support}. Since $\beta \geq 1$, we have $\delta^{-1}\Gamma_{\beta,\delta}\delta \subset \IwH^1$ from \eqref{eq:Iw_H^1}, so the action of $\Gamma_{\beta,\delta}$ on $\cD_\Omega$ preserves $\cD_\Omega^1$, and we can consider $(\cD_\Omega^1)_{\Gamma_{\beta,\delta}}$. Moreover, $H(\Q)$ (hence $\Gamma_{\beta,\delta}$) acts trivially on $\cD(\Zp^\times,\cO_\Omega)$ and $L$, and $\kappa_\Omega$ and $\kappa_{\lambda,j}$ are $\IwH^1$-equivariant; combining, each of these maps factors through the coinvariants, giving maps $(\cD_\Omega^1)_{\Gamma_{\beta,\delta}} \to \cD(\Zp^\times,\cO_\Omega)$ and $V_{\lambda}^\vee(L)_{\Gamma_{\beta,\delta}} \to L$. Then by Proposition \ref{prop:distribution commute} (for the top squares) and Lemma \ref{lem:big diagram equivariant} (for the bottom squares) we have a commutative diagram
		\begin{equation}\label{eq:minus delta}
			\xymatrix@C=24mm{
				(\cD_\Omega^1)_{\Gamma_{\beta,\delta}} \ar[d]^-{\kappa_\Omega}\ar[r]^-{\mathrm{sp}_\lambda} & \cD_\lambda^1(L)_{\Gamma_{\beta,\delta}} \ar[d]^-{\kappa_\lambda}\ar[r]^{r_\lambda} &    V_\lambda^\vee(L)_{\Gamma_{\beta,\delta}}  \ar[d]^-{\kappa_{\lambda,j}}\\
				\cD(\Zp^\times,\cO_\Omega)\ar[r]^{\mathrm{sp}_\lambda}\ar[d]^-{\delta * -} &
				\cD(\Zp^\times,L)\ar[r]^{\mu \mapsto \mu(z^j)}\ar[d]^-{\delta * -} &
				L\ar[d]^-{\delta * -}\\
				\cD(\Zp^\times,\cO_\Omega)\ar[r]^{\mathrm{sp}_\lambda} &
				\cD(\Zp^\times,L)\ar[r]^{\mu \mapsto \mu(z^j)} &
				L.
			}
		\end{equation}
		Since $\mathrm{Ev}_{B,\beta,[\delta]}^\Omega = \delta * [\kappa_\Omega \circ \mathrm{Ev}_{B,\beta,\delta}^{\cD_\Omega}]$ and $\cE_{B,\beta,[\delta]}^{j,\sw} = \delta*[\kappa_{\lambda,j} \circ \mathrm{Ev}_{B,\beta,\delta}^{V_\lambda^\vee}]$, we can complete the proof by combining \eqref{eq:functorial} with \eqref{eq:minus delta}. This is possible by Lemma \ref{lem:Ev support in D_beta} below, noting if $\beta \geq 1$, then $\cD_{\Omega}^\beta \subset \cD_\Omega^1$.
	\end{proof}
	
	\begin{lemma} \label{lem:Ev support in D_beta}
		Let $\beta \geq 1$ and $\Phi \in \hc{t}(S_K,\sD_\Omega)$. We have $\mathrm{Ev}_{B,\beta,\delta}^{\cD_\Omega}(\Phi) \in (\cD_\Omega^\beta)_{\Gamma_{\beta,\delta}}.$
	\end{lemma}
	\begin{proof}
		Note that $\mathrm{Ev}_{B,\beta,\delta}^{\cD_\Omega}(\Phi) \in [(u^{-1} t_p^\beta) * \cD_\Omega]_{\Gamma_{\beta,\delta}}$ by construction of evaluation maps.  Thus it suffices to show that if $\mu \in \cD_\Omega$, then $u^{-1}t_p^\beta * \mu \in \cD_\Omega^\beta$. By Lemma \ref{lem:support alternative}, this is equivalent to showing that $(u^{-1}t_p^\beta * \mu)(f) = 0$ for all $f \in \cA_\Omega$ with support away from $\Iw^\beta$ (that is, such that $f(n) = 0$ for all $n \in N^\beta(\Zp)$). 
			
		Suppose $f \in \cA_\Omega$ has support away from $\Iw^\beta$. By definition (see \S\ref{sec:slope-decomp}), we have
		\[
			(u^{-1}t_p^\beta * \mu)(f) = \mu(ut_p^\beta * f),
		\] 
		where $(ut_p^\beta * f)(n) = f(t_p^\beta n t_p^{-\beta} u)$ for any $n \in N(\Zp)$.		It is easily seen that $t_p^\beta n t_p^{-\beta} \equiv 1_{2n} \newmod{p^\beta}$ (e.g.\ from \cite[\S2.5, Rem.\ 4.19]{BW20}). By definition of $N^\beta(\Zp)$ we deduce $t_p^\beta nt_p^{-\beta}u \in N^\beta(\Zp)$.  Since $f$ vanishes on $N^\beta(\Zp)$, we deduce that $ut_p^\beta * f$ vanishes on $N(\Zp)$, so is the zero function in $\cA_\Omega$. It follows that $(u^{-1}t_p^\beta * \mu)(f) = \mu(0) = 0$ for all $f$ with support away from $\Iw^\beta$, hence $u^{-1}t_p^\beta * \mu \in \cD_\Omega^\beta$ by Lemma \ref{lem:support alternative}. This proves the lemma (and thus Proposition \ref{prop:evs main commutative diagram}).
	\end{proof}

	\subsection{Further support conditions}
	
	Our ultimate goal is to interpolate the classical evaluation maps $\cE_{B,\chi}^{j,\eta_0}$ from \S\ref{sec:classical evaluations}. In Remark \ref{rem:classical diagram}, we described this map as a composition of four maps; and in Proposition \ref{prop:evs main commutative diagram}, we have used the branching laws of \S\ref{sec:branching laws} to interpolate the first two maps of this composition. We will combine over $\beta$ and $\delta$ to interpolate the final two maps in Remark \ref{rem:classical diagram}. First, we give a more precise description of the image of $\mathrm{Ev}_{\beta,[\delta]}^{\Omega}$.
	
	Note that for $\beta \geq 1$, we have decompositions
	\begin{equation}\label{eq:A decomposition}
		\cA(\Zp^\times,\cO_\Omega) = \bigoplus\limits_{\brep \in (\Z/p^\beta)^\times} \cA(\brep + p^\beta\Zp,\cO_\Omega), \hspace{12pt} \cD(\Zp^\times,\cO_\Omega) = \bigoplus\limits_{\brep \in (\Z/p^\beta)^\times} \cD(\brep + p^\beta\Zp,\cO_\Omega).
	\end{equation}
	As in Definition \ref{def:distribution support}, a distribution $\mu \in \cD(\Zp^\times,\cO_\Omega)$ lies in the summand $\cD(\brep + p^\beta\Zp,\cO_\Omega)$ if and only if $\mu(f) = \mu(f|_{\brep + p^\beta\Zp})$ for all $f \in \cA(\Zp^\times,\cO_\Omega)$.
	
	\begin{lemma}\label{lem:kappa omega 1+pbeta}
		If $\mu \in \cD_\Omega^\beta$, then $\kappa_\Omega(\mu) \in \cD(1+p^\beta\Zp,\cO_\Omega).$ 
	\end{lemma}
	
	\begin{proof}
		Let $f \in \cA(\Zp^\times,\cO_\Omega)$. If $g \in N^\beta(\Zp)$, then by Proposition \ref{prop:support in N1}, we know $v_{(n),i}(g) \in 1+p^\beta\Zp$. Thus by the definition of $v_\Omega(f)$, we see $v_\Omega(f)(g) = v_\Omega(f|_{1+p^\beta\Zp})(g)$, that is, 
		\[
		v_\Omega(f)|_{N^\beta(\Zp)} = v_\Omega(f|_{1+p^\beta\Zp})|_{N^\beta(\Zp)}.
		\]
		By the transformation law \eqref{eq:iwahori induction}, the function $v_\Omega(f)|_{\Iw^\beta} \in \cA_\Omega$ depends only on $v_\Omega(f)|_{N^\beta(\Zp)}$,  so we deduce $v_\Omega(f)|_{\Iw^\beta} = v_\Omega(f|_{1+p\Zp})|_{\Iw^\beta}$. Thus if $\mu$ has support on $\Iw^\beta$, for any $f \in \cA(\Zp^\times,\cO_\Omega)$ we have
		\begin{align*}
			\kappa_\Omega(\mu)(f) = \mu[v_\Omega(f)] &= \mu[v_\Omega(f)|_{\Iw^\beta}]\\ 
			&= \mu[v_\Omega(f|_{1+p^\beta\Zp})|_{\Iw^\beta}] = \mu[v_\Omega(f|_{1+p^\beta\Zp})] = \kappa_\Omega(\mu)(f|_{1+p^\beta\Zp}),
		\end{align*}
		so $\kappa_\Omega(\mu)$ is supported on $1+p^\beta\Zp$, as required.
	\end{proof}
	
	Recall $\mathrm{pr}_\beta : \pi_0(X_\beta) \to (\Z/p^\beta)^\times$ from \eqref{eq:pr_beta}.
	\begin{corollary}\label{cor:ev support in beta}
		If $\Phi \in \hc{t}(S_K,\sD_\Omega)$, we have $\mathrm{Ev}_{B,\beta,[\delta]}^\Omega(\Phi) \in \cD(\brep + p^\beta\Zp,\cO_\Omega)$, where $\brep = \mathrm{pr}_\beta([\delta])$.
	\end{corollary}
	
	\begin{proof}
		Recall $\mathrm{Ev}_{B,\beta,[\delta]}^\Omega \defeq \delta *[\kappa_\Omega \circ \mathrm{Ev}_{B,\beta,\delta}^{\cD_\Omega}]$. By Lemma \ref{lem:Ev support in D_beta}, we have $\mathrm{Im}(\mathrm{Ev}_{B,\beta,[\delta]}^{\cD_\Omega}) \subset (\cD_\Omega^\beta)_{\Gamma_{\beta,\delta}}$. Since $\kappa_\Omega$ factors through the coinvariants, Lemma \ref{lem:kappa omega 1+pbeta} implies that $\kappa_\Omega[(\cD_\Omega^\beta)_{\Gamma_{\beta,\delta}}] \subset \cD(1+p^\beta\Zp,\cO_\Omega)$.
		
		Finally the action of $\delta$ on $\cD(\Zp^\times,\cO_\Omega)$ is by $\delta^{-1}$ on $\cA(\Zp^\times,\cO_\Omega)$, which includes translation on $z \in \Zp^\times$ by $\det(\delta_1\delta_2^{-1})$. As this is a representative of $\brep = \mathrm{pr}_\beta([\delta]) \in (\Z/p^\beta)^\times,$ translation by $\det(\delta_1\delta_2^{-1})$ sends $1+p^\beta\Zp$ to $\brep+p^\beta\Zp$. This induces a map $\delta * - : \cD(1+p^\beta\Zp,\cO_\Omega) \to \cD(\brep +p^\beta\Zp,\cO_\Omega)$. Combining all of the above gives the corollary.
	\end{proof}

	\subsection{Interpolation of classical evaluations}
	
	Let $\eta_0$ be any character of $(\Z/m)^\times$. For $\beta \geq 1$ and $\brep \in (\Z/p^\beta\Z)^\times,$ motivated by Definition \ref{def:ev chi}, we define a map
	\begin{align*}
		\mathrm{Ev}_{B,\beta, \brep}^{\Omega,\eta_0} : \htc(S_K, \sD_\Omega) &\longrightarrow \cD(\brep + p^\beta\Zp, \cO_\Omega)\\
		\Phi &\longmapsto \sum_{[\delta] \in \mathrm{pr}_{\beta}^{-1}(\brep)}\eta_{0}\big(\mathrm{pr}_2([\delta])\big) \  \mathrm{Ev}_{B,\beta, [\delta]}^\Omega(\Phi).
	\end{align*}
	Combining under \eqref{eq:A decomposition}, we finally obtain an evaluation map
	\begin{align}\label{eq:final evaluation}
		\mathrm{Ev}_{B,\beta}^{\Omega,\eta_0} \defeq \bigoplus_{\brep \in (\Z/p^\beta)^\times}& \mathrm{Ev}_{B,\beta,\brep}^{\Omega,\eta_0} : \htc(S_K, \sD_\Omega) \longrightarrow \cD(\Zp^\times, \cO_\Omega)\\
		\Phi\   &\longmapsto \sum_{[\delta] \in \pi_0(X_\beta)} \eta_0\big(\mathrm{pr}_2([\delta])\big) \  \times \left(\delta * \left[\kappa_\Omega \circ \mathrm{Ev}_{B,\beta,\delta}^{\cD_\Omega}(\Phi)\right]\right).\notag
	\end{align}

	\begin{remark}\label{rem:overconvergent diagram}
		We have an analogue of Remark~\ref{rem:classical diagram}; $\mathrm{Ev}_{B,\beta}^{\Omega,\eta_0}$ is the composition
		\begin{equation}\label{eq:explicit overconvergent}
			\xymatrix@R=10mm@C=1mm{
				\hc{t}(S_K,\sD_\Omega) \ar@/^3pc/[rrrrrrrrr]_-{\oplus \mathrm{Ev}^\Omega_{B,\beta,[\delta]}} \ar@/_3pc/[rrrrrrrrrrrrrr]^-{\ \ \ \ \ \ \ \ \ \ \ \ \ \sum_{\brep} \mathrm{Ev}_{B,\beta,\brep}^{\Omega,\eta_0}}\ar[rrrr]^-{\oplus\mathrm{Ev}_{B,\beta,\delta}^{\cD_\Omega}}  &&&& \displaystyle\bigoplus_{[\delta]}(\cD_\Omega^\beta)_{\Gamma_{\beta,\delta}} \ar[rrrrr]^-{\delta *\kappa_\Omega} &&&&&
				\displaystyle\bigoplus_{[\delta]}\cD(\mathrm{pr}_\beta([\delta]) + p^\beta\Zp,\cO_\Omega) \ar[rrrrr]^-{\sum_{\brep}\Xi_{\brep}^{\eta_0}} &&&&&
				\cD(\Zp^\times, \cO_\Omega),
			}
		\end{equation}
		where again $\Xi_{\brep}^{\eta_0}$ sends a tuple $(m_{[\delta]})_{[\delta]}$ to $\sum_{[\delta]\in\mathrm{pr}_\beta^{-1}(\brep)} \eta_0(\mathrm{pr}_2([\delta])) \times m_{[\delta]}$.
	\end{remark}
	
	Combining all of the results of this section, we finally deduce:
	
	\begin{proposition}\label{prop:galois evs commute main diagram}
		Suppose $\beta \geq 1$ and $\chi$ is a finite order Hecke character of conductor $p^\beta$. Then for every $\lambda \in \Omega$, and $j \in \mathrm{Crit}(\lambda)$, we have a commutative diagram
		\[
		\xymatrix@C=18mm{
			\hc{t}(S_K,\sD_\Omega) \ar[d]^-{\mathrm{Ev}_{B,\beta}^{\Omega,\eta_0}}\ar[r]^-{\mathrm{sp}_\lambda} & \hc{t}(S_K,\sD_\lambda(L)) \ar[d]^-{\mathrm{Ev}_{B,\beta}^{\lambda,\eta_0}}\ar[r]^{r_\lambda} &    \hc{t}(S_K,\sV_\lambda^\vee(L))  \ar[d]^-{\cE_{B,\chi}^{j,\eta_0}}\\
			\cD(\Zp^\times,\cO_\Omega)\ar[r]^{\mathrm{sp}_\lambda} &
			\cD(\Zp^\times,L)\ar[r]^{\mu \mapsto \mu\big[\chi(z)z^j\big]} &
			L.}
		\]
	\end{proposition}
	\begin{proof}
		By combining Proposition \ref{prop:evs main commutative diagram} with Corollary \ref{cor:ev support in beta}, and taking a direct sum over $[\delta] \in \pi_0(X_\beta)$, there is a commutative diagram
		\begin{equation}\label{eq:galois ev commute 1}
			\xymatrix@C=18mm{
				\hc{t}(S_K,\sD_\Omega) \ar[d]^-{\mathrm{Ev}_{B,\beta,[\delta]}^\Omega}\ar[r]^-{\mathrm{sp}_\lambda} & \hc{t}(S_K,\sD_\lambda(L)) \ar[d]^-{\mathrm{Ev}_{B,\beta,[\delta]}^\lambda}\ar[r]^{r_\lambda} &    \hc{t}(S_K,\sV_\lambda^\vee(L))  \ar[d]^-{\cE_{B,\beta,[\delta]}^{j,\sw}}\\
				\bigoplus_{[\delta]} \cD(\brep + p^\beta\Zp,\cO_\Omega)\ar[r]^{\mathrm{sp}_\lambda} &
				\bigoplus_{[\delta]} \cD(\brep + p^\beta\Zp,L)\ar[r]^{\mu \mapsto \mu(z^j)} &
				\bigoplus_{[\delta]} L,}
		\end{equation}
		where $\brep = \mathrm{pr}_\beta([\delta]) \in (\Z/p^\beta)^\times$. Also, there is a commutative diagram
		\begin{equation}\label{eq:galois ev commute 2}
			\xymatrix@C=18mm{
				\bigoplus_{[\delta]} \cD(\brep + p^\beta\Zp, \cO_\Omega) \ar[r]^-{\mathrm{sp}_{\lambda}}\ar[d]^-{\Xi_{\brep}^{\eta_0}} &	
				\bigoplus_{[\delta]} \cD(\brep + p^\beta\Zp, L) \ar[r]^-{\mu \mapsto \mu(z^j)}\ar[d]^-{\Xi_{\brep}^{\eta_0}} &
				\bigoplus_{[\delta]} L\ar[d]^-{\Xi_{\brep}^{\eta_0}}\\
				\bigoplus_{\brep} \cD(\brep + p^\beta\Zp, \cO_\Omega) \ar[r]^-{\mathrm{sp}_{\lambda}}\ar[d]^-{\sum_{\brep}} &	
				\bigoplus_{\brep} \cD(\brep + p^\beta\Zp, L) \ar[r]^-{\mu \mapsto \mu(z^j)}\ar[d]^-{\sum_{\brep}} &
				\bigoplus_{\brep} L\ar[d]^-{(\ell_{\brep}) \mapsto \sum_{\brep}\chi(\brep)\ell_{\brep}}\\
				\cD(\Zp^\times,\cO_\Omega) \ar[r]^-{\mathrm{sp}_\lambda}& 
				\cD(\Zp^\times,L) \ar[r]^-{\mu \mapsto \mu\big[\chi(z)z^j\big]} &
				L,
			}
		\end{equation}
		where the direct sums are over $[\delta] \in \pi_0(X_\beta)$ and $\brep \in (\Z/p^\beta)^\times$. Indeed the top squares and bottom-left square all commute directly from the definitions; and the bottom right square commutes since for any $\mu \in \cD(\Zp^\times,L)$, we have
		\[
		\int_{\Zp^\times} \chi(z)z^j \cdot d\mu = \sum_{\brep \in (\Z/p^\beta)^\times} \chi(\brep) \int_{\brep+p^\beta\Zp} z^j \cdot \mu.
		\]
		Now, in line with Remarks \ref{rem:classical diagram} and \ref{rem:overconvergent diagram}, the proposition follows by combining \eqref{eq:galois ev commute 1} and \eqref{eq:galois ev commute 2}.
	\end{proof}
	
	\subsection{$p$-adic $L$-functions attached to RASCARs} \label{sec:p-adic L-functions}

	\begin{proposition}\label{prop:oc evaluation changing beta}
		If $\beta \geq 1$, then $\mathrm{Ev}_{B,\beta+1}^{\Omega,\eta_0} = \mathrm{Ev}_{B,\beta}^{\Omega,\eta_0} \circ U_p^\circ  : \hc{t}(S_K,\sD_\Omega) \to \cD(\Zp^\times,\cO_\Omega).$
	\end{proposition}
	\begin{proof}
		This follows from Proposition \ref{prop:evaluations changing beta} (cf.\ \cite[Prop.\ 6.16]{BDW20}).
	\end{proof}
	
	\begin{definition}
		Let $\Phi \in \hc{t}(S_K,\sD_\Omega)$ be a $U_p^\circ$-eigenclass with invertible eigenvalue $\alpha_p^\circ \in \cO_\Omega^\times$, fix $\beta \geq 1$, and define
		\[
		\mu^{\Omega,\eta_0}(\Phi) \defeq (\alpha_p^\circ)^{-\beta} \cdot \mathrm{Ev}_{B,\beta}^{\Omega,\eta_0}(\Phi).
		\]
		By Proposition \ref{prop:oc evaluation changing beta}, this is independent of the choice of $\beta$.
	\end{definition}
	
	Now let $\tilde\pi = (\pi,\alpha)$ be a $p$-refined RACAR of weight $\lambda$ satisfying Conditions~\ref{cond:running assumptions}; so it admits an $(|\cdot|^{\sw},\psi)$-Shalika model, with $\sw$ the purity weight of $\lambda$. In particular, we have $\eta_0 = \mathbf{1}$ trivial.  For $K$ as in \eqref{eq:K}, let $\phi_{\tilde{\pi}}^\pm \in \hc{t}(S_{K},\sV_\lambda^\vee(L))\locpi^\pm$ as \eqref{eq:phi}. By (C4) and Theorem \ref{thm:control}, $\tilde{\pi}$ is non-critical (Definition~\ref{def:non-Q-critical}), hence $\phi_{\tilde{\pi}}^\pm$ lifts uniquely to an eigenclass $\Phi_{\tilde{\pi}}^\pm \in \hc{t}(S_K,\sD_\lambda)^\pm\locpi$ with $U_{p}^\circ$-eigenvalue $\alpha_{p}^\circ$, recalling $\alpha_{p}^\circ = \lambda(t_{p})\alpha_{p}$. 
	
	For $h \in \Q_{\geq 0}$, recall the notion of $\mu \in \cD(\Zp^\times,\cO_\Omega)$ having \emph{growth of order $h$} \cite[Def.\ 3.10]{BDJ17}.

	\begin{definition}\label{def:non-critical slope Lp}
		Let $[\det(-w_n)] \in \cD(\Zp^\times,L)$ be the (bounded) Dirac measure defined by $\int_{\Zp^\times} f \cdot d[\det(-w_n)] = f(\det(-w_n))$.
		
		Let $\cL_p^\pm(\tilde\pi) \defeq \mu^{\lambda,\mathbf{1}}(\Phi_{\tilde\pi}^\pm) \in \cD(\Zp^\times,L).$ Let $\Phi_{\tilde\pi} = \Phi_{\tilde\pi}^+ + \Phi_{\tilde\pi}^-$, and define the \emph{$p$-adic $L$-function} attached to $\tilde{\pi}$ to be 
		\[
			\cL_p(\tilde\pi) = [\det(-w_n)] \cdot \mu^{\lambda,\mathbf{1}}(\Phi_{\tilde\pi}) = [\det(-w_n)] \cdot(\cL_p^+(\tilde\pi) + \cL_p^-(\tilde\pi)).
		\] 
		(The presence of $[\det(-w_n)]$ will become clear). Let $\sX \defeq (\mathrm{Spf}\ \Zp\lsem \Zp^\times \rsem )^{\mathrm{rig}}$. The Amice transform allows us to consider $\cL_p(\tilde\pi, -) : \sX \to \overline{\Q}_p$ as an element of $\cO(\sX)$.
	\end{definition}

	\begin{definition}[Modified Euler factors]  \label{def:modified euler factors} Let $\chi$ be a finite order character of conductor $p^\beta$ with $\beta \geq 0$, and $j \in \mathrm{Crit}(\lambda)$. At $p$, let 
				\[
			e_p(\tilde\pi,\chi,j) \defeq \bigg(\frac{p^{nj + \tfrac{n^2-n}{2}}}{\alpha_{p,n}}\bigg)^\beta \tau(\chi)^n \ \text { for }\chi \neq \mathbf{1}, \qquad \text{and} \qquad e_p(\tilde\pi,\mathbf{1},j) = \prod_{i=n+1}^{2n}
			\frac{1-p^{-1}\alpha_{i,j}^{-1}}{1 - \alpha_{i,j}},
			\]
			where $\alpha_{i,j} = \UPS_i(p)/p^{j+1/2}$, recalling $\pi_p = \Ind_B^G \UPS$. At $\infty$, let
			\[
				e_\infty(\pi, \chi, j) \defeq  i^{-jn}  \cdot L(\pi_\infty \otimes \chi_\infty, j+1/2).
			\]
	\end{definition}

	\begin{theorem} \label{thm:non-ordinary} 	Let $\tilde\pi$ be a $p$-refined RACAR of weight $\lambda_\pi$ satisfying Conditions \ref{cond:running assumptions}. 
		The distribution $\cL_p({\tilde{\pi}})$ has growth of order $h_p = v_p(\alpha_p^\circ)$. For every finite order character $\chi$ of $\Q^\times\backslash\A^\times$ of conductor $p^\beta$ with $\beta \in \Z_{\geq 0}$, and all $j \in \mathrm{Crit}(\lambda)$, we have
		\begin{align*}
			\cL_p(\tilde\pi, \chi(z)z^j) \defeq \int_{\Zp^\times} \chi(z) z^j \cdot d\cL_p(\tilde\pi)
			=  \gamma_{(pm)} \cdot e_p(\tilde\pi,\chi,j) \cdot e_\infty(\pi,\chi,j) \cdot \frac{L^{(p)}\big(\pi\otimes\chi, j+\tfrac{1}{2}\big)}{\Omega_\pi^{\pm}},
		\end{align*}
		where $\pm1 = \chi_\infty(-1)(-1)^j$. Here $\gamma_{(pm)}$ is defined before Corollary \ref{cor:L-value}, and all further notation is as in Definition \ref{def:modified euler factors} and Theorem~\ref{thm:critical value}. 
	\end{theorem}
	
	\begin{proof}
		The growth property follows as in \cite[Prop.\ 6.20]{BDW20}. The Dirac measure is bounded so does not affect the growth.
		
		For ramified $\chi$, the interpolation is analogous to \cite[Thm.\ 6.23]{BDW20}.  In particular, 
		we have
		\begin{align*}
			\int_{\Zp^\times} \chi(z)z^j \cdot d\mu^{\lambda,\mathbf{1}}(\Phi_{\tilde\pi}) &=  (\alpha_p^\circ)^{-\beta}\Big[\mathrm{Ev}_{B,\beta}^{\lambda,\eta_0}(\Phi_{\tilde\pi})\Big](\chi(z)z^j)\\
			&= (\alpha_p^\circ)^{-\beta}\Big[\cE_{B,\chi}^{j,\eta_0} \circ r_\lambda\Big](\Phi_{\tilde\pi}) \qquad \text{(by Prop.\ \ref{prop:galois evs commute main diagram})}\\
			&=  (\alpha_p^\circ)^{-\beta}\cE_{B,\chi}^{j,\eta_0}(\phi_{\tilde\pi}) \qquad\qquad\ \  \text{(by definition of $\Phi_{\tilde\pi}$)}.
		\end{align*}
		 We  computed $\cE_{B,\chi}^{j,\eta_0}(\phi_{\tilde\pi})$ in Corollary \ref{cor:L-value}; the term $(\alpha_p^\circ)^{-\beta}$ here combines with the $(\alpha_{p}^\circ/\alpha_{p,n})^{\beta}$ there to give $\alpha_{p,n}^{-\beta}$. This is the difference between $Q'(\pi,\chi,j)$ there and $e_p(\tilde\pi, \chi, j)$ here, so we get the claimed factors at $p$.
		 
		 We now have a term $\chi(\det(-w_n))\zetainfty$. We recall that $\zetainfty$ depends linearly on the choice of branching law $\kappa_{\lambda,j}$ and the choice $\Xi_\infty^\pm$ from \eqref{eq:pi to cohomology}. We may make the same choice of $\Xi_\infty^\pm$ as in \cite{BDW20}, and then the analogous integral was computed in \cite[\S5.4]{BDW20} -- using the period relations proved by Jiang--Sun--Tian and Geng \cite{JST, Geng} -- and shown to equal $e_\infty(\pi,\chi,j)$. To obtain the relation here, we need only compare the branching laws $\kappa_{\lambda,j}^{\mathrm{BDW}}$ of \cite[\S5.2]{BDW20} with the ones $\kappa_{\lambda,j}$ of \S\ref{sec:branching laws} here. By definition, $\kappa_{\lambda,j}^{\mathrm{BDW}}$ (resp.\ $\kappa_{\lambda,j}$) is dual to an element $v_{\lambda,j}^{\mathrm{BDW}}$ (resp.\ $v_{\lambda,j}$) of $V_{\lambda}$. Since both lie in a line we know there exists $C_j$ such that $v_{\lambda,j}^{\mathrm{BDW}} = C_jv_{\lambda,j}$. Both are algebraic functions on $G(\Zp)$, and we evaluate them at $u = \smallmatrd{1_n}{w_n}{0}{1_n}$. 
		 
		 By Lemmas 3.6 and 5.11 of \cite{BDW20}, $v_{\lambda,j}^{\mathrm{BDW}}\smallmatrd{1}{1}{0}{1} = (-1)^{jn}\det(w_n)^j = \det(-w_n)^j$. We have chosen (in Proposition \ref{prop:support in N1}) that $v_{\lambda,j}(u) = 1$ always. Thus $C_j = \det(-w_n)^j$, and we deduce
		 \[
		 	\zetainfty = \det(-w_n)^j \cdot e_{\infty}(\pi,\chi,j).
		 \]
		 Combining all of the above, for any ramified character we get exactly the claimed formula, but with an additional factor of $\chi(\det(-w_n))(\det(-w_n))^j$. This is removed when we multiply by the Dirac measure $[\det(-w_n)]$.

This leaves interpolation at the unramified character $\chi = \mathbf{1}$. We will complete the proof in this case in Proposition \ref{prop:unramified interpolation}.
	\end{proof}

	\begin{remark}
		The factors $e_p(\tilde\pi,\chi,j)$ and $e_\infty(\pi,\chi,j)$ are exactly as predicted by the Coates--Perrin-Riou conjecture from \cite{coates89}; for $p$, this is mostly explained in \cite[\S3]{AG94}.
	\end{remark}

	\section{Shalika families and $p$-adic $L$-functions}\label{sec:shalika families}
	
	We fix a sufficiently large coefficient field $L/\Qp$ and drop it from most of the notation. Let $\tilde\pi$ be a $p$-refined RACAR of weight $\lambda_\pi$ satisfying (C1-4) of Conditions \ref{cond:running assumptions}. Recall $K = \Iw\prod_{\ell \neq p}\GL_{2n}(\Z_\ell)$.

	\subsection{Existence and \'etaleness of Shalika families}
	
	This entire subsection is dedicated to the proof of Theorem \ref{thm:shalika family} below, in which we will use our evaluation maps to construct Shalika families. The proof closely follows the proofs of \cite[Thms.\ 7.6, 8.12]{BDW20}.

	\begin{definition}\label{def:local piece S}
		Modify Definition \ref{def:hecke algebra p} by defining normalised Hecke algebras $\cH_p^\circ \defeq \Z[U_{p,r}^\circ : r = 1,...,2n-1]$, and $\cH^\circ \defeq \cH'\otimes\cH_p^\circ$. At single weights the $\cH$ and $\cH^\circ$ eigenspaces in cohomology agree, but (unlike $\cH$) the normalised algebra $\cH^\circ$ acts on the cohomology in families.
		
		Define $\T\Uha$  to be the image of $\cH^\circ\otimes\cO_\Omega$ in $ \End_{\cO_\Omega}\big(\htc(S_K,\sD_{\Omega})^{\leqslant h}\big)$. Define $\sE\Uha \defeq \mathrm{Sp}(\T\Uha)$, a rigid analytic space.
		
		Let $w : \sE\Uha \rightarrow \Omega$ be the \emph{weight map} induced by the structure map  $\cO_\Omega \rightarrow \T\Uha$. Also write $\bT^{\pm}_{\Omega,h}$ and $\sE^{\pm}_{\Omega,h}$ for the analogues using $\pm$-parts of the cohomology. By \cite[Thm.~3.2.1]{JoNew}, $\sE_{\Omega,h}^{\pm}$ embeds as a closed subvariety of $\sE_{\Omega,h}$, and $\sE_{\Omega,h} = \sE_{\Omega,h}^{+}\sqcup \sE_\Omega^-$. 
	\end{definition}
	
	By definition, $\sE\Uha$ is a rigid space whose $L$-points $y$ biject with non-trivial homomorphisms $\T\Uha \rightarrow L$, i.e.\ with systems of eigenvalues of  $\psi_y^\circ : \cH^\circ \to L$ appearing  in $\htc(S_K,\sD_{\Omega})\ssh$.

	\begin{definition}
		A point $y \in \sE_{\Omega,h}$ is \emph{classical} if there exists a cohomological automorphic representation $\pi_y$ of $G(\A)$ of weight $\lambda_y \defeq w(y)$ such that $\psi_y^\circ$ appears in $\pi_y^{K}$, whence $\tilde\pi_y = (\pi_y,\alpha_y^\circ)$ is a $p$-refined RACAR (where $\alpha_y^\circ = \psi_y^\circ|_{\cH^\circ_p})$.  A classical point $y$ is \emph{cuspidal} if $\pi_y$ is. 
		A \emph{$(\mathbf{1},\psi)$-Shalika point} is a classical cuspidal point $y$ such that $\pi_y$ admits an $(|\cdot|^{\sw_y},\psi)$-Shalika model, for $\sw_y$ the purity weight of $\lambda_y$.
		
		A \emph{classical family} in $\sE_{\Omega,h}$ is an irreducible component $\sI$ in $\sE_{\Omega,h}$ containing a Zariski-dense set of classical points. A \emph{$(\mathbf{1},\psi)$-Shalika family} is a classical family containing a Zariski-dense set of $(\mathbf{1},\psi)$-Shalika points.
	\end{definition}

	Since $\tilde\pi$ satisfies Conditions \ref{cond:running assumptions}, it is strongly non-critical by (C4) and Theorem \ref{thm:control}. Let $\Lambda = \cO_{\Omega,\m_{\lambda_\pi}}$ be the algebraic localisation of $\cO_\Omega$ at $\lambda_\pi$.

	\begin{lemma}\label{lem:spec isomorphism}
		We have (Hecke-equivariant) isomorphisms
		\[
		\hc{t}(S_K,\sD_\Omega)_{\tilde\pi}^\pm \otimes \Lambda/\m_{\lambda_\pi} \cong \hc{t}(S_K,\sD_{\lambda_\pi})_{\tilde\pi}^\pm \cong \hc{t}(S_K,\sV_{\lambda_\pi}^\vee)^\pm_{\tilde\pi}.
		\]
	\end{lemma}
	\begin{proof}
		The first isomorphism is proved identically to \cite[Lem.\ 2.9]{BDJ17} (cf.\ \cite[Prop.\ 7.8]{BDW20}). The second follows from non-criticality of $\tilde\pi$.
	\end{proof}
	
	Let $\sC^\pm = \mathrm{Sp}(T^\pm)$ be the connected components of $\sE_{\Omega,h}^\pm$ through $x_{\tilde\pi}^\pm$. 
	
	\begin{corollary}\begin{enumerate}[(i)]\setlength{\itemsep}{0pt} \label{cor:cyclic}
			\item There exist ideals $I_{\tilde\pi}^\pm \subset \Lambda$ such that $\hc{t}(S_K,\sD_\Omega)_{\tilde\pi}^\pm \cong \Lambda/I_{\tilde\pi}^\pm$. 
			\item  Possibly shrinking $\Omega$, there exist ideals $I^\pm_{\sC} \subset \cO_\Omega$ such that $\hc{t}(S_K,\sD_\Omega)^{\leq h,\pm} \otimes_{\bT_{\Omega,h}^\pm} T^\pm \cong \cO_\Omega/I_{\sC}^\pm$.
		\end{enumerate}
	\end{corollary}
	\begin{proof}
		(i) By Proposition \ref{prop:mult one}, the right-hand side in Lemma \ref{lem:spec isomorphism} is a line. Since $\hc{t}(S_K,\sD_\Omega)_{\tilde\pi}^\pm \cong \hc{t}(S_K,\sD_\Omega)^{\leq h, \pm}_{\tilde\pi}$ is finite over $\Lambda$, we may use Nakayama's lemma, whence $\hc{t}(S_K,\sD_\Omega)^\pm_{\tilde\pi}$ is generated by one element over $\Lambda$; but every cyclic $\Lambda$-module has the form $\Lambda/I_{\tilde\pi}^\pm$ for some $I_{\tilde\pi}^\pm$.
		
		(ii) This follows from rigid delocalisation of (i) (as in \cite[Cor.\ 4.7]{BW18} and \cite[\S2.7]{BW-Iwasawa}).
	\end{proof}
	
	\begin{proposition}\label{prop:etale} Suppose $\lambda_{\pi,n} > \lambda_{\pi,n+1}$. Then, up to shrinking $\Omega$:
		\begin{enumerate}[(i)]\setlength{\itemsep}{0pt}
			\item $T^\pm$ is free of rank one over $\cO_\Omega$.
			\item $\hc{t}(S_K,\sD_\Omega)^{\leq h,\pm} \otimes_{\bT_{\Omega,h}^\pm} T^\pm$ is free of rank one over $T^\pm$.
		\end{enumerate}
	\end{proposition}
	\begin{proof}
		By the weight condition, as in \cite[Lem.\ 7.4]{BDW20}, there exist $\beta \geq 1$, $j \in \mathrm{Crit}(\lambda_\pi)$ and finite order Hecke characters $\chi^\pm$ of conductor $p^\beta$ with $\chi_\infty^\pm(-1)(-1)^j = \pm 1$ such that $L^{(p)}(\pi\otimes\chi^\pm ,j+\tfrac{1}{2}) \neq 0$. 
		
		As the $\sC^\pm$ are connected components, there exist idempotents $e^\pm$ such that $T^\pm = e^\pm \bT_{\Omega,h}^\pm$, and $\hc{t}(S_K,\sD_\Omega)^{\leq h,\pm} \otimes_{\bT_{\Omega,h}^\pm}T^\pm = e^\pm \hc{t}(S_K,\sD_\Omega)^{\leq h,\pm} \subset \hc{t}(S_K,\sD_\Omega)^{\leq h,\pm}$. Restricting $\mathrm{Ev}_{B,\beta}^{\Omega,\mathbf{1}}$ from \eqref{eq:final evaluation}, we get
		\[
		\mathrm{Ev}_{B,\beta}^{\Omega,\mathbf{1}}	: \hc{t}(S_K,\sD_\Omega)^{\leq h,\pm} \otimes_{\bT_{\Omega,h}^\pm} T^\pm \to \cD(\Zp^\times,\cO_\Omega).
		\]
		Using Corollary \ref{cor:cyclic}, let $\Phi_{\sC}^\pm$ be a generator of $\hc{t}(S_K,\sD_\Omega)^{\leq h,\pm} \otimes_{\bT_{\Omega,h}^\pm} T^\pm$ over $\cO_\Omega$. Note that $\mathrm{Ann}_{\cO_\Omega}(\Phi_{\sC}^\pm) = I^\pm_{\sC}$. Then via the interpolation of Theorem \ref{thm:non-ordinary}, we have
		\[
		\int_{\Zp^\times} \chi^\pm(z)z^j \cdot d\left(\mathrm{sp}_{\lambda_\pi}\circ\mathrm{Ev}_{B,\beta}^{\Omega,\mathbf{1}}(\Phi_{\sC}^\pm)\right) = (*) \times L^{(p)}(\pi\otimes\chi^\pm, j+\tfrac{1}{2}) \neq 0,
		\]
		where $(*)$ is non-zero. As $\mathrm{sp}_{\lambda_\pi}\left(\mathrm{Ev}_{B,\beta}^{\Omega,\mathbf{1}}(\Phi_{\sC}^\pm)\right) \neq 0$, we deduce $\mathrm{Ev}_{B,\beta}^{\Omega,\mathbf{1}}(\Phi_{\sC}^\pm) \neq 0$, that is, it is not the zero distribution. Since $\cD(\Zp^\times,\cO_\Omega)$ is a torsion-free $\cO_\Omega$-module, it follows that $\mathrm{Ann}_{\cO_\Omega}(\Phi_{\sC}^\pm) = 0$ (cf.\ \cite[\S7.3]{BDW20}). Thus $I^\pm_{\sC} = 0$, and $\hc{t}(S_K,\sD_\Omega)^{\leq h,\pm}\otimes_{\bT_{\Omega,h}^\pm}T^\pm \cong \cO_\Omega$.
		
		We deduce $T^\pm$ is the image of $\cH^\circ$ in $\mathrm{End}_{\cO}(\Omega)$. Since this image is non-zero we deduce (i). Finally since the actions of $\cO_\Omega$ and $T^\pm$ are compatible on $\hc{t}(S_K,\sD_\Omega)^{\leq h,\pm}\otimes_{\bT_{\Omega,h}^\pm}T^\pm$, and both $T^\pm$ and $\hc{t}(S_K,\sD_\Omega)^{\leq h,\pm}\otimes_{\bT_{\Omega,h}^\pm}T^\pm$ are free rank one $\cO_\Omega$-modules, we deduce (ii).
	\end{proof}
	
	We finally arrive at the main result of this section.
	
	\begin{theorem}\label{thm:shalika family}
		Let $\tilde\pi$ be a $p$-refined RACAR of weight $\lambda_\pi$ satisfying Conditions \ref{cond:running assumptions}. Suppose $\lambda_{\pi,n} > \lambda_{\pi,n+1}$. Then for any $h \geq v_p(\alpha_p^\circ)$, we have:
		\begin{enumerate}[(i)]\setlength{\itemsep}{0pt}
			\item There exists a point $x_{\tilde\pi} \in \cE_{\Omega,h}$ attached to $\tilde\pi$, and $w: \cE_{\Omega,h} \to \Omega$ is \'etale at $x_{\tilde\pi}$.
			\item The connected component $\sC = \mathrm{Sp}(T)$ in $\cE_{\Omega,h}$ through $x_{\tilde\pi}$ contains a very Zariski-dense set $\sC_{\mathrm{nc}}$ of classical points corresponding to $p$-refined RACARs $\tilde\pi_y$.
			\item There exist Hecke eigenclasses $\Phi_{\sC}^\pm \in \hc{t}(S_K,\sD_\Omega)^{\pm}$ such that for every $y \in \sC_{\mathrm{nc}}$, the specialisation $\mathrm{sp}_{\lambda_y}(\Phi_{\sC}^\pm)$ generates $\hc{t}(S_K,\sD_{\lambda_y})^\pm_{\tilde\pi_y}$, where $\lambda_y \defeq w(y)$.
			\item Up to shrinking $\Omega$, for each $y \in \sC_{\mathrm{nc}}$ the $p$-refined RACAR $\tilde\pi_y$ satisfies Conditions \ref{cond:running assumptions}.
		\end{enumerate} 
	\end{theorem}
	
	\begin{proof}
		(i) Lemma \ref{lem:spec isomorphism} and Proposition \ref{prop:mult one} show $\m_{\tilde\pi}$ appears in $\hc{t}(S_K,\sD_\Omega)^{\leq h,\pm}$, giving points $x_{\tilde\pi}^\pm \in \cE_{\Omega,h}^\pm$. Moreover Proposition \ref{prop:etale} shows $\sE_{\Omega,h}^\pm \to \Omega$ is \'etale at $x_{\tilde\pi}^\pm$.  Let $\sC^\pm \subset \sE_{\Omega,h}^\pm$ be the unique connected component through $x_{\tilde\pi}^\pm$. Using strong non-criticality of $\tilde\pi$, we deduce that $\sC^\pm$ contain very\footnote{In \cite[Prop.\ 5.15]{BW20}, only Zariski-density of classical points is treated; but the same proof shows very Zariski-density, using that the classical weights are very Zariski-dense in $\Omega$.}  Zariski-dense sets $\sC_{\mathrm{nc}}^\pm$ of cuspidal non-critical slope classical points by \cite[Prop.\ 5.15]{BW20}. Now, as in \cite[\S8.3.4]{BDW20}, we can exhibit a bijection between $\sC_{\mathrm{nc}}^+$ and $\sC_{\mathrm{nc}}^-$ and a canonical isomorphism $\sC^+ \isorightarrow \sC^-$, whence $\sC\cong \sC^+ \cong \sC^-$ is independent of sign. Part (i) follows immediately.
		
		(ii) We let $\sC_{\mathrm{nc}} = \sC_{\mathrm{nc}}^+ = \sC_{\mathrm{nc}}^-$ be the set used in (i).
		
		(iii) Let $\Phi_{\sC}^\pm$ be $\cO_\Omega$-module generators of 
		\[
			\hc{t}(S_K,\sD_\Omega)^{\leq h,\pm}\otimes_{\bT_{\Omega,h}^\pm}T^\pm \subset \hc{t}(S_K,\sD_\Omega)^{\leq h,\pm} \subset \hc{t}(S_K,\sD_\Omega)^{\pm},
		\] which are Hecke eigenclasses by Proposition \ref{prop:etale}. For each $y \in \sC_{\mathrm{nc}}$, let $\m_y \subset T^\pm$ be the attached maximal ideal. Reduction modulo $\m_{\lambda_y}$ induces a map
		\[
		\mathrm{sp}_{\lambda_y} : \hc{t}(S_K,\sD_\Omega)^{\leq h,\pm} \otimes_{\bT_{\Omega,h}^\pm}T^\pm \twoheadrightarrow \hc{t}(S_K,\sD_{\lambda_y})^\pm_{\tilde\pi_y},
		\]
		which is surjective by combining \'etaleness of $w$ at $y$ with Lemma \ref{lem:spec isomorphism}. By Proposition \ref{prop:etale}, we deduce $\hc{t}(S_K,\sD_{\lambda_y})^\pm_{\tilde\pi_y}$ is a line, generated by $\mathrm{sp}_{\lambda_y}(\Phi_{\sC}^\pm)$.

		(iv) Every $y \in \sC_{\mathrm{nc}}$ has non-critical slope, hence satisfies (C4). Recall $\pi_y^K \cong \hc{t}(S_K,\sV_\lambda^\vee)^\pm_{\tilde\pi_y}$ by \eqref{eq:pi to cohomology}. As in (iii), the right-hand side is a line (using non-criticality), so $\pi_y^K \neq 0$. For each $\ell \neq p$, this ensures $\pi_\ell^{G(\Z_\ell)} \neq 0$, so $\pi_{y,\ell}$ is spherical, giving (C3). 
		
		Now, let $\beta, j, \chi^\pm$ be as in the proof of Proposition \ref{prop:etale}, and define a map
		\begin{align*}
			\mathrm{Ev}_{B,\chi,j}^\Omega : \hc{t}(S_K,\sD_\Omega)^\pm &\longrightarrow \cO_\Omega\\
			\Phi &\longmapsto \int_{\Zp^\times}\chi(z)z^j \cdot d\mathrm{Ev}_{B,\beta}^{\Omega,\mathbf{1}}(\Phi).
		\end{align*}
		As in the proof of Proposition \ref{prop:etale}, we have $\mathrm{Ev}_{B,\chi,j}^\Omega(\Phi_{\sC}^\pm)(\lambda_\pi) = (*) \times L^{(p)}(\pi\otimes\chi^\pm, j+\tfrac{1}{2}) \neq 0$. Possibly shrinking $\Omega$, we may thus suppose $\mathrm{Ev}_{B,\chi,j}^\Omega(\Phi_{\sC}^\pm)$ is everywhere non-vanishing on $\Omega$. Let 
		\[
		\phi_{\tilde\pi_y}^\pm \defeq r_{\lambda_y} \circ \mathrm{sp}_{\lambda_y}(\Phi_{\sC}^\pm) \in \hc{t}(S_K,\sV_\lambda^\vee)^\pm_{\tilde\pi_y}.
		\]
		Then by Proposition \ref{prop:galois evs commute main diagram}, we have 
		\[
		\cE_{\chi}^{j,\mathbf{1}}(\phi_{\tilde\pi_y}^\pm) = \int_{\Zp^\times} \chi(z)z^j d\left[\mathrm{Ev}_{B,\beta}^{\Omega,\mathbf{1}} \circ \mathrm{sp}_{\lambda_y}\right] = \mathrm{Ev}_{B,\chi,j}^\Omega(\Phi_{\sC}^\pm)(\lambda_y) \neq 0.
		\]
		We deduce $\pi_y$ satisfies (C1) by Proposition \ref{prop:shalika non-vanishing}, i.e. $\pi_y$ admits a $(|\cdot|^{\sw_y},\psi)$-Shalika model. 
		
		It remains to show (C2). By \cite[Prop.\ 5.5]{classical-locus}, up to replacing $\sC_{\mathrm{nc}}$ with a smaller (but still very Zariski-dense) subset, we may assume $\pi_{y,p}$ is spherical for all $y \in \sC_{\mathrm{nc}}$.  For such $y$, let $\alpha_{y,p,r}^\circ = \alpha_y^\circ(U_{p,r}^\circ)$ be the $U_{p,r}$-eigenvalue of $\phi_{\tilde\pi_y}^\pm$. As $\hc{t}(S_K,\sV_\lambda^\vee)^\pm_{\tilde\pi_y}$ is a line, by \eqref{eq:pi to cohomology} again we deduce that $\cS_\psi^{|\cdot|^{\sw_y}}(\pi_{y,p}^{\Iw})[U_{p,r}^\circ - \alpha_{y,p,r}^\circ : r = 1,...,2n-1]$ is a line, so $\tilde\pi_y$ is a regular $p$-refinement. Let $W_{y,p}$ be a generator; then we can relate $W_{y,p}\smallmatrd{w_n}{}{}{1}$ to a non-zero multiple of $\mathrm{Ev}_{B,\chi,j}^\Omega(\phi_{\tilde\pi_y}^\pm)$ exactly as in \cite[\S8.3.5]{BDW20}. In particular, $W_{y,p}\smallmatrd{w_n}{}{}{1} \neq 0$, so $\tilde\pi_y$ is Shalika. It then follows that $\tilde\pi_y$ is spin, by the non-critical slope special case of Expectation \ref{conj:shalika = spin} proved in \cite{classical-locus}: particularly, combining \cite[Thm.\ 8.9, Prop.\ 8.7, Lem.\ 5.1]{classical-locus}.
	\end{proof}

	\subsection{$p$-adic $L$-functions in Shalika families}
	
	We finally give our main construction, of the variation of $p$-adic $L$-functions of RASCARs in pure weight families. Let $\tilde\pi$ of weight $\lambda_{\pi}$ satisfy (C1-4) of Conditions \ref{cond:running assumptions}, and suppose $\lambda_{\pi,n} > \lambda_{\pi,n+1}$. By Theorem \ref{thm:shalika family}(i), the eigenvariety for $G$ is \'etale over weight space at $\tilde\pi$, and by Theorem \ref{thm:shalika family}(iv) its connected component $\sC$ through $\tilde\pi$ contains a very Zariski-dense set $\sC_{\mathrm{nc}}$ of classical points satisfying Conditions \ref{cond:running assumptions}. Let $\Phi_{\sC}^\pm \in \hc{t}(S_K,\sD_\Omega)^\pm$ be the classes of Theorem \ref{thm:shalika family}(iii). Possibly rescaling by $\cO_\Omega^\times$, we may always assume $\mathrm{sp}_{\lambda_\pi}(\Phi_{\sC}^\pm) = \Phi_{\tilde\pi}^\pm$.

	\begin{definition} \label{def:Lp families} Let $\cL_{p}^{\sC, \pm} \defeq [\det(-w_n)]\mu^{\Omega,\mathbf{1}}(\Phi_{\sC}^\pm)$, now considering the Dirac measure as being in $\cD(\Zp^\times,\cO_\Omega)$. Also let $\Phi_{\sC} = \Phi_{\sC}^+ + \Phi_{\sC}^- \in \hc{t}(S_{K(\tilde\pi)},\sD_\Omega)$, a Hecke eigenclass, and define the \emph{$p$-adic $L$-function over $\sC$} to be 
		\[
		\cL_p^{\sC} \defeq [\det(-w_n)] \mu^{\Omega,\mathbf{1}}(\Phi_{\sC}) = \cL_p^{\sC,+} + \cL_p^{\sC,-} \in \cD(\Gal_p,\cO_\Omega).
		\]
		Via the Amice transform  (cf.\ Definition~\ref{def:non-critical slope Lp}), we consider $\cL_p^{\sC}$ as a rigid function $\sC \times \sX \to \overline{\Q}_p$.
	\end{definition}

	\begin{theorem}\label{thm:family p-adic L-functions}
		Let $y \in \sC_{\mathrm{nc}}$ be a classical cuspidal point attached $p$-refined RACAR $\tilde\pi_y$ satisfying Conditions \ref{cond:running assumptions}. There exist $p$-adic periods $c_y^\pm \in L^\times$ such that
		\[
		\cL_p^{\sC,\pm}(y, -) = c_y^\pm \cdot \cL_p^\pm(\tilde\pi_y, -)
		\]
		as functions $\sX \to \overline{\Q}_p$. In particular, $\cL_p^{\sC}$ satisfies the following interpolation: for any $j \in \mathrm{Crit}(w(y))$, and for any finite order Hecke character $\chi$ of conductor $p^\beta > 1$, we have
		\begin{equation}\label{eq:spec in families}
			\cL_p^{\sC}(y,\ \chi(z)z^j) = c_y^\pm \cdot \gamma_{(pm)} \cdot e_p(\tilde\pi_y,\chi,j) \cdot e_\infty(\pi_y,\chi,j) \cdot L^{(p)}(\pi_y\times\chi, j+\tfrac{1}{2})/\Omega_{\pi_y}^\pm,
		\end{equation}
		where $\chi_\infty(-1)(-1)^j = \pm 1$ and with notation as in Theorem~\ref{thm:non-ordinary}. Finally we have $c_{x_{\tilde\pi}}^\pm = 1$.
	\end{theorem}
	
	The `$p$-adic periods' $c_y^\pm$ $p$-adically align the natural algebraic structures in $\{\tilde\pi_y : y \in \sC_{\mathrm{nc}}\}.$

	\begin{proof}
		Let $y$ be as in the theorem. Using that $y$ satisfies Conditions \ref{cond:running assumptions}, let $W_{y,f} \in \cS_{\psi_f}^{\eta_{y,f}}(\pi_{y,f}^{K})$ be as defined as in \eqref{eq:W_f}, and pick complex periods $\Omega_{\pi_y}^\pm$ as in \S\ref{sec:running assumptions}. Since $y \in \sC$ is defined over $L$, as in \S\ref{sec:running assumptions}, there exists a class 
		\[
		\phi_{y}^\pm\defeq \Theta^\pm(W_{y,f})\big/i_p(\Omega_{\pi_y}^\pm) \in \hc{t}(S_{K}, \sV_{\lambda_y}^\vee(L))_{\tilde\pi_y}^\pm.
		\] 
		As $y$ is non-critical, we can lift $\phi_{y}^\pm$ to a non-zero class $\Phi_{y}^\pm \in \hc{t}(S_{K},\sD_{\lambda_y}(L))^\pm_{\tilde\pi_y}$. By Theorem \ref{thm:shalika family}(iii), this space is  $L\cdot \mathrm{sp}_{\lambda_y}(\Phi_{\sC}^\pm)$, so there exists $c_y^\pm \in L^\times$ such that 
		\[
		\mathrm{sp}_{\lambda_y}(\Phi_{\sC}^\pm) = c_y^\pm \cdot \Phi_{y}^\pm.
		\]
		By definition, $\cL_p^\pm(\tilde\pi_y) = [\det(-w_n)]\mathrm{Ev}^{B,\lambda_y,\mathbf{1}}(\Phi_{y}^\pm).$ As evaluation maps commute with weight specialisation (Proposition \ref{prop:galois evs commute main diagram}), we deduce $\mathrm{sp}_{\lambda_y}(\cL_p^{\sC,\pm}) = c_y^\pm \cdot \cL_p^\pm(\tilde\pi_y)$, which when combined with Theorem \ref{thm:non-ordinary} gives \eqref{eq:spec in families}. Finally, our normalisation of $\Phi_{\sC}^\pm$ ensures $c_{x_{\tilde\pi}}^\pm = 1$.
	\end{proof}

\section{Comparison to existing constructions} \label{sec:parahoric-vs-iwahoric}

We finally show that the $p$-adic $L$-functions we've constructed at Iwahori level agree with previous constructions, and deduce their interpolation property at unramified characters.

 Let  $\tilde\pi$ be a non-critical slope regular (Iwahoric) spin $p$-refinement. Then $\tilde\pi^Q \defeq (\pi, \alpha_{p,n})$ is a non-$Q$-critical slope Shalika $Q$-refinement as in \cite[\S2.7,\S3.5]{BDW20}, where being Shalika follows by combining (the proof of) \cite[Thm.\ 4]{Roc20} with \cite[Lem.\ 3.6]{DJR18}. Now \cite[Thm.\ A]{BDW20} attaches a $p$-adic $L$-function $\cL_p(\tilde\pi^Q) \in \cD(\Zp^\times,L)$ to $\tilde\pi^Q$.
	
	\begin{proposition}\label{cor:BDW}
There exists a constant $\Upsilon \in \Q$, independent of $\pi$, such that 
\[
	\cL_p(\tilde\pi) = \Upsilon \cdot \cL_p(\tilde\pi^Q).
\]
	\end{proposition}
	\begin{proof}
Since $\tilde\pi^Q$ is non-$Q$-critical slope, we have
\begin{equation}\label{eq:very small slope}
v_p(\alpha_{p,n}^\circ) < \#\mathrm{Crit}(\lambda_\pi),
\end{equation}
 where $\lambda_\pi$ is the weight of $\pi$. 	In this case, both distributions have sufficiently small growth that they are uniquely determined by their interpolation properties at ramified characters (see e.g.\ \cite[Lem.\ 1.2.5]{DucNam}). Their respective interpolation properties agree exactly except at the volume terms at $p$ (denoted $\gamma$ in \cite{DJR18,BDW20} and $\gamma_{(pm)}$ in the present paper). In particular at every character, the interpolation formulas differ by a rational number $\Upsilon$ independent of $\pi$. The result follows.
\end{proof}

\begin{remark}
That the $p$-adic $L$-function depends only on the $Q$-refinement, not the full Iwahori refinement, should be expected; it is predicted by the Panchiskin condition from \cite{Pan94}.
\end{remark}

Now let $\tilde\pi^Q$ be any non-$Q$-critical $Q$-refinement of $\pi$, and let $\cL_p(\tilde\pi^Q)$ be the $p$-adic $L$-function attached by \cite[Thm.\ A]{BDW20}. The following strengthens the results of \cite{DJR18,BDW20}.

\begin{proposition} \label{prop:unramified interpolation}
The $p$-adic $L$-function $\cL_p(\tilde\pi^Q)$ satisfies the interpolation
		\begin{align*}
			\cL_p(\tilde\pi, z^j)
			= \Upsilon^{-1} \cdot \gamma_{(pm)} \cdot e_p(\tilde\pi,\mathbf{1},j) \cdot e_\infty(\pi,\mathbf{1},j)\cdot \frac{L^{(p)}\big(\pi, j+\tfrac{1}{2}\big)}{\Omega_\pi^{\pm}},
		\end{align*}
		for all $j \in \mathrm{Crit}(\lambda_\pi)$. Here $(-1)^j = \pm 1$ and all other notation is as in Theorem \ref{thm:non-ordinary}.
		
		In particular, $\cL_p(\tilde\pi)$ satisfies the interpolation of Theorem \ref{thm:non-ordinary} at $\chi = \mathbf{1}$.
\end{proposition}

\begin{proof}
Let $\phi_{\tilde\pi^Q}^\pm$ be the class defined in \cite[\S6.6]{BDW20}. Then, taking $\beta=1$, we have
\begin{align*}
	\cL_p(\tilde\pi, z^j) &= (\alpha_{p,n}^\circ)^{-1} \cdot \cE_{Q,1}^{j,\eta_0}(\phi_{\tilde\pi^Q}^\pm)\\
	&= \Upsilon_Q \cdot \frac{p^{n^2} }{\alpha_{p,n}} \cdot e_\infty(\pi,\mathbf{1},j) \cdot \frac{L^{(p)}(\pi,j+\tfrac{1}{2})}{\Omega_\pi^\pm} \cdot \zeta_p\Big(j+\tfrac{1}{2}, (u^{-1}t_{p,n}^\beta)\cdot W_p, \mathbf{1}\Big).
\end{align*}
Here the first equality is shown in \cite[Thm.\ A]{BDW20}. The second is Theorem \ref{thm:critical value}, noting that $\delta_B(t_Q^{-1}) = p^{n^2}$ and $\lambda(t_Q)/\alpha_{p,n}^\circ = 1/\alpha_{p,n}$. This local zeta integral was computed in Proposition \ref{p:local-zeta}. We find that this exactly agrees with the claimed formula (as by definition, $\Upsilon$ tracks the difference between the volume factor $\gamma_{(pm)}$ at Iwahori level and $\gamma$, from \cite{DJR18}, at parahoric level).

The final statement follows immediately from the first part and Proposition \ref{cor:BDW}.
\end{proof}

\begin{remark}
	Exactly the same proof shows more generally that the $p$-adic $L$-functions of \cite{DJR18, BDW20}, for $\GL_{2n}$ over a general totally real field and at arbitrary tame level, satisfy the interpolation formula at unramified characters predicted by Coates--Perrin-Riou and Panchishkin.
\end{remark}

\newpage
\section*{Glossary of key notation/terminology}
\footnotesize 
\vspace{-10pt}
\addcontentsline{toc}{section}{Glossary of notation}

\begin{multicols}{2}
\noindent $(-)^\vee$ \dotfill dual \S\ref{sec:notation}\\
$\langle -,-\rangle_G, \langle -,-\rangle_{\cG}$ \dotfill \S\ref{sec:structure gspin}\\
$\cA_\Omega$ \dotfill Defn.\ \ref{def:A_Omega}\\
$\alpha$ \dotfill system of eigenvalues \S\ref{sec:hecke p}\\
$\alpha_p^\circ$ \dotfill $U_p^\circ$-eigenvalue\\
$\alpha^{\cG}$ \dotfill system of e'values for $\cG$\\
$\alpha_0, \alpha_1,...,\alpha_n$ \dotfill basis of weights \eqref{eq:alpha_i}\\
$\alpha_{p,r}$ \dotfill $\alpha(U_{p,r})$\\
$B \subset G$ \dotfill upper-triangular Borel\\
$\beta$ \dotfill integer $\geq 1$\\
$\overline{B} \subset G$ \dotfill opposite Borel\\
$\sC$ \dotfill Shalika family, Thm.\ \ref{thm:shalika family}\\
$\mathrm{Crit}(\lambda)$ \dotfill \S\ref{sec:algebraic weights}\\
$\cl(I)$ \dotfill \eqref{eq:ray class group}\\
$\cD_\Omega$ \dotfill Defn.\ \ref{def:D_Omega}\\
$\sD_\Omega$ \dotfill local system for $\cD_\Omega$\\
$\cD_\Omega^\beta$ \dotfill Defn.\ \ref{def:distribution support}\\
$\Delta_P$ \dotfill semigroup, \S\ref{sec:abstract evaluations}\\
$\Delta_{\UPS}$ \dotfill function, \eqref{eq:delta ups}\\
$\delta_B$ \dotfill \S\ref{sec:notation}\\
$\mathrm{Ev}_{M,-}^{P,-}$ \dotfill evaluation maps \S\ref{sec:abstract evaluations}\\
$\mathrm{Ev}^\Omega_{B,-}$ \dotfill distribution-valued evaluation maps \S\ref{sec:distribution evaluations}\\
$\cE_{P,\chi}^{j,\eta_0}$ \dotfill classical evaluation map \S\ref{sec:classical evaluations}\\
$\sE_{\Omega,\leq h}, \sE_{\Omega,\leq h }^\pm$ \dotfill local pieces of eigenvariety, Defn.\ \ref{def:local piece S}\\
$e_i, e_i^*$ \dotfill root data \S\ref{sec:structure gspin}\\
$e_p(\tilde\pi, \chi,j)$ \dotfill Coates--Perrin-Riou factor, Thm.\ \ref{thm:non-ordinary}\\
$\eta, \eta_0$ \dotfill Shalika characters \S\ref{sec:shalika models}\\
$F_\delta, F_\delta^\sigma$ \dotfill Defn.\ \ref{def:F_w}\\
$f_i, f_i^*$ \dotfill root data \S\ref{sec:structure gspin}\\
$f^\sigma \in \Ind_B^G \UPS^\sigma$ \dotfill \S\ref{ss:intertwining} \\
$f_w^\sigma \in \Ind_B^G \UPS^\sigma$ \dotfill Defn.\ \ref{def:F_w}\\
$G$ \dotfill $\GL_{2n}$\\
$\cG$ \dotfill $\mathrm{GSpin}_{2n+1}$\\
$\mathrm{Gal}_p$ \dotfill \S\ref{sec:notation}\\
$\gamma_{(pm)}$ \dotfill constant \S\ref{sec:running assumptions}\\
$H$ \dotfill $\GL_n \times \GL_n$\\
$\hc{\bullet}(S_K,-)_\pi$ \dotfill \S\ref{sec:localisation}\\
$\cH'$ \dotfill spherical Hecke algebra \S\ref{sec:hecke algebra}\\
$\cH$ \dotfill Hecke algebra \S\ref{sec:hecke algebra}\\
$\cH^\circ, \cH_p^\circ$ \dotfill Defn.\ \ref{def:local piece S}\\
$\cH_p^{\cG}$ \dotfill Hecke algebra \S\ref{sec:spin via gspin}\\
$\iota : H \to G$ \dotfill $(h_1,h_2) \mapsto \mathrm{diag}(h_1,h_2)$\\
$\iota_\beta^P$ \dotfill \eqref{eq:iota beta}\\
$\mathrm{Iw}_G \subset \GL_{2n}(\Qp)$ \dotfill Iwahori subgroup\\
$\mathrm{Iw}_G^\beta$ \dotfill \eqref{eq:Iw_H^1}\\
$\mathrm{Iw}_{G_n} \subset \GL_n(\Qp)$ \dotfill Iwahori subgroup\\
$i_p$ \dotfill fixed isomorphism $\C \isorightarrow \overline{\Q}_p$\\
$J_P$ \dotfill parahoric \S\ref{sec:abstract evaluations}\\
$\jmath, \jmath^\vee$ \dotfill \S\ref{sec:j}\\
$K \subset G(\A_f)$ \dotfill level group\\
$\kappa_{\lambda,j}$ \dotfill \eqref{eq:kappa lambda j}\\
$\kappa_\Omega$ \dotfill \eqref{eq:kappa omega}\\
$L^{(p)}$ \dotfill $L$-function without factor at $p$\\
$L_\beta^P, L_p^{P,\beta}, L^p$ \dotfill Defn.\ \ref{def:auto cycle}\\
$\cL_p(\tilde\pi), \cL_p^\pm(\tilde\pi)$ \dotfill $p$-adic $L$-functions, Defn.\ \ref{def:non-critical slope Lp}\\
$\cL_p^{\sC}$ \dotfill $p$-adic $L$-function in family, Defn. \ref{def:Lp families}\\
$\lambda = (\lambda_1,...,\lambda_{2n})$ \dotfill weight\\
$M_{w_n}, M_\rho$ \dotfill intertwinings \eqref{eq:M_wn}\\
$m$ \dotfill auxiliary integer, Defn.\ \ref{def:auto cycle}\\
$\m_\pi, \m_{\tilde\pi}$ \dotfill maximal ideals \S\ref{sec:hecke algebra}\\
$N^\beta(\Zp)$ \dotfill \S\ref{sec:support conditions}\\
non-critical slope \dotfill Defn.\ \ref{def:non-critical slope}\\
$\Omega$ \dotfill affinoid in weight space \\
$\Omega_\pi^\pm \in \C^\times$ \dotfill complex period\\
$P \subset G$ \dotfill standard parabolic\\
$\mathrm{pr}_\beta$ \dotfill \eqref{eq:pr_beta}\\
$\pi$ \dotfill RASCAR of $G(\A)$, Conditions \ref{cond:running assumptions}\\
$\phi_{\tilde\pi}^\pm$ \dotfill cohomology class \eqref{eq:phi}\\
$\varphi_p \in \tilde\pi_p$ \dotfill \S\ref{sec:hecke p}\\
$\psi_{\tilde\pi}$ \dotfill \S\ref{sec:hecke p}\\
$\sigma \in \cW_G$ \dotfill Weyl element\\
$\tilde\pi$ \dotfill $p$-refinement \S\ref{sec:hecke p}\\
$Q \subset G$ \dotfill parabolic with Levi $H$\\
$Q'(\pi,\chi,j)$ \dotfill factor at $p$, Cor.\ \ref{cor:L-value}\\
$\rho_G, \rho_{\cG}$\dotfill half sum positive roots \eqref{eq:rho G}\\
RASCAR \dotfill p.\pageref{RASCAR}\\
$r_\lambda$ \dotfill \eqref{eqn:specialisation}\\
$\cS_{\psi}^\eta$ \dotfill Shalika intertwining \S\ref{sec:shalika models}\\
$S_K$ \dotfill locally symm.\ space \S\ref{sec:automorphic cohomology}\\
Shalika refinement \dotfill Defn.\ \ref{def:shalika refinement}\\
spin refinement \dotfill Defn.\ \ref{def:spin-refinement}\\
$T \subset G$ \dotfill diagonal torus\\
$\bT_{\Omega,\leq h}, \bT_{\Omega,\leq h}$ \dotfill Hecke algebras, Defn.\ \ref{def:local piece S}\\
$t_{p,r}$ \dotfill \eqref{eq:t_p}\\
$t_P$ \dotfill Defn.\ \ref{def:parabolic}\\
$\tau$ \dotfill $\smallmatrd{1}{}{}{w_n} \in \GL_{2n}(\Qp)$\\
$\tau(\chi)$ \dotfill Gauss sum \ref{eq:gauss sum}\\
$\Theta^\pm$ \dotfill \S\ref{sec:cohomology for RACARs}\\
$\UPS$ \dotfill $\pi_p = \mathrm{Ind}_B^G \UPS$\\
$U_{p,r}$ \dotfill Hecke operator \S\ref{sec:hecke p 1}\\
$U_{p,r}^\circ$ \dotfill \S\ref{sec:hecke p}\\
$u$ \dotfill $\smallmatrd{1}{w_n}{0}{1}$ (eqn.\ \eqref{eq:u})\\
$\Upsilon_P$ \dotfill constant, Thm.\ \ref{thm:critical value}\\
$\Upsilon''$ \dotfill Prop.\ \ref{prop:spin refinements are shalika}\\
$V_\lambda$ \dotfill algebraic representation \S\ref{sec:algebraic weights}\\
$V_{(j,-j-\sw)}^H$ \dotfill \S\ref{sec:classical evaluations}\\
$\nu_{p,r}$\dotfill cocharacter \S\ref{sec:spin via gspin}\\
$\nu_\rho = \smallmatrd{1}{}{}{\rho} \in \cW_G$ \dotfill $\rho \in \cW_{n}$\\
$\cV, \sV$ \dotfill local systems \S\ref{sec:local systems}\\
$v_{(i)}, v_{(n),j}$ \dotfill Notation \ref{not:v_i}\\
$v_{\lambda,j}$ \dotfill Prop.\ \ref{prop:product for v lambda j}\\
$v_\Omega$ \dotfill Definition \ref{def:v_Omega}\\
$W$, $W_\ell, W_f$ \dotfill Shalika vectors \S\ref{sec:shalika models}, \eqref{eq:W_f}\\
$W_0, W_\delta \in \cS_{\psi}^\eta(\pi_p)$ \dotfill Prop.\ \ref{prop:spin eigenvector}\\ 
$W_\ell^\circ$ \dotfill spherical vector \S\ref{sec:shalika models}\\
$\sW^G, \sW^G_0$ \dotfill weight spaces, Defn.\ \ref{def:weight space}\\
$\cW_G$ \dotfill Weyl group of $G$ ($\mathrm{S}_{2n}$) \\
$\cW_{\cG}$ \dotfill Weyl group of $\cG$\\
$\cW_G^0$ \dotfill \S\ref{sec:structure gspin}\\
$\cW_n$ \dotfill Weyl group of $\GL_n$ ($\mathrm{S}_n$)\\
$\sw$ \dotfill purity weight \S\ref{sec:algebraic weights}\\
$\sw_\Omega$ \dotfill purity weight in family \eqref{eq:sw_omega}\\
$w_{2n} \in \cW_G$ \dotfill longest Weyl element\\
$w_n \in \cW_n$ \dotfill longest Weyl element\\
$w_{\chi}$ \dotfill character \eqref{eq:v circ}\\
$X_\beta^P$ \dotfill Defn.\ \ref{def:auto cycle}\\
$X_\beta^P[\delta]$ \dotfill \S\ref{sec:auto cycles subsec}\\
$\cX, \cX^\vee$ \dotfill Prop.\ \ref{prop:spin root system}\\
$x_{\tilde\pi}$ \dotfill point of eigenvariety, Thm.\ \ref{thm:shalika family} \\
$\Xi$ \dotfill averaging map \eqref{rem:classical diagram}\\
$\chi$ (in \S2-4, \S8, \S12-14) \dotfill finite order Hecke character\\
$\chi$ (in \S5-7) \dotfill local character $\Q_p^\times \to C^\times$\\
$\chi$ (in \S9) \dotfill local character $F^\times \to \C^\times$\\
$\chi$ (in \S10) \dotfill character $T(\Zp) \to R^\times$\\
$\chi_{\cyc}$ \dotfill cyclotomic character\\
$y$ \dotfill point of eigenvariety\\
$z$ \dotfill $\mathrm{diag}(p^{n-1},p^{n-2},\dots, p,1)$\\
$\zeta(s,W,\chi)$ \dotfill Friedberg--Jacquet integral \S\ref{sec:friedberg-jacquet}\\
$\zeta_j(W_\infty^\pm)$ \dotfill zeta integral at infinity, Thm.\ \ref{thm:critical value}

\end{multicols}

	\footnotesize
	\renewcommand{\refname}{\normalsize References} 
	\bibliography{master_references}{}

\begin{thebibliography}{LPSZ21}

\bibitem[AG94]{AG94}
Avner Ash and David Ginzburg.
\newblock {$p$}-adic {$L$}-functions for {${\rm GL}(2n)$}.
\newblock {\em Invent. Math.}, 116(1-3):27--73, 1994.

\bibitem[AS06]{AS06}
Mahdi Asgari and Freydoon Shahidi.
\newblock Generic transfer for general spin groups.
\newblock {\em Duke Math. J.}, 132(1):137--190, 2006.

\bibitem[AS14]{AS14}
Mahdi Asgari and Freydoon Shahidi.
\newblock Image of functoriality for general spin groups.
\newblock {\em Manuscripta Math.}, 144(3-4):609--638, 2014.

\bibitem[Asg02]{Asg02}
Mahdi Asgari.
\newblock Local {$L$}-functions for split spinor groups.
\newblock {\em Canad. J. Math.}, 54(4):673--693, 2002.

\bibitem[Ash80]{Ash80}
Avner Ash.
\newblock Non-square-integrable cohomology of arithmetic groups.
\newblock {\em Duke Math. J.}, 47(2):435--449, 1980.

\bibitem[BDJ22]{BDJ17}
Daniel Barrera, Mladen Dimitrov, and Andrei Jorza.
\newblock {$p$}-adic {$L$}-functions of {H}ilbert cusp forms and the trivial
  zero conjecture.
\newblock {\em J. Eur. Math. Soc. (JEMS)}, 24(10):3439--3503, 2022.

\bibitem[BDW]{BDW20}
Daniel {Barrera Salazar}, Mladen Dimitrov, and Chris Williams.
\newblock On $p$-adic {$L$}-functions for $\mathrm{{GL}}(2n)$ in finite slope
  {S}halika families.
\newblock Preprint: \url{https://arxiv.org/abs/2103.10907}.

\bibitem[BFG92]{BFG92}
Daniel Bump, Solomon Friedberg, and David Ginzburg.
\newblock Whittaker-orthogonal models, functoriality, and the
  {R}ankin-{S}elberg method.
\newblock {\em Invent. Math.}, 109(1):55--96, 1992.

\bibitem[BGW]{classical-locus}
Daniel {Barrera Salazar}, Andrew Graham, and Chris Williams.
\newblock On $p$-refined {F}riedberg--{J}acquet integrals and the classical
  symplectic locus in the {$\mathrm{GL}_{2n}$}-eigenvariety.
\newblock Preprint: \url{https://arxiv.org/abs/2308.02649}.

\bibitem[BW19]{BW_CJM}
Daniel {Barrera Salazar} and Chris Williams.
\newblock {$p$}-adic {$L$}-functions for {${\rm GL}_2$}.
\newblock {\em Canad. J. Math.}, 71(5):1019--1059, 2019.

\bibitem[BW21a]{BW18}
Daniel {Barrera Salazar} and Chris Williams.
\newblock Families of {B}ianchi modular symbols: critical base-change
  {$p$}-adic {$L$}-functions and {$p$}-adic {A}rtin formalism.
\newblock {\em Selecta Math. (N.S.)}, 27(82), 2021.
\newblock Appendix by Carl Wang-Erickson.

\bibitem[BW21b]{BW-Iwasawa}
Daniel {Barrera Salazar} and Chris Williams.
\newblock {Overconvergent cohomology, {$p$}-adic {$L$}-functions and families
  for {$\rm GL(2)$} over {CM} fields}.
\newblock {\em J. Th\'{e}or. Nombres Bordeaux}, 33(3, part 1):659--701, 2021.

\bibitem[BW21c]{BW20}
Daniel {Barrera Salazar} and Chris Williams.
\newblock Parabolic eigenvarieties via overconvergent cohomology.
\newblock {\em Math. Z.}, 299(1-2):961--995, 2021.

\bibitem[Cas80]{Cas80}
W.~Casselman.
\newblock The unramified principal series of {$\mathfrak{p}$}-adic groups. {I}.
  {T}he spherical function.
\newblock {\em Compositio Math.}, 40(3):387--406, 1980.

\bibitem[Che04]{Che04}
Ga\"{e}tan Chenevier.
\newblock Familles {$p$}-adiques de formes automorphes pour {${\rm GL}_n$}.
\newblock {\em J. Reine Angew. Math.}, 570:143--217, 2004.

\bibitem[Clo90]{Clo90}
Laurent Clozel.
\newblock Motifs et formes automorphes: applications du principe de
  fonctorialit\'{e}.
\newblock In {\em Automorphic forms, {S}himura varieties, and {$L$}-functions,
  {V}ol. {I} ({A}nn {A}rbor, {MI}, 1988)}, volume~10 of {\em Perspect. Math.},
  pages 77--159. Academic Press, Boston, MA, 1990.

\bibitem[Coa89]{coates89}
John Coates.
\newblock On {$p$}-adic {$L$}-functions attached to motives over {${\bf Q}$}.
  {II}.
\newblock {\em Bol. Soc. Brasil. Mat. (N.S.)}, 20(1):101--112, 1989.

\bibitem[CS20]{CS20}
Fulin Chen and Binyong Sun.
\newblock Uniqueness of twisted linear periods and twisted {S}halika periods.
\newblock {\em Sci. China Math.}, 63(1):1--22, 2020.

\bibitem[DJ]{DJ-parahoric}
Mladen Dimitrov and Andrei Jorza.
\newblock Parahoric level $p$-adic {$L$}-functions for automorphic
  representations of {GL}(2n) with {S}halika models.
\newblock In preparation.

\bibitem[DJR20]{DJR18}
Mladen Dimitrov, Fabian Januszewski, and A.~Raghuram.
\newblock ${L}$-functions of {$\mathrm{GL}(2n)$}: {$p$}-adic properties and
  nonvanishing of twists.
\newblock {\em Compositio Math.}, 156(12):2437--2468, 2020.

\bibitem[FJ93]{FJ93}
Solomon Friedberg and Herv\'{e} Jacquet.
\newblock Linear periods.
\newblock {\em J. Reine Angew. Math.}, 443:91--139, 1993.

\bibitem[Gen]{Geng}
Zhibin Geng.
\newblock On the existence of twisted {S}halika periods: the archimedean case.
\newblock Preprint: \url{https://arxiv.org/abs/2501.11917}.

\bibitem[GH24]{GetzHahn}
Jayce~R. Getz and Heekyoung Hahn.
\newblock {\em An introduction to automorphic representations---with a view
  toward trace formulae}, volume 300 of {\em Graduate Texts in Mathematics}.
\newblock Springer, Cham, [2024] \copyright 2024.

\bibitem[GR14]{GR2}
Harald Grobner and A.~Raghuram.
\newblock On the arithmetic of {S}halika models and the critical values of
  {$L$}-functions for {${\rm GL}_{2n}$}.
\newblock {\em Amer. J. Math.}, 136(3):675--728, 2014.
\newblock With an appendix by Wee Teck Gan.

\bibitem[Han17]{Han17}
David Hansen.
\newblock Universal eigenvarieties, trianguline {G}alois representations and
  $p$-adic {L}anglands functoriality.
\newblock {\em J. Reine. Angew. Math.}, 730:1--64, 2017.

\bibitem[HS16]{HS16}
Joseph Hundley and Eitan Sayag.
\newblock Descent construction for {GS}pin groups.
\newblock {\em Mem. Amer. Math. Soc.}, 243(1148):v+124, 2016.

\bibitem[JN19]{JoNew}
Christian Johansson and James Newton.
\newblock Irreducible components of extended eigenvarieties and interpolating
  {L}anglands functoriality.
\newblock {\em Math. Res. Lett.}, 26(1):159--201, 2019.

\bibitem[JPSS81]{JPSS}
Herv\'{e} Jacquet, Ilja Piatetski-Shapiro, and Joseph Shalika.
\newblock Conducteur des repr\'{e}sentations g\'{e}n\'{e}riques du groupe
  lin\'{e}aire.
\newblock {\em C. R. Acad. Sci. Paris S\'{e}r. I Math.}, 292(13):611--616,
  1981.

\bibitem[JS90]{JS90}
Herv\'{e} Jacquet and Joseph Shalika.
\newblock Exterior square {$L$}-functions.
\newblock In {\em Automorphic forms, {S}himura varieties, and {$L$}-functions,
  {V}ol. {II} ({A}nn {A}rbor, {MI}, 1988)}, volume~11 of {\em Perspect. Math.},
  pages 143--226. Academic Press, Boston, MA, 1990.

\bibitem[JST]{JST}
Dihua Jiang, Binyong Sun, and Fangyang Tian.
\newblock Period relations for standard {$L$}-functions of symplectic type.
\newblock Preprint: \url{https://arxiv.org/abs/1909.03476}.

\bibitem[Kat04]{Kat04}
Kazuya Kato.
\newblock $p$-adic {H}odge theory and values of zeta functions of modular
  forms.
\newblock {\em Ast\'{e}risque}, 295:117--290, 2004.

\bibitem[Loe21]{loeffler-parabolics}
David Loeffler.
\newblock Spherical varieties and norm relations in {I}wasawa theory.
\newblock {\em J. Th\'{e}or. Nombres Bordeaux}, 33(3, part 2):1021--1043, 2021.
\newblock Iwasawa 2019 special issue.

\bibitem[LPSZ21]{LPSZ19}
David Loeffler, Vincent Pilloni, Christopher Skinner, and Sarah~Livia Zerbes.
\newblock Higher {H}ida theory and {$p$}-adic {$L$}-functions for
  {$\mathrm{GSp}_4$}.
\newblock {\em Duke Math. J.}, 170(18):4033--4121, 2021.

\bibitem[LSZ22]{LSZ17}
David Loeffler, Christopher Skinner, and Sarah~Livia Zerbes.
\newblock Euler systems for {${\mathrm{GSp}}(4)$}.
\newblock {\em J. Eur. Math. Soc. (JEMS)}, 24(2):669--733, 2022.

\bibitem[LZa]{LZ-HilbertQuadratic}
David Loeffler and Sarah~Livia Zerbes.
\newblock Iwasawa theory for quadratic {H}ilbert modular forms.
\newblock Preprint: \url{https://arxiv.org/abs/2006.14491}.

\bibitem[LZb]{LZ20}
David Loeffler and Sarah~Livia Zerbes.
\newblock On the {B}loch--{K}ato conjecture for $\mathrm{{GSp}}(4)$.
\newblock Preprint: \url{https://arxiv.org/abs/2003.05960}.

\bibitem[LZ23]{LZ-cube}
David Loeffler and Sarah~Livia Zerbes.
\newblock On the {B}loch-{K}ato conjecture for the symmetric cube.
\newblock {\em J. Eur. Math. Soc. (JEMS)}, 25(8):3359--3363, 2023.

\bibitem[Mat13]{Mat13}
Nadir Matringe.
\newblock Essential {W}hittaker functions for {$GL(n)$}.
\newblock {\em Doc. Math.}, 18:1191--1214, 2013.

\bibitem[Ngu22]{DucNam}
Duc~Nam Nguyen.
\newblock {\em On functional equations of $p$-adic {$L$}-functions for
  {$\mathrm{GL}$}(2)}.
\newblock PhD thesis, Universit\'e de Lille, 2022.

\bibitem[Nie09]{Nia09}
Chufeng Nien.
\newblock Uniqueness of {S}halika models.
\newblock {\em Canad. J. Math.}, 61(6):1325--1340, 2009.

\bibitem[OST23]{OST19}
Masao Oi, Ryotaro Sakamoto, and Hiroyoshi Tamori.
\newblock Iwahori-{H}ecke algebra and unramified local {$L$}-functions.
\newblock {\em Math. Z.}, 303(3):Paper No. 59, 42, 2023.

\bibitem[Pan94]{Pan94}
Alexei~A. Panchishkin.
\newblock Motives over totally real fields and {$p$}-adic {$L$}-functions.
\newblock {\em Ann. Inst. Fourier (Grenoble)}, 44(4):989--1023, 1994.

\bibitem[Roc]{rockwood-spherical}
Rob Rockwood.
\newblock Spherical varieties and non-ordinary families of cohomology classes.
\newblock {\em J.\ Number Theory}.
\newblock To appear: \url{https://arxiv.org/abs/2204.07116}.

\bibitem[Roc23]{Roc20}
Rob Rockwood.
\newblock Plus/minus {$p$}-adic {$L$}-functions for {${\rm GL}_{2n}$}.
\newblock {\em Ann. Math. Qu\'{e}.}, 47(1):177--193, 2023.

\bibitem[Roh89]{Roh89}
David~E. Rohrlich.
\newblock Nonvanishing of {$L$}-functions for {${\rm GL}(2)$}.
\newblock {\em Invent. Math.}, 97(2):381--403, 1989.

\bibitem[SU14]{SU14}
Christopher Skinner and Eric Urban.
\newblock The {I}wasawa {M}ain {C}onjectures for {GL}$(2)$.
\newblock {\em Invent. Math.}, 195 (1):1--277, 2014.

\bibitem[Sun19]{Sun19}
Binyong Sun.
\newblock Cohomologically induced distinguished representations and
  cohomological test vectors.
\newblock {\em Duke Math. J.}, 168(1):85--126, 2019.

\bibitem[Urb11]{Urb11}
Eric Urban.
\newblock Eigenvarieties for reductive groups.
\newblock {\em Ann. of Math. (2)}, 174(3):1685--1784, 2011.

\bibitem[Zel80]{Zel80}
A.~V. Zelevinsky.
\newblock Induced representations of reductive {${p}$}-adic groups. {II}. {O}n
  irreducible representations of {${\rm GL}(n)$}.
\newblock {\em Ann. Sci. \'{E}cole Norm. Sup. (4)}, 13(2):165--210, 1980.

\end{thebibliography}
	\bibliographystyle{alpha}

\end{document}